\documentclass[runningheads]{svjour2}
\smartqed  
\usepackage{graphicx}
\usepackage{amssymb}
\usepackage{amsmath}
\usepackage{bbm}

\usepackage{amsthm}
\usepackage{color}
\usepackage[usenames,dvipsnames]{xcolor}
\usepackage[ruled]{algorithm2e}

\usepackage{algcompatible}
\usepackage{enumitem}
\usepackage{comment}
\usepackage{url}

\makeatletter
\newcommand{\algorithmstyle}[1]{\renewcommand{\algocf@style}{#1}}
\makeatother
\newtheoremstyle{cited}%
  {3pt}
  {3pt}
  {\itshape}
  {}
  {\bfseries}
  {.}
  {.5em}
  {\thmname{#1} \thmnumber{#2} \thmnote{\normalfont#3}}
  
\newtheorem{assumption}[theorem]{Assumption}
\newtheorem{thm}[theorem]{Theorem}
\newtheorem{prop}[theorem]{Proposition}
\newtheorem{cor}[theorem]{Corollary}
\newtheorem{lem}[theorem]{Lemma}
\newtheorem{defi}[theorem]{Definition}
\newtheorem{rem}[theorem]{Remark}
\theoremstyle{cited}
\newtheorem{citedthm}[theorem]{Theorem}
\newtheorem{citedprop}[theorem]{Proposition}
\newtheorem{citedcor}[theorem]{Corollary}

\newcommand{\R}{\mathbb{R}} 
\newcommand{\SP}{\mathbb{S}^2} 
\newcommand{\sspace}{H^2\left({\SP}\right)}

\def\makeheadbox{}

\begin{document}

\title{Critical points and bifurcations of the three-dimensional Onsager model for liquid crystals
}

\titlerunning{Critical points and bifurcations of the three-dimensional Onsager model}        

\author{Michaela A. C. Vollmer}


\institute{University of Oxford\\ 
			Andrew Wiles Building\\
			Woodstock Road\\
			Oxford OX2 6GG\\
              \email{michaela.vollmer@maths.ox.ac.uk}           
}

\date{Received: date / Accepted: date}

\maketitle

\begin{abstract}
We study the bifurcation diagram of the Onsager free-energy functional for liquid crystals with orientation parameter on the sphere in three dimensions. In particular, we concentrate on a general class of two-body interaction potentials including the Onsager kernel. The problem is reformulated as a non-linear eigenvalue problem for the kernel operator, and a general method to find the corresponding eigenvalues and eigenfunctions is presented. Our main tools for this analysis are spherical harmonics and a special algorithm for computing expansions of products of spherical harmonics in terms of spherical harmonics. We find an explicit expression for the set of all bifurcation points. Using a Lyapunov-Schmidt reduction, we derive a bifurcation equation depending on five state variables. The dimension of this state space is further reduced to two dimensions by using the rotational symmetry of the problem and the invariant theory of groups. On the basis of these results, we show that the first bifurcation occurring in the case of the Onsager interaction potential is a transcritical bifurcation and that the corresponding solution is uniaxial. 
In addition, we prove some global properties of the bifurcation diagram such as the fact that the trivial solution is the unique local minimiser for high temperatures, that it is not a local minimiser if the temperature is low, the boundedness of all equilibria of the functional and that the bifurcation branches are either unbounded or that they meet another bifurcation branch.


\keywords{Isotropic-nematic phase transition \and Onsager
intermolecular potential \and Maier-Saupe intermolecular potential \and Spherical harmonics \and Rotational symmetry \and Bifurcation analysis \and Axial Symmetry}
\end{abstract}

\section{Introduction}
\label{intro}
The isotropic-to-nematic phase transition of rod-like molecules is one of the most studied phenomena in the theory of liquid crystals. It is characterised by an onset of orientational order due to an excluded volume effect when either the concentration of molecules is increased or the temperature of the system is decreased. The first model describing such a phase transition was introduced by Onsager in 1949 \cite{Onsager1949}. 
Let $\rho:\SP\rightarrow \R$ be a probability density function characterising the orientation of the molecules, that is
\begin{align}\label{eqt: assumptions probability densities}
\rho(p)\geq 0 \quad \text{ for all } p\in\SP\qquad \text{ and }\qquad\int_{\SP}\rho(p)\;dp=1.
\end{align}
Using a second virial approximation, Onsager derived the free-energy functional
\begin{align}\label{eqt:free energy}
\mathcal{F}(\rho):= \int_{\SP}\left(k_B\tau\rho(p)\ln(\rho(p))+\frac{1}{2}U(\rho)(p)\rho(p)\right)\;dp
\end{align}
where $U:L^1(\SP)\rightarrow L^\infty(\SP)$ denotes the interaction operator given by 
\begin{align}\label{eqt:interaction operator}
U(\rho)(p):=\int_{\SP\times \SP}K(p,q)\rho(q)\;dq.
\end{align}
The interaction kernel $K(\cdot,\cdot):\SP\times \SP\rightarrow \R$  describes the excluded volume effect given by the interaction of one test rod with all other polymer rods in the system. In particular Onsager was interested in the kernel
\begin{align}\label{eqt:Onsager kernel}
K_O(p,q)=|p\times q|=\sqrt{1-(p\cdot q)^2}
\end{align}  
which has been named after him.
However, due to the non-analytic nature of this particular interaction potential, the equilibria of the energy functional in (\ref{eqt:free energy}) cannot be computed explicitly. 
Subsequently, many theories and variations of Onsager's model have been formulated. One of the most popular is the so called Maier-Saupe potential \cite{Maier-Saupe1958}
\begin{align}\label{eqt:Maier Saupe potential}
K_\text{MS}(p,q)=\frac{1}{3}-(p\cdot q)^2.
\end{align}
It is based on an approximation of the mean field approach used by Onsager and assumes that the interactions between the rods are represented by the second moment of the probability density function for the orientation of the rods. A third well-known intermolecular potential is the asymmetric interaction potential which is also called the dipolar potential
\begin{align}
\label{eqt: dipolar potential}
K_d(p,q)=-p\cdot q.
\end{align}

The equilibrium states of the free-energy functional above equipped with the Maier-Saupe kernel in \eqref{eqt:Maier Saupe potential} have been studied extensively in the past. Among the first researchers who became interested in this problem was Freiser \cite{Freiser1970}, who extended the Maier-Saupe interaction potential to the case of asymmetric molecules. He proved both the existence of a first-order phase transition to a nematic state and a second-order phase transition to a biaxial state. Constantin, Kevrekidis and Titi \cite{Constantin-Titi22004} established the equivalence of the equilibrium states of the Onsager functional with the Maier-Saupe potential in three dimensions to the solutions of a transcendental matrix equation, and thus derived the high concentration asymptotics in terms of the eigenvalues of this particular matrix. A complete classification of the equilibrium states in three dimensions has been provided by Fatkullin and Slastikov \cite{Fatkullin-Slastikov2005} and independently by Liu, Zhang and Zhang \cite{ZhangZhang2005}. Using the properties of spherical harmonics, both groups derived explicit formulae for all critical points and proved their axial symmetry. Similarly, Fatkullin and Slastikov also obtained results for the dipolar interaction potential in (\ref{eqt: dipolar potential}).

One of the first approaches to the original problem involving the Onsager kernel in (\ref{eqt:Onsager kernel}) has been undertaken by Kayser and Ravech\'e \cite{Kayser-Raveche1978}. By reformulating the problem as an eigenvalue problem, they derive an iterative scheme that allows them to compute all axially symmetric equilibria of the functional.
Using the Fourier coefficients of the Onsager potential in two dimensions, Wang and Zhou \cite{Wang-Zhou2007,Wang-Zhou2010} showed that there exist in fact infinitely many bifurcation branches with different symmetries. Revisiting  the same question, Chen, Li and Wang \cite{Chen-Wang2010} also proved that all solutions are axially symmetric in $2$D. For a class of interaction potentials involving only a finite number of Fourier modes, Lucia and Vukadinovic \cite{LuciaVukadinovic2010} verified the existence of continuous branches in two dimensions and they characterised the structure of the bifurcation diagram in terms of the size of the spectral gaps of the interaction operator. A very recent approach to the same problem has also been undertaken by Niksirat and Yu  \cite{Niksirat-Yu2015} who obtained the local bifurcation structure of all equilibrium states in two dimensions and the uniqueness of the trivial solution for high temperatures.

However, the problem of classifying all critical points of the Onsager free-energy functional with the Onsager interaction kernel in three dimensions has not yet been addressed. In this paper, we reformulate the problem as an eigenvalue problem and derive a method that allows us to compute the eigenvalues of general interaction kernels. On the basis of this result, we obtain a complete set of all bifurcation points and we find the local bifurcation structure of the first of these points corresponding to the first phase transition occurring at the highest temperature. In particular, we prove the existence of a transcritical bifurcation and we show that all critical points of this bifurcation branch are axially symmetric.
These results about the local bifurcation structure of the Onsager free-energy functional are summarised in the following theorem.
\begin{theorem}[Local bifurcations of the Onsager  free-energy functional]
Let the Onsager free-energy functional $\mathcal{F}$ in \textup{(\ref{eqt:free energy})} be equipped with the Onsager kernel $K_O(p,q)=\sqrt{1-(p\cdot q)^2}$.
Then a transcritical bifurcation from the trivial solution $\rho=\frac{1}{4\pi}$ occurs locally at the temperature 
$$\tau_\star=\frac{\pi}{32 k_B}$$ and it is uniaxial. Moreover, all other local bifurcations from the trivial solution occur at the temperatures
$$\tau_s=\frac{\Gamma(s/2+\frac{1}{2})\Gamma(s/2-1/2)}{8k_B\Gamma(s/2+1)\Gamma(s/2+2)}$$
where $s\in 2\mathbb{N}$ and the trivial solution $\rho=\frac{1}{4\pi}$ is a local minimiser for all $\tau>\frac{\pi}{32 k_B}$ and it is not a local minimiser for $\tau<\frac{\pi}{32k_B}$.
\label{thm: local bifurcation structure summary}
\end{theorem}
The proof of this theorem consists of two crucial steps. The derivation of an explicit expression of all eigenvalues of the interaction operator $U$ corresponding to the Onsager kernel and the exploitation of the symmetry properties inherent to the problem. In particular, the symmetry can be expressed in terms of a group action acting on the state space of solutions. Because we are only interested in solutions up to rotational symmetry, the state space of solutions can be reduced to the orbit space of this group action. Our approach involving invariant theory for groups is generally applicable to bifurcation problems for functions defined on the sphere and is of interest on its own (for details see Sections \ref{sec:dimension reduction using symmetry} and \ref{sec: solving the bifurcation equation}).


We would like to point out that even though Theorem \ref{thm: local bifurcation structure summary} only refers to the Onsager kernel, our methods apply to a very general class of intermolecular potentials. One critical property of most physical intermolecular potentials is rotational symmetry. 
\begin{defi}\label{def: rotational symmetry}
An interaction potential $K(\cdot,\cdot)$ is rotationally symmetric, if $K(p,q)=K(Rq,Rp)$ for all $p,q\in\SP$ and for all $R\in SO(3)$.
\end{defi}
An important consequence of this property is the following proposition which is based on the basic representation theorem of simultaneous invariants of vectors due to Cauchy.
\begin{citedprop}[{{\cite[page 29]{TruesdellNoll2004}}}]
If the interaction potential is rotationally symmetric,
then it can be written as a function of a scalar $K(p,q)=k(p\cdot q)$ where $k: \mathbb{R}\rightarrow \mathbb{R}$.
\label{prop: K as a scalar function}
\end{citedprop}
\begin{proof} Omitted.
\end{proof}
\begin{rem}
Within this paper, we use the notations $K(p,q)$ and $k(p\cdot q)$ interchangeably to denote the interaction kernel of interest. 
\end{rem}
The following assumption characterises a large class of interaction potentials to which our methods can be applied to.

\begin{assumption}\label{assumption}
The interaction potential $K(\cdot,\cdot):\SP\times \SP\rightarrow\R$ satisfies
\begin{enumerate}[leftmargin=*,labelindent=16pt,label=(\alph*)]
\item $K(\cdot,\cdot)$ is continuous in both variables;
\item $K(\cdot,\cdot)$ is symmetric, that is $K(p,q)=K(q,p)$ for all $p,q\in\SP$;
\item $K(\cdot,\cdot)$ is rotationally symmetric, see Definition \textup{\ref{def: rotational symmetry}};
\item Each spherical harmonic $Y^m_l$ is an eigenfunction of the interaction operator $U$ in \textup{\eqref{eqt:interaction operator}}, that is $UY^m_l=\lambda_l Y^m_l$, with eigenvalue $\lambda_l$ so that $|\lambda_l|\rightarrow 0$ as $l\rightarrow \infty$.
\end{enumerate} 
\end{assumption}
For a brief introduction to spherical harmonics see Appendix \ref{app: notation spherical harmonics}. \\

In this article we restrict our attention to bifurcations of the Euler-Lagrange equation of the Onsager free-energy functional at $\rho=\frac{1}{4\pi}$ equipped with the Onsager kernel. However, our approach is generally applicable as long as sufficient information of the eigenvalues $\lambda_l$ is available. On this basis our methods allow a derivation of an expansion of the bifurcation equation, see Remark \ref{BifurcationGeneralKernel} for more details. Applying a similar dimension reduction, one is faced with a similar recognition problem that needs to be solved on a case to case basis. All in all,  a characterisation of the local bifurcation can be achieved using the methods presented in this research article.



\begin{rem}\label{rem: Wachsmuth}Similar results about the local bifurcation structure of the Onsager free-energy functional with interaction potential satisfying Assumptions \textup{\ref{assumption} (a)-(c)} can also be found in the unpublished thesis of Jakob Wachsmuth \textup{\cite{Wachsmuth2006}}, which was only drawn to the attention of the author after completion of this work. The major differences between his and our work are as follows: 
\begin{itemize}
\item Instead of working with the set of spherical harmonics, Wachsmuth uses an equivalent description of this set of functions, namely homogeneous polynomials restricted to the sphere. This allows him to perform an elegant dimension reduction to a state space of two dimensions. 
\item Wachsmuth does not derive an explicit expression for the set of eigenvalues of the interaction operator $U$. He shows that the Laplace-Beltrami operator and $U$ commute which allows him to conclude that the spherical harmonics are indeed the eigenvectors of $U$ as well. In contrast, we derive a method that allows us to compute the eigenvalues of $U$ equipped with any kernel admitting an appropriate Taylor expansion, see Theorem \textup{\ref{thm:eigenvalues and eigenfunctions of general interaction kernels}}.
\item Wachsmuth imposes the assumption that the absolute values of the eigenvalues form a decreasing sequence. Assumption \textup{\ref{assumption} (d)} is more general and allows the sequence of eigenvalues to fluctuate if it converges to zero.
\item Crucially, we prove local uniaxiality of the solutions occurring at the bifurcation point $\tau^\star$, while Wachsmuth only shows that the solution is infinitesimally uniaxial. 
\end{itemize}
\end{rem}
 Last but not least, we also prove properties of the global bifurcation diagram which are summarised in the following theorem.

\begin{theorem}[Global bifurcations of the Onsager free-energy functional]\label{thm:global bifurcation picture}
Let $K(\cdot,\cdot):\SP\times \SP\rightarrow \R$ be an interaction kernel satisfying Assumption \textup{\ref{assumption} (a)-(c)}.
Then the uniform distribution $\rho_0(p)=\frac{1}{4\pi}$ is a critical point for all temperature values $\tau$. If $\tau$ is such that $$\frac{8\pi  M}{k_B} \exp\left(\frac{16M}{k_B \tau}\right) \leq \tau$$
where $M=\max_{p,q\in \SP} K(p,q)$, then the uniform distribution is in fact the unique solution of the Onsager free-energy functional in three dimensions. Moreover, any non-trivial solution of the Onsager free-energy functional is bounded and all bifurcation branches either meet infinity or they meet another bifurcation branch. 
\end{theorem}

This work is structured as follows. In Section \ref{sec:Bifurcation Points Onsager Functional} we derive a method that allows us to compute all eigenvalues and eigenfunctions of interaction operators of the form given in (\ref{eqt:interaction operator}) which correspond to interaction kernels satisfying Assumption \ref{assumption}. Carrying out a Lyapunov-Schmidt reduction, we reduce the infinite-dimensional problem to a five-dimensional bifurcation equation in Section \ref{sec: Lyapunov Schmidt reduction and bifurcatione quation}. In Section \ref{sec:dimension reduction using symmetry} we reduce the dimension of the problem further by exploiting the rotational symmetry of our setting which results in a simplified two-dimensional bifurcation equation. We solve the corresponding recognition problem in Section \ref{sec: solving the bifurcation equation}. Following these results, we restrict our attention to the case of uniaxial solutions in Section \ref{sec: uniaxial solutions}. In particular, we prove that a transcritical bifurcation also occurs in case of the restricted problem and we therefore deduce that the solutions of the full problem must be uniaxial as well. Section \ref{sec: Qualitative behaviour of the global bifurcation diagram} is devoted to a brief analysis of some global properties of the bifurcation diagram, such as the boundedness of all solutions, the uniqueness of the trivial solution as local minimiser for high temperatures and the continuity of all bifurcation branches. We complete this work by summarising our results in Section \ref{sec:conclusion}.


\section{A complete description of all bifurcation points of the Onsager free-energy functional}\label{sec:Bifurcation Points Onsager Functional}
The novel contribution of this section is the full characterisation of all bifurcation points of the Euler-Lagrange equation of the Onsager free-energy functional. In particular, we use the Taylor expansion of the interaction kernel in three dimensions. 

The Euler-Lagrange equation of the free-energy functional in (\ref{eqt:free energy}) is given by
\begin{align*}
\lambda \ln \rho(p)+\int_{\SP}K(p,q)\rho(q)\;dq=c
\end{align*}
where $\lambda:=k_B\tau$. Its derivation is well known and can for example be found in \cite{Fatkullin-Slastikov2005} or \cite{ZhangZhang2005}. The constant $c$ is obtained through rearranging the equation and imposing the constraint that the orientation distribution function $\rho$ integrates to one, see the conditions in (\ref{eqt: assumptions probability densities}). In particular, $c$ is given by
\begin{align*}
c=-\lambda \ln Z \text{ where } Z=\int_{\SP}\exp\left(-\frac{1}{\lambda}\int_{\SP}K(p,q)\rho(q)\;dq\right)\;dp
\end{align*}
denotes the partition function. Introducing the thermodynamic potential $\phi:\SP\rightarrow\R$
$$\phi(p):=\frac{1}{\lambda}\int_{\SP}K(p,q)\rho(q)\;dq,$$
it follows that $\rho(p)=Z^{-1}\exp(-\phi(p))$ and we can rewrite the Euler-Lagrange equation as
\begin{align*}
\lambda \phi(p)-\frac{1}{Z(\phi)}\int_{\mathbb{S}^2} K(p,q)\exp(-\phi(q))\;dq=0
\end{align*}
for $$Z(\phi)=\int_{\SP}\exp(-\phi(q))\;dq.$$
The addition or subtraction of a constant to the free-energy functional in \eqref{eqt:free energy} does not change its minimisers. Therefore we assume without loss of generality that
$$\int_{\SP}k(p\cdot q)\;dq=0.$$
Each minimiser of the Onsager free-energy functional must be a solution of the Euler-Lagrange equation. In particular, we prove in Proposition \ref{prop:regularityCriticalPoint} that all critical points are elements of the space $C^\infty(\SP)$. The ground state of the Euler-Lagrange equation is given by $\phi_0(p)=0$ which corresponds to the uniform probability distribution $\rho_0(p)=\frac{1}{4\pi}$ and solves the Euler-Lagrange equation for all temperatures $\lambda=k_B \tau$. It represents the isotropic phase and we are interested in analysing whether other solutions exist. Mathematically speaking, we are interested in those values of $\lambda$ for which new solutions emerge.  Locally, these are the values of $\lambda$ for which the implicit function theorem is not applicable to the operator $E: \sspace \rightarrow \sspace$ given by 
\begin{align}\label{eqt:Euler Lagrange operator}
E(\phi,\lambda):=\lambda \phi(p)-\frac{1}{Z(\phi)}\int_{\mathbb{S}^2} k(p\cdot q)\exp(-\phi(q))\;dq.
\end{align}
More details about the space $\sspace$ and the differentiability of $E$ and the regularity of critical points can be found in Appendix \ref{app: regularity of the EL eqt}. In particular, $E$ is infinitely many times Fr\'{e}chet differentiable as a mapping from $\sspace \rightarrow \sspace$, see Lemma \ref{lem:differentiablityE}. The implicit function theorem is not applicable if
$\mathcal{L}_\lambda:\sspace \times \mathbb{R}\rightarrow \sspace$,
given by
\begin{align*}
\mathcal{L}_\lambda(\phi):=\frac{dE(\phi+\epsilon \eta,\lambda)}{d\epsilon}\bigg|_{\epsilon=0}=\lambda \phi(p)+\frac{1}{4\pi}\int_{\mathbb{S}^2} k(p\cdot q)\phi(q)\;dq,
\end{align*}
is not invertible. Hence we are interested in all non-zero solutions $(\phi,\lambda)$ of
$$\lambda \phi(p)+\frac{1}{4\pi}\int_{\mathbb{S}^2} k(p\cdot q)\phi(q)\;dq=0$$ and therefore in the nullspace of the operator $\mathcal{L}_{\lambda}$.
Rearranging the equation $\mathcal{L}_{\lambda}\eta(p)=0$, it becomes apparent that this is in fact equivalent to the eigenvalue problem for the interaction operator in \eqref{eqt:interaction operator}
\begin{align*}
U\eta(p)=\int_{\mathbb{S}^2} k(p\cdot q)\eta(q)\;dq.
\end{align*}
In particular, we observe that for any eigenvalue $\mu$ of $U$, a solution to the Euler-Lagrange equation is given by the corresponding eigenvector of $U$ and
\begin{align}\label{eqt:relation bifurcation point eigenvalue}
\lambda=-\frac{\mu}{4\pi}.
\end{align}
Even though this eigenvalue problem involves only a linear integral operator, it is not trivial. In the following theorem, we derive an explicit expression for all eigenvalues of $U$, and thus we find the set of all possible bifurcation points of the Onsager free-energy functional locally around $\phi_0$. In particular we make use of the fact that functions in $L^2(\SP)$ can be expanded in terms of spherical harmonics, see Appendix \ref{app: notation spherical harmonics}. In case of the function space $\sspace$ these expansions converge uniformly.

\begin{thm}\label{thm:eigenvalues and eigenfunctions of general interaction kernels}
Let $K(\cdot,\cdot):\SP\times \SP\rightarrow\R$ be an interaction kernel satisfying Assumptions \textup{\ref{assumption} (a)-(c)} and assume that $K(\cdot,\cdot)$ admits a Taylor expansion such that $K(p,q)=\sum_{r=0}^\infty a_r(p\cdot q)^r$ converges for all $(p,q)\in \{p,q\in \SP| |p\cdot q|<1\}$ and $a_r$ satisfy the condition that 
\begin{align}\label{eqt: new assumption 5d}
\sum_{l=0}^\infty\sum_{r=l,l+2,l+4,\dots} \left|\frac{4\pi (4l+1)^{3/2} a_r r!}{2^{(r-l)/2}(\frac{1}{2}(r-l))!(l+r+1)!!}\right|<\infty.
\end{align}
Then the eigenfunctions of the corresponding interaction operator $$U\eta (p)=\int_{\mathbb{S}^2}k(p\cdot q)\eta(q)\;dq$$ are given by the spherical harmonics $Y^m_l:\SP \rightarrow \mathbb{C}$.
The corresponding eigenvalues are given by
\begin{align}\label{eqt:eigenvalues arbitrary kernel}
\mu_s=\sum_{r=0}^\infty \frac{4\pi a_{s+2r} (s+2r)!}{2^{r}r!(2s+2r+1)!!}
\end{align}
where $s\in \mathbb{N}$.
\end{thm}
\begin{rem}
Observe that we have not used Assumption \textup{\ref{assumption} (d)} in the statement of Theorem \textup{\ref{thm:eigenvalues and eigenfunctions of general interaction kernels}}. Instead, we provide a sufficient condition implying Assumption \textup{\ref{assumption} (d)} by assuming that $K(p,q)$ can be written as a convergent Taylor series and that its coefficients $a_r$ satisfy condition \textup{\eqref{eqt: new assumption 5d}} (if the sum of the absolute values of all eigenvalues in \textup{\eqref{eqt: new assumption 5d}} is finite, the sequence of eigenvalues converges to zero).
\end{rem}

\begin{proof}
By $P_l(x)$ we denote the associated Legendre polynomials of order $l$ (for more details see Appendix \ref{app: notation spherical harmonics}). Lemma \ref{lem: polynomial recurrence Legendre Polynomials} in Appendix \ref{app: eigenvalue chapter} states that
\begin{align}\label{eqt: polynomial expression Legendre polynomials}
x^r=\sum_{l=r,r-2,\dots}\frac{(2l+1)r!}{2^{(r-l)/2}(\frac{1}{2}(r-l))!(l+r+1)!!}P_l(x)
\end{align}
where the sum is taken over all $r\leq l$ such that $l-r\equiv 0 \text{ mod }2$.
By assumption
$$k(p\cdot q)=\sum_{r=0}^\infty a_r (p\cdot q)^r$$
denotes the Taylor expansion of an interaction potential for all $(p,q)\in \{p,q\in \SP| |p\cdot q|<1\}$. Hence an application of \eqref{eqt: polynomial expression Legendre polynomials} yields
$$k(p\cdot q)=\sum_{r=0}^\infty \sum_{l=r,r-2,\dots}a_r\frac{(2l+1)r!}{2^{(r-l)/2}(\frac{1}{2}(r-l))!(l+r+1)!!}P_l(p\cdot q).$$
Using the addition theorem for Legendre polynomials \cite[page $395$]{WhittakerWatson1996} 
$$P_l(x\cdot x')=\frac{4\pi}{2l+1}\sum_{m=-l}^lY_l^m(x){Y^\star}_l^m(x'),$$
we obtain that
\begin{align}\label{eqt:Expansion of k(pq) in terms of SH}
k(p\cdot q)=\sum_{r=0}^\infty \sum_{l=r,r-2,\dots}\sum_{m=-l}^l\underbrace{\frac{4\pi a_r r!}{2^{(r-l)/2}(\frac{1}{2}(r-l))!(l+r+1)!!}}_{=:c_l^r}Y_l^m(p){Y^\star}_l^m(q).
\end{align}
Integrating the interaction kernel against an arbitrary spherical harmonic $Y^n_s(q)$, swapping the order of the sum and the integral and using the orthogonality of spherical harmonics, it follows that
\begin{align}\label{eqt:intgrating interaction operator and SH}
\int_{\mathbb{S}^2}\hspace{-0.1cm}k(p\cdot q)Y^n_s(q)\;dq\notag&=\sum_{r=0}^\infty \sum_{l=r,r-2,\dots}\sum_{m=-l}^l\hspace{-0.1cm}c_l^rY_l^m(p)\int_{\mathbb{S}^2}{Y^\star}^m_l(q)Y^n_s(q)\mathbbm{1}_{|p\cdot q|<1}\;dq\notag\\
&=\sum_{r=0}^\infty \sum_{l=r,r-2,\dots}\sum_{m=-l}^lc_l^rY_l^m(p)\delta_{sl}\delta_{mn}.
\end{align}
Notice that the interchange of the infinite sum and the integral needs of course to be verified. Viewing the infinite sum as an integral with respect to the counting measure, Fubini's theorem applies if
\begin{align}\label{eqt: original condition for Fubini}
\sum_{r=0}^\infty\int_{\mathbb{S}^2}\left|\sum_{l=r,r-2,\dots}\sum_{m=-{l}}^{l} c_l^rY_{l}^m(p){Y^\star}_{l}^m(q)Y^n_s(q)\right|\;dq<\infty.
\end{align}
In particular, using the Cauchy-Schwarz inequality and the fact that all spherical harmonics have unit mass with respect to the $L^2$-norm, we obtain 
\begin{align*}
\sum_{r=0}^\infty\int_{\mathbb{S}^2}&\left|\sum_{l=r,r-2,\dots}\sum_{m=-{l}}^{l} c_l^rY_{l}^m(p){Y^\star}_{l}^m(q)Y^n_s(q)\mathbbm{1}_{|p\cdot q|<1}\right|\;dq\\
&\leq\sum_{r=0}^\infty\sum_{l=r,r-2,\dots}\sum_{m=-{l}}^{l} |c_l^r||Y_{l}^m(p)|\int_{\mathbb{S}^2}\left|{Y^\star}_{l}^m(q)Y^n_s(q)\mathbbm{1}_{|p\cdot q|<1}\right|\;dq\\
&\leq\sum_{r=0}^\infty\sum_{l=r,r-2,\dots}\sum_{m=-{l}}^{l} |c_l^r||Y_{l}^m(p)|.
\end{align*}
Moreover, it follows from \cite[Proposition 7.0.1]{Garrett2011} that any spherical harmonic $Y^m_l$ is bounded by
$$||Y^m_l||_\infty\leq \sqrt{\frac{2l+1}{\text{vol }(\mathbb{S}^{n-1})}}.$$
Hence, 
\begin{align}
\sum_{r=0}^\infty\int_{\mathbb{S}^2}&\left|\sum_{l=r,r-2,\dots}\sum_{m=-{l}}^{l} c_l^rY_{l}^m(p){Y^\star}_{l}^m(q)Y^n_s(q)\mathbbm{1}_{|p\cdot q|<1}\right|\;dq\notag\\
&\leq\frac{1}{4\pi}\sum_{r=0}^\infty\sum_{l=r,r-2,\dots} (4l+1)^{3/2}|c_l^r|.\label{eqt: order of summation}
\end{align}
Again we can view the two sums as integrals with respect to the counting measure and we may swap the order of summation according to Fubini's theorem if the absolute value of the underlying function is integrable with respect to any order of integration. Therefore the validity of swapping the order of summation in \eqref{eqt: order of summation} is simultaneously proved when we establish our original claim in \eqref{eqt: original condition for Fubini}. Therefore a basic condition that suffices to be proved in order to guarantee the validity of the interchange of the integrals in \eqref{eqt:intgrating interaction operator and SH} is given by 
\begin{align}\label{eqt: condition for swapping integrals}
\sum_{l=0}^\infty\sum_{r=l,l+2,l+4,\dots} (4l+1)^{3/2}|c_l^r|<\infty.
\end{align}
In particular, one needs to show on a case to case basis that the coefficients $|c_l^r|$ decay faster than $l^{3/2}$. In Lemma \ref{lem:interchange of the sum and the integral in case of the Onsager kernel} we prove this condition in case of the Onsager kernel. For all other cases this condition is assumed to hold (see the statement of the theorem).

Having established \eqref{eqt:intgrating interaction operator and SH}, we deduce that
\begin{align}\label{eq:UappliedToSH}
\int_{\mathbb{S}^2}k(p\cdot q)Y^n_s(q)\;dq=\sum_{r=0}^\infty \frac{4\pi a_{s+2r} (s+2r)!}{2^{r}r!(2s+2r+1)!!}Y_s^n(p).
\end{align}
From the last equation we may conclude that the eigenfunctions of the interaction operator $U\eta (p)=\int_{\mathbb{S}^2}k(p\cdot q)\eta(q)\;dq$ associated to the class of two-body interaction potentials $k(p\cdot q)$ described in the statement of this theorem are given by the spherical harmonics $Y_k^n$ and that their corresponding eigenvalues are given by
\begin{align*}
\mu_s=\sum_{r=0}^\infty \frac{4\pi a_{s+2r} (s+2r)!}{2^{r}r!(2s+2r+1)!!}
\end{align*}
with $s\in \mathbb{N}$.
\end{proof}

\begin{cor}\label{cor:eigenvalues and eigenfunctions of the Onsager kernel}
The eigenfunctions and eigenvalues of the interaction operator $U$ associated to the Onsager kernel $k_O(p\cdot q)=\sqrt{1-(p\cdot q)^2}$ are given by the spherical harmonics $$\{Y^n_s\in L^2(\mathbb{S}^2):s\in \mathbb{N} \text{ even }, -s\leq n\leq s\}$$ and
$$\mu_O(s):=\begin{cases}-\frac{\pi\Gamma(s/2+\frac{1}{2})\Gamma(s/2-1/2)}{2\Gamma(s/2+1)\Gamma(s/2+2)}&\text{ if }s\text{ is even}\\ 0 &\text{ if }s\text{ is odd}\end{cases}.$$
\end{cor}

\begin{proof}
Based on the Taylor expansion of the square root function $\sqrt{1-x}$
for all $x\in (-1,1)$, it is easy to see that the Taylor expansion of the Onsager kernel is given by
$$k_O(p\cdot q)=\sum_{r=0}^\infty \frac{(2r)!}{(1-2r)(r!)^2(4^r)}(p\cdot q)^{2r}$$
for all $p,q\in \SP$ such that $|p\cdot q|<1$. We deduce that
\begin{align*}
k_O(p\cdot q)=\sum_{r=0}^\infty a_r (p\cdot q)^r
\text{ with }a_{r}=\begin{cases}\frac{r!}{(1-r)(\frac{r}{2}!)^2(2^r)}&\text{ if }r\text{ is even}\\0 &\text{ if }r\text{ is odd}\end{cases}
\end{align*}
and thus using (\ref{eqt:Expansion of k(pq) in terms of SH})
\begin{align}\label{eq:Onsager kernel as series}
k_O&(p\cdot q)\notag\\
&=\sum_{r=0}^\infty \sum_{l=2r,2(r-1),\dots}\sum_{m=-l}^l \frac{4\pi (2r)!^2Y_l^m(p){Y^\star}_l^m(q)}{(1-2r)(r!)^2(4^r)2^{r-l/2}(r-l/2)!(l+2r+1)!!}\notag\\
&=\sum_{r=0}^\infty \sum_{l=0}^r\sum_{m=-2l}^{2l} \frac{4\pi (2r)!^2Y_{2l}^m(p){Y^\star}_{2l}^m(q)}{(1-2r)(r!)^2(4^r)2^{r-l}(r-l)!(2l+2r+1)!!}.
\end{align}

Having expressed the Onsager kernel in terms of spherical harmonics, we may apply (\ref{eqt:eigenvalues arbitrary kernel}) in Theorem \ref{thm:eigenvalues and eigenfunctions of general interaction kernels} in order to get
\begin{align*}
\int_{\mathbb{S}^2}&k_O(p\cdot q)Y^n_s(q)\,dq\\
&=\int_{\mathbb{S}^2}\bigg(\sum_{r=0}^\infty \sum_{l=0}^r\sum_{m=-{2l}}^{2l} \frac{4\pi (2r)!^2Y_{2l}^m(p){Y^\star}_{2l}^m(q)Y^n_s(q)\mathbbm{1}_{|p\cdot q|<1}}{(1-2r)(r!)^2(4^r)2^{r-l}(r-l)!(2l+2r+1)!!}\bigg)\;dq\notag\\
&=\sum_{l=0}^\infty\sum_{r=l}^\infty \frac{4\pi (2r)!^2}{(1-2r)(r!)^2(4^r)2^{r-l}(r-l)!(2l+2r+1)!!}\hspace{-0.2cm}\sum_{m=-2l}^{2l}\hspace{-0.2cm}Y_{2l}^m(p)
\delta_{mn}\delta_{(2l)s}\\
&=\begin{cases}\sum_{r=s/2}^\infty \frac{4\pi (2r)!^2}{(1-2r)(r!)^2(4^r)2^{r-s/2}(r-s/2)!(s+2r+1)!!}Y_s^n(p) &\text{ if }s\text{ is even}\\ 0 &\text{ if }s \text{ is odd}\end{cases}
\end{align*}
where we used the orthogonality of the spherical harmonics and Lemma \ref{lem:interchange of the sum and the integral in case of the Onsager kernel} in which we verify the condition in \eqref{eqt: condition for swapping integrals} that allows an interchange of the sum and the integral.


%
%

In fact, using the theory of generalised hypergeometric functions, it can be shown that this expression for the eigenvalues can be written in a closed form (for details see Lemma \ref{lem:closed form of eigenvalues for the Onsager kernel} in Appendix \ref{app: eigenvalue chapter}). It reads
\begin{align*}
\mu_{O}(s)= \begin{cases}-\frac{\pi \Gamma(s/2+\frac{1}{2})\Gamma(s/2-1/2)}{2\Gamma(s/2+1)\Gamma(s/2+2)}&\text{ if }s\text{ is even}\\ 0 &\text{ if }s\text{ is odd}\end{cases}.
\end{align*}
\end{proof}

Based on Corollary \ref{cor:eigenvalues and eigenfunctions of the Onsager kernel} and \eqref{eqt:relation bifurcation point eigenvalue}, we therefore conclude that a set of all possible bifurcation points is given by
$$\lambda_O(s)=\begin{cases}\frac{\Gamma(s/2+\frac{1}{2})\Gamma(s/2-1/2)}{8\Gamma(s/2+1)\Gamma(s/2+2)}&\text{ if }s\text{ is even}\\ 0 &\text{ if }s\text{ is odd}\end{cases}.$$

\begin{rem}
We will show in Theorem \textup{\ref{thm: rabinowitz}} in Section \ref{sec: Qualitative behaviour of the global bifurcation diagram} that in fact any point $\lambda_{O}(s)$ is a bifurcation point.
\end{rem}

\section{The bifurcation equation of the Onsager free-energy functional}\label{sec: Lyapunov Schmidt reduction and bifurcatione quation}
Having established a general method for deriving the bifurcation points of the interaction operator, we direct our attention to the corresponding bifurcation equations. We would like to mention that this procedure is quite general and may be applied to a very large class of interaction kernels that satisfy Assumptions \ref{assumption} (a)-(c). However, in the following analysis we will restrict our attention again to the Onsager potential.

The presentation of our results is divided into two parts. We begin by giving a brief introduction to the Lyapunov-Schmidt decomposition, which is our main tool for computing the bifurcation equation, see Section \ref{sec:Lyapunov Schmidt}. In Section \ref{sec: practicalities} we focus on the practicalities that are involved when we apply the Lyapunov-Schmidt decomposition to the Euler-Lagrange operator given in \eqref{eqt:Euler Lagrange operator}. In particular, Section \ref{sec: practicalities} is divided into four parts: a reformulation of the problem in the language of Section \ref{sec:Lyapunov Schmidt}, see Section \ref{sec: reformulation of the problem}, an algorithmic procedure that allows us to fulfil the decomposition, see Sections \ref{sec:implicit function theorem} and \ref{sec: step b-derivation of the bifurcation equation}, and the presentation of an algorithm that allows the fast computation of products of spherical harmonics in terms of spherical harmonics, see Section \ref{sec:products of SHs}.


\subsection{The theory: The Lyapunov-Schmidt decomposition}\label{sec:Lyapunov Schmidt}
This brief introduction to the Lyapunov-Schmidt decomposition is based on \cite{Chossat-Lauterbach2000}. The main idea of the Lyapunov-Schmidt decomposition is the reduction of a possibly infinite-dimensional bifurcation problem to a finite-dimensional one.
Let 
\begin{align}\label{eqt: equation of interest, LS decomposition}
F(w,\lambda)=0
\end{align}
be the equation of interest with $w\in X$, $\lambda \in \mathbb{R}$ and $F:X\times \mathbb{R}\rightarrow \mathbb{R}$ where $X$ denotes a Banach space. Without loss of generality we assume that $$F(0,0)=0.$$ 
The linear and non-linear parts of \eqref{eqt: equation of interest, LS decomposition} decompose as
$$F(w,\lambda)=\mathcal{L}(w,\lambda)+\mathcal{R}(w,\lambda),$$
where
\begin{align}\label{eqt: definition L}
\mathcal{L}(\eta):=\frac{\partial F(0+\epsilon \eta,0)}{\partial \epsilon}\Bigg |_{\epsilon=0}
\end{align}
and where $\mathcal{R}$ denotes its remainder. We assume that $\mathcal{L}$ is a Fredholm operator of index zero, that means that the kernel $N(\mathcal{L})$ has finite dimension $d$, the range $R(\mathcal{L})$ has finite codimension $r$ and that both dimensions $d$ and $r$ agree.

The idea of the Lyapunov-Schmidt decomposition is to reduce the dimension of the equation and its solution by projecting it onto the kernel of $\mathcal{L}$ which is, due to the assumptions on $\mathcal{L}$, finite-dimensional. Let $P:X\rightarrow N(\mathcal{L})$ denote the projection onto the nullspace of $\mathcal{L}$ and let $(1-P)$ be the projection onto its complement. Then \eqref{eqt: equation of interest, LS decomposition} is equivalent to the system of equations 
\begin{subequations}
\begin{align}
PF(u+v,\lambda)=0\label{eqt: projected onsager euler lagrange equation - P}\\
(1-P)F(u+v,\lambda)=0\label{eqt: projected onsager euler lagrange equation - Q}
\end{align}
\end{subequations}
where $u:=Pw$ denotes the projection of the solution onto $N(\mathcal{L})$ and $v:=(1-P)w$ its projection onto the complement. The second of these equations can uniquely be solved for $v$ using the implicit function theorem and the solution $v(u,\lambda)$ can then be plugged back into the first equation, which yields the bifurcation equation
\begin{align}\label{eqt:definition bifurcation equation}
f(u,\lambda):=P\mathcal{R}(u+v(u,\lambda),\lambda).
\end{align} 
Note that we dropped the first term appearing in (\ref{eqt: projected onsager euler lagrange equation - P}) because $\mathcal{L}$ is linear, and thus it follows that $P\mathcal{L}(u+v)=0$. Similarly, we can drop the linear part of \eqref{eqt: projected onsager euler lagrange equation - Q}. The solutions of \eqref{eqt:definition bifurcation equation} are equivalent to the solutions of our original problem in \eqref{eqt: equation of interest, LS decomposition}, see \cite{Chossat-Lauterbach2000}, because it only depends on $u$, which is an element of the finite-dimensional kernel of $\mathcal{L}$, and thus reduces the possibly infinite-dimensional problem to a finite-dimensional one. 

\subsection{Practicalities: An algorithmic procedure to derive the bifurcation equation}\label{sec: practicalities}
According to the previous section, the computations for the derivation of the bifurcation equation in \eqref{eqt:definition bifurcation equation} are divided into two parts:

\begin{enumerate}[leftmargin=*,labelindent=16pt,label=\bfseries Step \Alph*.]
  \item An application of the implicit function theorem to (\ref{eqt: projected onsager euler lagrange equation - Q}) that gives us an explicit expression of $v$ in terms of $u$ and $\lambda$.
  \item Plugging $v$ back into (\ref{eqt: projected onsager euler lagrange equation - P}) which results in the bifurcation equation.
\end{enumerate}

By using the implicit function theorem in the first of these two steps, it is often not possible to derive an expression in closed form. Instead, we obtain a Taylor expansion of $v(u,\lambda)$  (which is justified by the  implicit function theorem \cite[Theorem 2.3]{Ambrosetti1993nonlinear}) and its coefficients are obtained by matching those of the same order that occur in \eqref{eqt: projected onsager euler lagrange equation - P}. The resulting expression for $v(u,\lambda)$ 
can then be plugged into the expansion of $\mathcal{R}$. An application of the projection $\mathcal{P}$ finally yields an expansion of the bifurcation equation in \eqref{eqt:definition bifurcation equation}.

In order to characterise the bifurcation occurring at $\lambda_2$, it is often sufficient to obtain an expansion of \eqref{eqt:definition bifurcation equation} up to a particular order. We know that a certain order is sufficient by looking at the so-called recognition problem that corresponds to the problem at hand. More details on the particular recognition problem that we need to apply in order to solve the bifurcation equation can be found in Section \ref{sec: solving the bifurcation equation}, where we will show that it suffices to compute the bifurcation equation up to third order.

In the following, we will first restate the problem in the language of Section \ref{sec:Lyapunov Schmidt} before directing our attention to the first of the two main steps mentioned above in Section \ref{sec:implicit function theorem}. In order to make all computations tractable, we write all expressions in terms of spherical harmonics. An essential tool for the comparison of coefficients with matching order is an algorithmic procedure that allows us the fast computation of products of spherical harmonics with a computer. This will be the subject of Section \ref{sec:products of SHs}. We conclude this section by performing Step B above which yields the bifurcation equation in five dimensions.



\subsubsection{Reformulation of the problem}\label{sec: reformulation of the problem}
In order to find the minimisers of the free-energy functional, we consider its Euler-Lagrange equation and thus the corresponding Euler-Lagrange operator
\begin{align}\label{eqt: translated Euler Lagrange operator}
E(\phi,\lambda)=(\lambda_s+\lambda)\phi(p)-\frac{1}{Z(\phi)}\int_{\mathbb{S}^2} k(p\cdot q)\exp(-\phi(q))\;dq.
\end{align}
\begin{rem}
Notice that \eqref{eqt: translated Euler Lagrange operator} is an update of the Euler-Lagrange operator in \eqref{eqt:Euler Lagrange operator} which has been translated by a factor $\lambda_s$ in order to ensure that the assumption $E(0,0)=0$ holds.
\end{rem}
The constant $\lambda_s$ denotes the bifurcation point of interest. According to \eqref{eqt:relation bifurcation point eigenvalue}, it is given by
$$\lambda_s=-\frac{\mu_s}{4\pi},$$ 
where $\mu_s$ denotes the eigenvalue of order $s$ of the interaction operator associated to the Onsager kernel, see Corollary \ref{cor:eigenvalues and eigenfunctions of the Onsager kernel} for details. The variable $\lambda$ denotes the deviation from this bifurcation point and will act as bifurcation parameter. Having this translation of the problem in mind, we are interested in bifurcations around the point $(\phi,\lambda)=(0,0)$. In particular, we are looking for the first phase transformation that occurs when the temperature is descending. Hence we consider the largest bifurcation point which is, in case of the Onsager kernel and according to the result of Corollary \ref{cor:eigenvalues and eigenfunctions of the Onsager kernel}, given by
$$\lambda_2=-\frac{\mu_2}{4\pi}=\frac{\pi}{32}.$$ 
In order to derive the decomposition of the above functional into its linear and non-linear part, we consider its Taylor expansion $\hat{E}$ up to fourth order in both variables $\phi$ and $\lambda$ around the point $(0,0)$, namely
\begin{align}\begin{aligned}\label{eqt:taylor expansion of EL}
\hat{E}(&\phi,\lambda)\\
:=&(\lambda_2+\lambda)\phi(p)+\frac{1}{4\pi}\int_{\SP}k(p\cdot q)\phi(q)\;dq+\frac{\int_{\SP}\phi(q)\;dq}{16\pi^2}\int_{\SP}k(p\cdot q)\phi(q)\;dq\\
&-\frac{1}{8\pi}\int_{\SP}k(p\cdot q)\phi^2(q)\;dq+\frac{\left(\int_{\SP}\phi(q)\;dq\right)^2}{64 \pi^3}\int_{\SP}k(p\cdot q)\phi(q)\;dq\\
&-\frac{\int_{\SP}\phi^2(q)\;dq}{32 \pi^2}\int_{\SP}k(p\cdot q)\phi(q)\;dq-\frac{\int_{\SP}\phi(q)\;dq}{32\phi^2}\int_{\SP}k(p\cdot q)\phi^2(q)\;dq\\
&+\frac{1}{24\pi}\int_{\SP}k(p\cdot q)\phi^3(q)\;dq+\frac{\left(\int_{\SP}\phi(q)\;dq\right)^3}{256 \pi^4}\int_{\SP}k(p\cdot q)\phi(q)\;dq\\
&-\frac{\int_{\SP}\phi(q)\;dq}{64 \pi^3}\int_{\SP}\phi^2(q)\;dq\int_{\SP}k(p\cdot q)\phi(q)\;dq\\
&+\frac{\int_{\SP}\phi^3(q)\;dq}{96\pi^2}\int_{\SP}k(p\cdot q)\phi(q)\;dq-\frac{\left(\int_{\SP}\phi(q)\;dq\right)^2}{128\pi^3}\int_{\SP}k(p\cdot q)\phi^2(q)\;dq\\
&+\frac{\int_{\SP}\phi^2(q)\;dq}{64 \pi^2}\int_{\SP}k(p\cdot q)\phi^2(q)\;dq +\frac{\int_{\SP}\phi(q)\;dq}{96 \pi^4}\int_{\SP}k(p\cdot q)\phi^3(q)\;dq\\
&-\frac{1}{96 \pi}\int_{\SP}k(p\cdot q)\phi^4(q)\;dq.
\end{aligned}
\end{align}
According to \eqref{eqt: definition L}, its linear and non-linear parts therefore decompose as 
\begin{align*}
\mathcal{L}(\phi)=\lambda_2 \phi(p)+\frac{1}{4\pi}\int_{\mathbb{S}^2} k(p\cdot q)\phi(q)\;dq
\end{align*}
and
\begin{align}\label{eqt:non linear terms}
\hat{\mathcal{R}}(\phi,\lambda):=\hat{E}(\phi,\lambda)-\mathcal{L}(\phi),
\end{align}
respectively. The fact that $\mathcal{L}$ is a Fredholm operator of index zero is a consequence of the following corollary in \cite{Abramovich2002}.
\begin{citedcor}[{{\cite[Corollary 4.47]{Abramovich2002}}}]\label{cor:fredholm operator}
Let $T,K:\mathcal{X}\rightarrow \mathcal{Y}$ be linear operators and assume in addition that $T$ is a Fredholm operator and that $K$ is compact. Then $T+K$ is also a Fredholm operator with $\mathrm{index }(T+K)=\mathrm{index }(T)$.
\end{citedcor}
In Lemma \ref{lem:compactInteraction} in Appendix \ref{app: regularity of the EL eqt} we prove that the interaction operator $$U(\phi)(p)=\int_{\mathbb{S}^2}k(p\cdot q)\exp(-\phi(q))\;dq$$ is a compact operator for all interaction kernels satisfying Assumption \ref{assumption}. In order to prove compactness, it is crucial that the operator satisfies Assumption \ref{assumption} (d). Schematically, $\mathcal{L}$ can be written as
$$\mathcal{L}=C \text{ Id}+K.$$ Hence in order to apply Corollary \ref{cor:fredholm operator}, we only have to show that the identity is a Fredholm operator of index zero. Let us make the following observations
\begin{enumerate}[leftmargin=*,labelindent=16pt,label=(\alph*)]
\item $\text{Id}$ is bounded and linear
\item $\mathcal{N}(\text{Id})=\{0\}$ and $\text{coker}(\text{Id})=\{0\}$.
\end{enumerate}
Hence the identity is a Fredholm operator of index zero and thus so is $\mathcal{L}$ as claimed.

In view of the results of Section \ref{sec:Bifurcation Points Onsager Functional}, it is easy to conclude that the kernel of the operator $\mathcal{L}$, denoted by $N(\mathcal{L})$, is the eigenspace of the spherical harmonics corresponding to the degree of the bifurcation point of interest, which is in our case $\lambda_2$. Thus,
$$N(\mathcal{L})=\{Y^m_2:-2\leq m\leq 2\}$$
which is a five-dimensional space.
It follows that the projection $P$ onto $N(\mathcal{L})$ is given by
$$P:L^2(\mathbb{S}^2)\rightarrow L^2(\mathbb{S}^2)\text{ where } Pw:=\sum_{m=-2}^2(w,Y^m_2)_{L^2}Y^m_2 \text{ for all } w\in L^2(\mathbb{S}^2).$$
Let $u(p):=P\phi(p)$ and $v(p):=(1-P)\phi(p)$ such that $u(p)+v(p)=w(p)$ and let us further define $u_m$ and $v_{l,m}$ such that 
\begin{align}\label{eqt: definition of u in terms of SH}
u(p):=\sum_{m=-2}^2u_{m}Y^m_2(p) \quad \text{ and }\quad v(p):=\sum_{\substack{l=0\\l\neq 2}}^\infty\sum_{m=-l}^l v_{l,m}Y^m_l(p),
\end{align}
respectively. We will see that it will be an essential step to write all expressions in terms of spherical harmonics as this simplifies the actions of all operators. In particular, we may deduce from Theorem \ref{thm:eigenvalues and eigenfunctions of general interaction kernels} that
$$U(Y^m_l)=\int_{\mathbb{S}^2}k(p\cdot q)Y^m_l(q)\;dq=\mu_lY^m_l(p).$$
Accordingly, we conclude that 
$$\mathcal{L}(Y^m_l)=\left(\lambda_2+\frac{\mu_l}{4\pi}\right)Y^m_l$$
and similarly, 
\begin{align}\label{eqt:inverse of L}
\mathcal{L}^{-1}(Y^m_l)=\frac{4\pi}{4\pi\lambda_2+\mu_l}Y^m_l.
\end{align}
Since the set of all spherical harmonics is an orthonormal basis for the space $L^2(\SP)$ and orthogonal for $\sspace$, these operators can in fact be viewed as multiplication operators on the space of functions $L^2(\SP)$, such that $\mathcal{L}w(p)=\left(\lambda_2+\frac{\mu_l}{4\pi}\right)w(p)$ for all $w\in L^2(\SP)$. This leads to an enormous simplification of our subsequent computations.

\subsubsection{Step A: An application of the Implicit Function Theorem}\label{sec:implicit function theorem}
It is straightforward to show that the implicit function theorem is applicable to
\begin{align}
\mathcal{L}v+(1-P)\mathcal{R}(u+v,\lambda)=0.\tag{\ref{eqt: projected onsager euler lagrange equation - Q}}
\end{align}
The operator $\mathcal{L}$ is invertible in $R(\mathcal{L})$ and hence its inverse is bounded. Since $D_u\mathcal{R}(0,0)=0$ by definition, the implicit function theorem is applicable. However, writing down the exact solution $v(u,\lambda)$ is not possible in practice. Instead, an algorithmic procedure is required which has been presented in \cite{Chossat-Lauterbach2000}. It is mainly based on writing $v$ as a Taylor expansion in terms of both variables $u$ and $\lambda$ which is justified by the differentiability of the implicit function, see implicit function theorem \cite[Theorem 2.3]{Ambrosetti1993nonlinear}. In particular, we expand $v=(1-P)\phi$ up to fourth order by
\begin{align*}
\hat{v}(u,\lambda)(p):=\sum_{0\leq i+j \leq 4}\frac{1}{i!j!}\hat{v}_{i,j}(u,\lambda)(p),
\end{align*}
where the terms $\hat{v}_{i,j}$ denote terms of $i^{\text{th}}$ order in $u$ and of $j^{\text{th}}$ order in $\lambda$. Furthermore, we identify $u$ as before with
$$u(p)=\sum_{m=-2}^2u_mY^m_2(p).$$ 
By definition $v\in N(\mathcal{L})^\perp$ and thus, without loss of generality, we assume that $\hat{v}_{i,j}\in (N(\mathcal{L}))^\perp$ for all $i,j$. 
Plugging this expression into (\ref{eqt: projected onsager euler lagrange equation - Q}) 
\begin{align*}
\mathcal{L}(v)=-(1-P)\mathcal{R}(u+v,\lambda)
\end{align*}
and matching the terms of the right order on both sides, we obtain expressions for each of the $\hat{v}_{ij}$. We will execute this calculation for the first couple of terms.
\begin{rem}
Because the spherical harmonics form an orthonormal basis for $L^2(\SP)$, the $v_{i,j}$ can be written as an infinite series expansion. In fact, when considering the first few steps of the procedure, we will see that this expansion is finite.
\end{rem}
Since there does not exist any term that is constant in both variables, we conclude that $$\hat{v}_{0,0}=0.$$ As a second step, we consider $\hat{v}_{1,0}$ which only depends on $u$. Hence we only consider terms of order one in $u$ on both sides of (\ref{eqt: projected onsager euler lagrange equation - Q}) taking the expansion of $\hat{E}$ in (\ref{eqt:taylor expansion of EL}) and the definition of $\mathcal{R}$ in (\ref{eqt:non linear terms}) into account. All terms on the right hand side depend on $\lambda$ or on higher  order terms of $u$, so they can be neglected when we match the coefficients. Thus, we obtain
$$\lambda \hat{v}_{1,0}(p)+\frac{1}{4\pi}\int_{\SP}k(p\cdot q)\hat{v}_{1,0}(q)\;dq=0.$$
This is equivalent to solving $\mathcal{L}\hat{v}_{1,0}=0$. Because we are assuming that $\hat{v}_{i,j}\in N(\mathcal{L})^\perp$, we conclude that $\hat{v}_{1,0}=0$. Similarly, one can deduce that $$\hat{v}_{0,1}=0.$$
As the first non-zero term, we consider $\frac{1}{2} \hat{v}_{2,0}$. Matching the order of terms on both sides of (\ref{eqt: projected onsager euler lagrange equation - Q}), we obtain
\begin{align*}
\frac{\lambda}{2}&\hat{v}_{2,0}(p)+\frac{1}{8\pi}\int_{\SP}k(p\cdot q)\hat{v}_{2,0}(q)\;dq\\
&=-(1-P)\left(\frac{\int_{\SP}u(q)\;dq}{16\pi^2}\int_{\SP}k(p\cdot q)u(q)\;dq-\frac{1}{8\pi}\int_{\SP}k(p\cdot q)u^2(q)\;dq\right).
\end{align*}
This yields an explicit expression for $\hat{v}_{2,0}$ if we bear in mind that the operator $\mathcal{L}$ and its inverse $\mathcal{L}^{-1}$ can actually be interpreted as simple multiplications, see (\ref{eqt:inverse of L}). However, in order to be able to apply this multiplication to the right hand side, we need to write it in terms of spherical harmonics. Recall that $u$ itself is already given as an expansion in terms of spherical harmonics, see \eqref{eqt: definition of u in terms of SH},
$$u(p)=\sum_{m=-2}^2u_{m}Y^m_2(p).$$
Hence we only have to expand the terms
$$\int_{\SP}k(p\cdot q)u(q)\;dq \int_{\SP}u(q)\;dq \quad \text{and}\quad \int_{\SP}k(p\cdot q)u^2(q)\;dq.$$
In fact,
$$\int_{\mathcal{S}^2}k(p\cdot q)u(q)\;dq \int_{\SP}u(q)\;dq=0$$
because the second factor, being the sum of integrals of spherical harmonics of degree $l=2$, is zero. In case of the second term, we observe that we have to compute $u^2$ which is based on computing products of spherical harmonics. In order to fulfil this task, we have developed an algorithm that is presented in Section \ref{sec:products of SHs}. 
Applying it to the case above yields an expansion of $u^2$ in terms of spherical harmonics. We are now in the position to apply the interaction operator $U$, the projection $1-P$ and the operator $\mathcal{L}^{-1}$ which, according to \eqref{eqt:inverse of L}, reduces to a simple multiplication in this case. We obtain
\begin{align*}
\hat{v}_{2,0}(p)=&-\mathcal{L}^{-1}\left((1-P)\int_{\SP}k(p\cdot q)u^2(q)\;dq)\right)
\end{align*}
\begin{align*}
=&\frac{\sqrt{\frac{5}{14}\pi ^{3}}}{64}u_{-2}^2 Y_4^{-4}(p)+\frac{\sqrt{\frac{5}{7}\pi ^{3}}}{64}  u_{-1} u_{-2} Y_{4}^{-3}(p)+\frac{\sqrt{15\pi ^{3}}}{448}  u_0 u_{-2} Y_{4}^{-2}(p)\\&
+\frac{\sqrt{5\pi ^{3}}}{448}  u_1 u_{-2} Y_{4}^{-1}(p)+\frac{\pi ^{3/2}}{448}  u_2 u_{-2} Y_{4}^{0}(p)+\frac{\sqrt{\frac{5}{2}\pi ^{3}}}{224}  u_{-1}^2 Y_{4}^{-2}(p)\\
&+\frac{\sqrt{\frac{15}{2}\pi ^{3}}}{224} u_{-1} u_0 Y_{4}^{-1}(p)
+\frac{3\pi ^{3/2}}{448} u_0^2 Y_{4}^{0}+\frac{\pi ^{3/2}}{112}  u_{-1} u_1 Y_{4}^{0}(p)\\&+\frac{\sqrt{\frac{15}{2}\pi ^{3}}}{224} u_0 u_1 Y_{4}^{1}(p)+\frac{\sqrt{5\pi ^{3}}}{448}  u_{-1} u_2 Y_{4}^{1}(p)+\frac{\sqrt{\frac{5}{2}\pi ^{3}}}{224}  u_1^2 Y_{4}^{2}(p)\\
&+\frac{\sqrt{15\pi ^{3}}}{448}  u_0 u_2 Y_{4}^{2}(p)+\frac{\sqrt{\frac{5}{7}\pi ^{3}}}{64}  u_1 u_2 Y_{4}^{3}(p)+\frac{\sqrt{\frac{5}{14}\pi ^{3}}}{64}  u_2^2 Y_{4}^{4}(p).
\end{align*}
This concludes the computation of $\hat{v}_{2,0}$. The computation of the other terms follows the same pattern. However, because they are more complex and tedious, we will omit them. In order to give a good overview of the individual steps for each case, we present a summary of the procedure in form of a general algorithm. 
\begin{algorithm}
\caption{\label{alg:implicit function theorem - matching terms}Matching the terms in Equation (\ref{eqt: projected onsager euler lagrange equation - Q})}
\begin{algorithmic}[1]
  \STATE Input: Fix $\hat{v}_{i,j}$ such that $0\leq i+j \leq 4$.
  \STATE Write down $\mathcal{R}(u+\hat{v})$ keeping only terms in $u$ and all terms in $\hat{v}_{i,j}$ up to $i$'th order and in $\lambda$ up to $j$'th order.
  \STATE Expand $\mathcal{R}(u+\hat{v})$ in terms of spherical harmonics using the results of Section \ref{sec:products of SHs}.
  \STATE Apply the projection $(1-P)$ by dropping all spherical harmonics of degree $l=2$.
  \STATE Apply $-\mathcal{L}^{-1}$ by multiplying $\mathcal{R}(u+\hat{v})$ by $\frac{-4\pi}{4\pi\lambda_2+\mu_l}$.
  \STATE Output: $\hat{v}_{i,j}=\frac{-4\pi}{4\pi\lambda_2+\mu_l}\mathcal{R}(u+\hat{v})$.
\end{algorithmic}
\end{algorithm}
We would like to conclude this section by remarking that the expansion of $\hat{v}$ resulting from this procedure is in fact finite and that this is always the case when considering a finite expansions of an equation. The reason for this is that $u$ is a linear combination of spherical harmonics that span the kernel of $\mathcal{L}$ which is itself finite. Considering products of up to four terms in $u$ and $\hat{v}_{i,j}$, we can therefore only reach spherical harmonics up to degree $l=8^4$ since the highest degree for $\hat{v}_{i,j}$ with $i+j=4$ can only be $l=8$.

\begin{rem}\label{BifurcationGeneralKernel}
The operator $\mathcal{L}$ depends on the bifurcation parameter and the eigenvalues of the interaction operator $U$. For any kernel satisfying Assumption \ref{assumption} the expansion of $v$ would only depend on finitely many eigenvalues corresponding to spherical harmonics with low order. That means the resulting bifurcation can then be studied with the methods presented here.
\end{rem}
\subsubsection{Products of spherical harmonics}\label{sec:products of SHs}
The crucial step in matching the coefficients in (\ref{eqt: projected onsager euler lagrange equation - Q}) is to expand products of spherical harmonics again in terms of spherical harmonics. It is known that this can be achieved using Clebsch-Gordan coefficients which arise in angular momentum coupling \cite{Hussain2014}. However, it is very hard to compute them explicitly because it takes too much computing time. Therefore we devised an algorithm for the computation of products of spherical harmonics with short run time on the computer. The exact algorithm that we implemented on the computer is given in Appendix \ref{sec:algorithm for product of SHs} while its crucial ideas are subject of this section.

We first recall the definition of spherical harmonics (see Appendix \ref{app: notation spherical harmonics})
\begin{align*}
Y_{l}^m(\varphi,\theta)=N_{lm}e^{im\theta}P^m_l(\cos \varphi), \qquad -l\leq m\leq l
\end{align*}
where $\varphi$ and $\theta$ denote the polar and the azimuth angle corresponding to the unit vector $p\in \SP$. The functions $P^m_l$ are called associated Legendre polynomials and their exact definition is also given in Appendix \ref{app: notation spherical harmonics}.
Thus, a product of two spherical harmonics is given by
\begin{align}\label{eqt: explaining products of SH}
Y_{l}^m(\varphi,\theta) Y_{p}^q(\varphi,\theta)=N_{lm}N_{pq}e^{i(m+q)\theta}P_{l}^m(\cos(\varphi))P_{p}^q(\cos(\varphi)).
\end{align}
Notice that we will abuse the notation and drop the $(\phi,\theta)$ or $p$ dependence in due course.
Due to the factor $e^{i(m+q)}$, we observe that the order of each spherical harmonic occurring in an expansion of the product in \eqref{eqt: explaining products of SH} has to be of order $m+q$ while its degree may be arbitrary. On the basis of this observation, we would like to investigate products of associated Legendre polynomials. In fact, we use the following recurrence formula from \cite{Wolfram-ALP}
\begin{align}\label{eqt:recurrence rule for Legendre polynomials}
(l-m+1)P_{l+1}^m=(2l+1)xP_{l}^m-(l+m)P_{l-1}^m.
\end{align}
Rearranging it gives
$$P_{l}^m=\frac{2l-1}{l-m}xP_{l-1}^m-\frac{l-1+m}{l-m}P_{l-2}^m$$
and thus,
$$P_{l}^mP_{p}^q=\frac{2l-1}{l-m}xP_{l-1}^m P_{p}^q-\frac{l-1+m}{l-m}P_{l-2}^m P_{p}^q.$$
Applying the original recurrence rule in (\ref{eqt:recurrence rule for Legendre polynomials}) again, but this time to the product $xP_{p,q}$, yields
$$P_{l}^mP_{p}^q=\frac{2l-1}{l-m}\hspace{-0.1cm}\left(\frac{p-q+1}{2p+1}P_{p+1}^qP^m_{l-1}+\frac{p+q}{2p+1}P_{p-1}^qP^m_{l-1}\hspace{-0.1cm}\right)-\frac{l-1+m}{l-m}P_{l-2}^mP^q_p.$$ 
Using the definition of spherical harmonics, we deduce that
\begin{align*}
Y_{l}^m Y_{p}^q
=&\frac{2l-1}{l-m}\frac{p-q+1}{2p+1}\frac{N_{l,m}}{N_{(l-1),m}}\frac{N_{p,q}}{N_{(p+1),q}}Y^m_{l-1} Y_{p+1}^q
\\&+\frac{2l-1}{l-m}\frac{p+q}{2p+1}\frac{N_{l,m}}{N_{(l-1),m}}\frac{N_{p,q}}{N_{(p-1),q}} Y^m_{l-1} Y_{p-1}^q\\&-\frac{l-1+m}{l-m}\frac{N_{l,m}}{N_{(l-2),m}} Y_{l-2}^mY^q_p
\end{align*}
and thus, we observe that we reduced the degree of the spherical harmonics $Y_{l}^m$ by one or two. We repeat this procedure until we reach either a term of the form $Y^s_{s+1}$ or $Y^s_{s}$ (which has to happen eventually). Hitting $Y^s_{s+1}$ first, we use another recurrence rule from  \cite{Wolfram-ALP}, namely 
$$P_{s+1}^s=(2s+1)xP_{s}^s \quad \text{ leading to }\quad Y^s_{s+1}Y^q_p=\frac{N_{(s+1),s}}{N_{s,s}}(2s+1)xY_{s}^sY^q_p,$$
so that we obtain in any case $Y^s_sY^q_p$. Again we apply a recurrence rule
$$P_{s}^s=(-1)^s(2s-1)!!(1-x^2)^{s/2}$$ 
which yields
$$ Y_{s}^sY^q_p=(-1)^s(2s-1)!!(1-x^2)^{s/2}N_{s,s}Y^q_p.$$ 
This last product can then be resolved by using
$$\sqrt{1-x^2}Y_{p}^q=\frac{1}{2p+1}\left(\frac{N_{p,q}}{N_{(p-1),(q+1)}} Y_{p-1}^{q+1}-\frac{N_{p,q}}{N_{(p+1),(q+1)}} Y_{p+1}^{q+1}\right)$$
repeatedly if necessary. Again this recurrence formula arises from using recurrence formulae for associated Legendre polynomials.

The computation of the first couple of steps of the algorithmic procedure above illustrates how an expansion of a product of two spherical harmonics in terms of spherical harmonics can be obtained. The general method involves a couple of other cases to start with, such as spherical harmonics of negative order for example, which also have to be taken into account and for which special rules need to be defined even though the general method works exactly like this. A detailed list of all rules that are necessary to cover all cases that might occur is given in Appendix \ref{sec:algorithm for product of SHs}. In order to implement this algorithm, it is important to use a conditioned replacement rule that makes it faster. That means that the algorithm itself remembers cases that it has computed already so that it will stop once it hits a known target. 

\subsubsection{Step B: Derivation of the bifurcation equation}\label{sec: step b-derivation of the bifurcation equation}
Having found an explicit expression for $\hat{v}(u,\lambda)$ in Section \ref{sec:implicit function theorem}, we are now in the position to plug this expression into (\ref{eqt: projected onsager euler lagrange equation - P}) which finally gives us an approximation of the bifurcation equation up to fourth order
\begin{align*}
\hat{f}(u,\lambda):=P\mathcal{\hat{R}}(u+\hat{v}(u,\lambda),\lambda).
\end{align*}
\begin{rem}
One can show that the bifurcation equation corresponding to the Onsager free-energy functional is smooth, see Lemma \textup{\ref{lem: smoothness of the bifurcation equation}}. 
\end{rem}
\begin{rem}\label{rem:order}
In fact, we used a fourth order Taylor approximation of the Euler-Lagrange operator and a Taylor approximation of $v(u,\lambda)$ up to a joint fourth order in both variables. Thus, we obtained an approximation of the bifurcation equation of at least fourth order. More precisely,
the bifurcation equation $\hat{f}=\hat{\mathcal{R}}(u+\hat{v}(u,\lambda),\lambda)=0$ is approximated based on the following expansions of $v$ and $\mathcal{R}$  
\begin{align*}
v(u,\lambda)&=\hat{v}(u,\lambda)+O\left(\sum_{i=1}^5||u||^i_{\sspace}|\lambda|^{5-i}\right)\\
\mathcal{R}(\phi,\lambda)&=\hat{\mathcal{R}}(\phi,\lambda)+O\left(\sum_{i=1}^5||\phi||^i_{\sspace}|\lambda|^{5-i}\right).
\end{align*}
The following calculation shows that $\hat{f}$ agrees with $f$ up to fourth order 
\begin{align*}
f=&\mathcal{R}(u+v(u,\lambda),\lambda)\\=&\hat{R}(u+v(u,\lambda),\lambda)+O\left(\sum_{i=1}^5||u+v(u,\lambda)||^i_{\sspace}|\lambda|^{5-i}\right)\\
=&\hat{\mathcal{R}}(u+\hat{v}(u,\lambda),\lambda)+O\left(\sum_{i=1}^6||u||^i_{\sspace}|\lambda|^{6-i}\right)\\&+O\left(\sum_{i=1}^5||u||^i_{\sspace}|\lambda|^{5-i}\right)\\
f(u,\lambda)=&\hat{f}(u,\lambda)+O\left(\sum_{i=1}^5||\phi||^i_{\sspace}|\lambda|^{5-i}\right).
\end{align*}
\end{rem}
The bifurcation equation can also be written as an expansion in spherical harmonics of degree $l=2$, and therefore admits the form
\begin{align}\label{eqt: complex bifurcation equation}
\hat{f}(u,&\lambda)(p)=\sum_{m=-2}^2\hat{f}_{m}(u,\lambda) Y^{m}_2(p).
\end{align}

The solutions of this equation are equivalent to the solutions of the Euler-Lagrange equation of the Onsager free-energy functional in (\ref{eqt:free energy}). In order to find the zeroes of (\ref{eqt: complex bifurcation equation}), each of the coefficients $\hat{f}_i$ for $-2\leq i\leq 2$ needs to be zero. Therefore we seek the solutions to a system of five polynomials, each depending on the six variables $u_{-2},u_{-1},u_0,u_1,u_2$ and $\lambda$ which reduces the infinite-dimensional state space $H^2(\SP)$ of our problem to the ten-dimensional one i.e. $\mathbb{C}^5$.

However, we are only interested in real solutions to this equation, so we can reduce the dimension further by restricting the equation onto the space of real spherical harmonics with real coefficients. In particular, the real spherical harmonics are defined as
$$Y_{l,m}:=\begin{cases}\frac{i}{\sqrt{2}}(Y^m_l-(-1)^mY^{-m}_l)&\text{ if }m<0\\Y^0_l &\text{ if }m=0\\\frac{1}{\sqrt{2}}(Y^{-m}_l+(-1)^mY^{m}_l)&\text{ if }m>0\end{cases}.$$
Thus, if the function $u(p)$ in \eqref{eqt: definition of u in terms of SH} is real-valued, its  coefficients $$\mathbf{u}:=(u_{-2},u_{-1},u_0,u_1,u_2)$$ can be written as 
\begin{align}\label{eq: mapping to real SH}
\mathbf{u}=T(\mathbf{a}):=\left(
\begin{array}{ccccc}
 \frac{i}{\sqrt{2}} & 0 & 0 & 0 & \frac{1}{\sqrt{2}} \\
 0 & \frac{i}{\sqrt{2}} & 0 & \frac{1}{\sqrt{2}} & 0 \\
 0 & 0 & 1 & 0 & 0 \\
 0 & \frac{i}{\sqrt{2}} & 0 & -\frac{1}{\sqrt{2}} & 0 \\
 -\frac{i}{\sqrt{2}} & 0 & 0 & 0 & \frac{1}{\sqrt{2}} \\
\end{array}
\right)\cdot
\left(\begin{array}{c}
a_{-2}\\
a_{-1}\\
a_0\\
a_1\\
a_2
\end{array}
\right)
\end{align}
where $T:\R^5\rightarrow\mathbb{C}^5$ for some  $\mathbf{a}:=(a_{-2},\dots,a_2)\in \mathbb{R}^5$. Using this transformation, we define the following real version of the bifurcation equation which corresponds to the set of real spherical harmonics $\{Y_{2,m}:-2\leq m\leq 2\}$
\begin{align*}
f_{\text{real}}(\mathbf{a},\lambda):=T^{-1}f(T\mathbf{a},\lambda).
\end{align*}
Equivalently, we will use the expression 
\begin{align}\label{eqt: real bifurcation equation}
\hat{f}_{\text{real}}(\mathbf{a},\lambda):=T^{-1}\hat{f}(T\mathbf{a},\lambda).
\end{align}
in order to denote its approximation up to fourth order. Due to its complicated form, we will not state the bifurcation equation in terms of real spherical harmonics explicitly in this section. Instead it can be found in Appendix \ref{app: real bifurcation equation}. 
\begin{rem}
From now on, we will refrain from referring to \eqref{eqt: real bifurcation equation} as bifurcation equation in terms of real spherical harmonics but we will also simply refer to it as bifurcation equation.
\end{rem}

\section{Reducing the dimension of the state space of the bifurcation equation}\label{sec:dimension reduction using symmetry}
The most common approach to investigate the local structure of  a system of polynomial equations at a given solution point  is to use its Groebner basis. A Groebner basis is the generating set of an ideal in a polynomial ring over a field $K[x_1,x_2,\dots,x_n]$. The advantage of this method is that it can be used in order to find the dimensionality of the solution set at a given point. However, due to the lengthy form of the bifurcation equation in \eqref{eqt: real bifurcation equation}, this method is not directly applicable. Instead, we make use of the symmetries of our problem in order to reduce its dimensionality. In particular we consider the Onsager free-energy functional with a rotationally symmetric intermolecular two-body potential. That means, that if we can find a rotation that translates one solution of the bifurcation equation (\ref{eqt: real bifurcation equation}) into another, we can consider both solutions to be equivalent.\\

In order to represent an arbitrary rotation matrix $R(\varphi_R,\psi_R,\theta_R)$ with respect to the Euler angles $\varphi_R,\psi_R$ and $\theta_R$, we use the so-called $y$-convention. In this convention, a rotation $R(\varphi_R,\psi_R,\theta_R)$ is given by the product
\begin{align*}
R(\varphi_R,&\psi_R,\theta_R)=R_z\cdot R_x\cdot R_y
\end{align*}
where the first rotation is by an angle $\varphi_R$ around the $z$-axis, the second rotation is by an angle $\theta_R$ around the $y$-axis and the third rotation is by an angle $\psi_R$ around the $z$-axis. In particular, the three matrices are given by
\begin{align*}
R_x&=\left(
\begin{array}{ccc}
 1 & 0 & 0 \\
 0 & \cos (\theta_R) & \sin (\theta_R) \\
 0 & -\sin (\theta_R) & \cos (\theta_R) \\
\end{array}
\right),
R_y=\left(
\begin{array}{ccc}
 -\sin (\varphi_R) & \cos (\varphi_R) & 0 \\
 -\cos (\varphi_R) & -\sin (\varphi_R) & 0 \\
 0 & 0 & 1 \\
\end{array}
\right),\\
R_z&=\left(\begin{array}{ccc}
 \sin (\psi_R) & -\cos (\psi_R) & 0 \\
 \cos (\psi_R) & \sin (\psi_R) & 0 \\
 0 & 0 & 1 \\
\end{array}
\right).
\end{align*}
In contrast to the rotation of points in the state space, the rotation of a function $u(p)$ written in terms of complex spherical harmonics of order $l=2$ corresponds to a linear transformation of its coefficients $\mathbf{u}=(u_{-2},\dots, u_2)$. This transformation is given by the $5\times 5$ Wigner matrix $D(\varphi_R,\psi_R,\theta_R)$. In particular, this relationship is formulated as
\begin{align}\label{eqt: introduction Wigner matrix - rotation of coefficients}
\left( \sum_{m=-2}^{2}  u_m Y^m_2\right)(R(\varphi_R, \psi_R,\theta_R) \cdot p)= \left( \sum_{m=-2}^{2}  (D(\varphi_R,\psi_R,\theta_R)u)_m Y^m_2\right)(p)
\end{align}
for all $p \in \SP$ \cite{MarinucciPeccati2011SphericalHarmonics}. If we are given two vectors $\mathbf{a},\mathbf{b}\in \R^5$ each representing the five coefficients of a real-valued solution written in terms of real spherical harmonics, then we say that $\mathbf{a}\sim \mathbf{b}$ if and only if there exists a tuple of rotation angles $(\varphi_R,\psi_R,\theta_R)$ such that
\begin{align}\label{eqt: definition matrix M - real Wigner matrix}
M(\varphi_R,\psi_R,\theta_R)\mathbf{a}:=T^{-1}D(\varphi_R,\psi_R,\theta_R)\cdot T\mathbf{a}=\mathbf{b}
\end{align}
where $T$ denotes the isomorphism given in \eqref{eq: mapping to real SH}. In other words, $M$ is the real version corresponding to the complex Wigner matrix $D(\varphi_R,\psi_R,\theta_R)$ given in \eqref{eqt: introduction Wigner matrix - rotation of coefficients}. 

Having introduced these concepts, we are now in the position to establish symmetry properties of our bifurcation equation.
\begin{prop}\label{prop: equivariance of f}
The bifurcation equation $\hat{f}_{\textup{real}}$ is equivariant with respect to the Wigner matrices $M(\varphi_R,\psi_R,\theta_R)$, that means that for all $M(\varphi_R,\psi_R,\theta_R)$ and every $\mathbf{a}\in \R^5$, 
$$\hat{f}_{\textup{real}}(M(\varphi_R,\psi_R,\theta_R) \mathbf{a},\lambda)=M(\varphi_R,\psi_R,\theta_R) \hat{f}_{\textup{real}}(\mathbf{a},\lambda).$$
\end{prop}
\begin{proof}
Let $R(\varphi_R,\psi_R,\theta_R)$ be the rotation matrix that corresponds to the Wigner matrix $M(\varphi_R,\psi_R,\theta_R)$. For ease of notation, we drop the dependence of $R$ on the Euler angles $\varphi_R,\psi_R$ and $\theta_R$. Defining $R\phi(p):=\phi(Rp)$, we see that
\begin{align*}
E(R\phi,\lambda)=&\lambda \phi(Rp)-\frac{1}{Z(R\phi)}\int_{\mathbb{S}^2} k(p\cdot q)\exp(-\phi(Rq))\;dq\\
=&\lambda \phi(Rp)-\frac{1}{\int_{\SP}\exp(-\phi(Rs))\;ds}\int_{\mathbb{S}^2} k(p\cdot q)\exp(-\phi(Rq))\;dq\\
=&\lambda \phi(Rp)-\frac{1}{\int_{\SP}\exp(-\phi(s))\;ds}\int_{\mathbb{S}^2} k(p\cdot R^tq)\exp(-\phi(q))\;dq\\
=&\lambda \phi(Rp)-\frac{1}{\int_{\SP}\exp(-\phi(s))\;ds}\int_{\mathbb{S}^2} k(Rp\cdot q)\exp(-\phi(q))\;dq\\
=&RE(\phi,\lambda).
\end{align*}
Hence the operator $E$ is equivariant with respect to the Wigner matrices. By definition $$f:=P\mathcal{R}(\mathbf{u}+\mathbf{v}(\mathbf{u},\lambda),\lambda)$$
where $\mathcal{R}$ denotes the non-linear part of $E$ and is therefore also equivariant. The projection $P$ maps all terms onto the spherical harmonics of degree $l=2$ and therefore also preserves the equivariance as well as the approximation of the equation up to fourth order. The restriction to the subspace of real solutions is also a linear transformation and therefore we can deduce that $\hat{f}_{\text{real}}$ is equivariant with respect to the Wigner matrices $M(\varphi_R,\psi_R,\theta_R)$. 
\end{proof}
Knowing that $f$ is in fact equivariant with respect to rotations, we can restrict our attention to one representative of each class of solutions that can be translated into each other by a rotation. In other words we are interested in the orbit space that corresponds to the action of $SO(3)$ on the space $\R^5$. To illustrate this idea, consider the example of $SO(2)$ acting on $\R^2$. The orbit space in this case can then be taken to be the non-negative part of the $x$-axis and is therefore one-dimensional. \\

The aim of this section is to show that the orbit space of $SO(3)$ acting on $\R^5$ can in fact be reduced to two dimensions. In particular, using invariant theory for groups, we can prove that the space 
\begin{align}\label{eqt: reduced state space}
\mathcal{S}:=\{(0,0,x,0,y)| x,y\in \R\}
\end{align}
contains at least one representative of every orbit.

In particular, one can show that the generators of the ring of $G$-invariant polynomials for any compact Lie-group $G$ separate the orbits; thus for any two distinct orbits $\Gamma$ and $\Gamma'$, there exists at least one of the generators  taking different values on $\Gamma$ and $\Gamma'$ \cite[Appendix C]{AbudSartori1982}. In the case of the group action given in \eqref{eqt: definition matrix M - real Wigner matrix} there exist two invariant polynomials $I_1$ and $I_2$ generating the ring. The fact that these generators separate the orbits can be reformulated as follows. If 
\begin{align*}
(I_{1}(\mathbf{a}),I_{2}(\mathbf{a}))=(I_{1}(\mathbf{b}),I_{2}(\mathbf{b})),
\end{align*}
then there exist rotation angles $\varphi_{R},\psi_{R}$ and $\theta_{R}$ such that $M(\varphi_{R},\psi_{R},\theta_{R})\mathbf{a}=\mathbf{b}$ and hence $\mathbf{a}\sim \mathbf{b}$. This allows us to check that $\mathcal{S}$ contains at least one representative for every orbit by verifying that
\begin{align}
\left\{ I_{1}(\mathbf{a}),I_{2}(\mathbf{a})\mid\mathbf{a}\in\mathbb{R}^5\right\} =\left\{ I_{1}(\mathbf{a}),I_{2}(\mathbf{a})\mid\mathbf{a}\in\mathcal{S}\right\}.\label{eqn: reduction condition}
\end{align}
In other words, we show that the state space can in fact be reduced to two dimensions.\\ 

In \cite{Wachsmuth2006} Wachsmuth achieves a somehow similar two-dimensional reduction by the interesting device of considering the set of homogeneous polynomials restricted to the sphere, which is an equivalent description of the set of spherical harmonics.


This section is structured as follows. In Section \ref{sec: generating invariants for the group action} we find an isomorphic representation of the group action, called the Cartan representation, and we derive an explicit way to separate its orbits based on invariant polynomials. Using the isomorphism between the Cartan representation and the group action given in \eqref{eqt: definition matrix M - real Wigner matrix} which is based on the Wigner matrix, see \eqref{eqt: isomorphism between the representations}, we obtain similar expressions for the invariant polynomials with respect to the group action in \eqref{eqt: definition matrix M - real Wigner matrix}. We conclude this section by proving that the reduced state space $\mathcal{S}$ contains a representative of every orbit, see Section \ref{sec: surjectivity of the parametrisation of the invariant polynomials}.

\subsection{The separation of orbits of the group action $SO(3)$ acting on $\R^5$}\label{sec: generating invariants for the group action}
There exists a unique irreducible representation of the group action of $SO(3)$ acting on $\mathbb{R}^5$ \cite{Friedrich2003}. Since the  complex and real representations based on the Wigner matrix given in Appendix \ref{app: complex Wigner matrix} and \ref{app: real Wigner matrix} are irreducible, it is isomorphic to any other irreducible representation of $SO(3)$ acting on the five-dimensional real space. One of these isomorphic representations is called the Cartan representation, see Appendix \ref{app: Cartan representation}. The explicit isomorphism between the two group actions is given in Lemma \ref{lem:isomorphism between the group actions} in Appendix \ref{app: isomorphism between the group representations}. The Cartan representation is given by 
$$\rho:SO(3)\times \mathfrak{su}(3)/\mathfrak{so}(3), \quad (S,X+\mathfrak{so}(3))\rightarrow SXS^{-1}+\mathfrak{so}(3)$$
where $\mathfrak{su}(3)$ and $\mathfrak{so}(3)$ denote the Lie-algebras associated with the Lie groups $SU(3)$ and $SO(3)$, respectively. The representation above then arises as the adjoint representation of $\mathfrak{so}(3)$ \cite{Dieck1985}.
In this setting the matrix $X\in \mathfrak{su}(3)/\mathfrak{so}(3)$ is a traceless skew-Hermitian matrix, that is $X^\star=-X$. In particular, if we write $X=U+i V$, with $U$ and $V$ both being real $3\times 3$ matrices, then we can make the following observation
$$X^\star=U^\star+(i V)^\star=U^t-i V^t.$$
Thus, in order for $X$ to be skew-Hermitian, $U$ needs to be a traceless skew-symmetric matrix, thus $U\in \mathfrak{so}(3)$, and $V$ needs to be traceless and symmetric. It follows that
\begin{align*}
\mathfrak{su}(3)/\mathfrak{so}(3)&=\{X+\mathfrak{so}(3)|X\in \mathfrak{su}(3)\}\\&=\{iV+\mathfrak{so}(3)| \text{tr}(V)=0 \text{ and }V^t=V\}.
\end{align*}
Let us choose the following basis for this space
\begin{align}\label{eqt: basis for traceless symmetric matrices}
\begin{aligned}
E=\Bigg\{e_1&=\left(\begin{array}{ccc}
 1 & 0 & 0 \\
 0 & 0 & 0 \\
 0 & 0 & -1 \\
\end{array}\right),
 e_2=\left(
\begin{array}{ccc}
 0 & 0 & 0 \\
 0 & 1 & 0 \\
 0 & 0 & -1 \\
\end{array}
\right),
e_3=\left(
\begin{array}{ccc}
 0 & 1 & 0 \\
 1 & 0 & 0 \\
 0 & 0 & 0 \\
\end{array}
\right),\\
&\hspace{2.7cm}e_4=\left(
\begin{array}{ccc}
 0 & 0 & 1 \\
 0 & 0 & 0 \\
 1 & 0 & 0 \\
\end{array}
\right),
e_5=\left(
\begin{array}{ccc}
 0 & 0 & 0 \\
 0 & 0 & 1 \\
 0 & 1 & 0 \\
\end{array}
\right) \Bigg\}.
\end{aligned}
\end{align}
Thus, any element in $\mathfrak{su}(3)/\mathfrak{so}(3)$ can be written as a linear combination of these matrices, that is for any $X\in \mathfrak{su}(3)/\mathfrak{so}(3)$ there exist $(x_1,\dots, x_5)$ such that
$$X=\sum_{n=1}^5x_5 e_n.$$
In particular, we say that any two elements $X$ and $Y\in \mathfrak{su}(3)/\mathfrak{so}(3)$ are equivalent, and thus are elements of the same orbit, if there exists a matrix $S(\varphi_R,\psi_R,\theta_R) \in SO(3)$ such that
\begin{align*}
S(\varphi_R,\psi_R,\theta_R)\left(\sum_{n=1}^5x_ne_n\right)S^{-1}(\varphi_R,\psi_R,\theta_R)=\sum_{n=1}^5y_ne_n
\end{align*}
where $\mathbf{x}:=(x_1,\dots,x_5)$ and $\mathbf{y}:=(y_1,\dots,y_5)$ correspond to $X$ and $Y$, respectively. 
Similarly, this relationship can also be represented by a $5\times5$ matrix mapping $\mathbf{x}$ onto $\mathbf{y}$. In particular, $\mathbf{x}\sim \mathbf{y}$ if and only if there exist $\varphi_R,\psi_R$ and $\theta_R$ such that 
$$M_C(\varphi_R,\psi_R,\theta_R)\mathbf{x}=\mathbf{y}$$
where the particular form of $M_C$ is given in Appendix \ref{app: isomorphism between the group representations}.
Having established an equivalence relation that describes the orbits in case of the Cartan representation, we are now in a position to describe its orbit space.

\begin{lem}\label{lem: separation of orbits}
Let $X$ and $Y$ be two real symmetric matrices. Then there exists a matrix $S\in SO(3)$ such that $SXS^{-1}=Y$, if and only if the characteristic polynomials of $X$ and $Y$ agree, that is $\chi(X)=\chi(Y)$.
\end{lem}
\begin{proof} Omitted.
\end{proof}

Based on Lemma \ref{lem: separation of orbits}, we deduce that the orbits described by the Cartan representation are separated by the coefficients of the characteristic polynomials of the corresponding matrix. In particular, the characteristic polynomial $\chi(X)$ of a matrix $X\in\mathfrak{su}(3)/\mathfrak{so}(3)$ written in terms of the basis $E$ in \eqref{eqt: basis for traceless symmetric matrices} is given by
\begin{align*}
\begin{aligned}
\chi(X)(k)=&-k^3+k \left(x_1^2+x_1x_2+x_2^2+x_3^2+x_4^2+x_5^2\right)\\&-x_1^2 x_2-x_1 x_2^2+x_1x_3^2-x_1x_5^2+x_2x_3^2-x_2x_4^2+2x_3x_4x_5.
\end{aligned}
\end{align*}
Thus, the two invariant polynomials for the group action written in terms of the Cartan representation are
\begin{align*}
I_{C1}:=&x_1^2+x_1x_2+x_2^2+x_3^2+x_4^2+x_5^2 \text{ and }\\
I_{C2}:=&-x_1^2 x_2-x_1 x_2^2+x_1x_3^2-x_1x_5^2+x_2x_3^2-x_2x_4^2+2x_3x_4x_5.
\end{align*}
In particular, both of these polynomials satisfy the relationship
$$I_{Cj}(M_C(\varphi_R,\psi_R,\theta_R)\mathbf{x})=I_{Cj}(\mathbf{x})$$
for $j=1,2$, respectively, and for all $(\varphi_R,\psi_R,\theta_R)\in [0,\pi)^2\times [0,2\pi)$.
Using the isomorphism given in \eqref{eqt: isomorphism between the representations}, the invariant polynomials that correspond to the group representation based on the Wigner rotation matrix in \ref{app: real Wigner matrix} are given by
\begin{align*}
I_{1} &= x_1^2+x_2^2+x_3^2+x_4^2+x_5^2 \text{ and }\\
I_{2} &= -\frac{2 x_1^2 x_3}{\sqrt{3}}+2 x_1 x_2 x_4+\frac{x_2^2 x_3}{\sqrt{3}}-x_2^2 x_5+\frac{2 x_3^3}{3 \sqrt{3}}+\frac{x_3 x_4^2}{\sqrt{3}}-\frac{2 x_3 x_5^2}{\sqrt{3}}+x_4^2 x_5.
\end{align*}

In order to follow the programme outlined at the beginning of this section, it would now only be left to show that $I_{1}$ and  $I_{2}$ are generators of the ring of invariant polynomials. However, in this case, we already know that both polynomials separate the orbits by using the isomorphism and Lemma \ref{lem: separation of orbits}.

\subsection{Reduction to a two-dimensional orbit space}\label{sec: surjectivity of the parametrisation of the invariant polynomials}
Having found the fundamental invariants corresponding to the group action of $SO(3)$ acting on $\mathbb{R}^5$, we have an explicit description of all orbits as each orbit corresponds to exactly one constant value in the image of the two invariants. In order to reduce the state space to $\mathcal{S}$, we need to verify that
\begin{align}
\left\{ I_{1}(\mathbf{a}),I_{2}(\mathbf{a})\mid\mathbf{a}\in\mathbb{R}^5\right\} =\left\{ I_{1}(\mathbf{a}),I_{2}(\mathbf{a})\mid\mathbf{a}\in\mathcal{S}\right\}.\tag{\ref{eqn: reduction condition}}
\end{align}
In particular, we can express the left hand side and the right hand side as 
\begin{align*}
\left\{ I_{1}(\mathbf{a}),I_{2}(\mathbf{a})\mid\mathbf{a}\in\mathbb{R}^5\right\}&=\bigcup_{r=0}^\infty \left\{ \{r,I_2(\mathbf{a})\} \mid \mathbf{a}\in\mathbb{R}^5\text{ and } I_1(\mathbf{a})=r \right\}\\
&=\bigcup_{r=0}^\infty \{r\}\times \left[\min_{\substack{I_1(\mathbf{a})=r\\\mathbf{a}\in\R^5}}I_2(\mathbf{a}),\max_{\substack{I_1(\mathbf{a})=r\\\mathbf{a}\in\R^5}}I_2(\mathbf{a})\right]
\end{align*}
and 
\begin{align*}
\left\{ I_{1}(\mathbf{a}),I_{2}(\mathbf{a})\mid\mathbf{a}\in \mathcal{S}\right\}&=\bigcup_{r=0}^\infty \left\{ \{r,I_2(\mathbf{a})\} \mid \mathbf{a}\in\mathcal{S} \text{ and }  I_1(\mathbf{a})=r \right\}\\
&=\bigcup_{r=0}^\infty \{r\}\times \left[\min_{\substack{I_1(\mathbf{a})=r\\\mathbf{a}\in\mathcal{S}}}I_2(\mathbf{a}),\max_{\substack{I_1(\mathbf{a})=r\\\mathbf{a}\in\mathcal{S}}}I_2(\mathbf{a})\right],
\end{align*}
respectively. Moreover, we observe that $\{I_1(\mathbf{a})=r\}$ is in fact a five-dimensional sphere which is connected and compact. Rewriting the problem like this, we observe that we are interested in the extrema of $I_2$ for points on the sphere with radius $r$. Since the sphere is compact, we know that these extrema are attained and and an application of the intermediate value theorem gives us that $I_2$ attains in fact all values in between its extrema. Taking these arguments into account, the verification of \eqref{eqn: reduction condition} has thus been reduced to checking that the extrema defining the intervals agree for all $r>0$. Because both invariant polynomials $I_1$ and $I_2$ are homogeneous, it actually suffices to verify this condition for the case $r=1$. Moreover, since $-I_2(\mathbf{a})=I_2(-\mathbf{a})$ and $\mathbf{a}\in \SP$ if and only if $-\mathbf{a}\in\SP$, it is enough to prove the equality of the maxima. Hence, we only need to show that the following relationship holds
\begin{align*}
\max_{\mathbf{a}\in\mathbb{R}^5, I_1(\mathbf{a})=1}I_2(\mathbf{a})=
\max_{\mathbf{a}\in\mathcal{S},I_1(\mathbf{a})=1} I_2(\mathbf{a}).
\end{align*}
A direct calculation shows that in both cases 
$$\max_{\mathbf{a}\in \R^5}I_2(\mathbf{a})=\max_{\mathbf{a}\in \mathcal{S}}I_2(\mathbf{a})=\frac{2}{3 \sqrt{3}}$$
if $I_1(\mathbf{a})=1$, which thus proves the claim, and we conclude that the state space can in fact be reduced to the two-dimensional space $\mathcal{S}$.

\section{Solving the reduced bifurcation equation}\label{sec: solving the bifurcation equation}
Having shown that the equivariant state space of the bifurcation equation can in fact be restricted to the two-dimensional space $\mathcal{S}$ in (\ref{eqt: reduced state space}), we are now in the position to solve this reduced problem. We denote the reduced form of the real bifurcation equation by $\hat{f}_{\text{real}}^\mathcal{S}$ which only depends on the two state variables $a_0$ and $a_2$. Its approximation up to fourth order is given by 
\begin{align*}
\hat{f}_{{\text{real}}_{-2}}^\mathcal{S}(\mathbf{a},\lambda)(p)=&0,\\
\hat{f}_{{\text{real}}_{-1}}^\mathcal{S}(\mathbf{a},\lambda)(p)=&0,\\
\hat{f}_{{\text{real}}_0}^\mathcal{S}(\mathbf{a},\lambda)(p)=&\lambda a_0+\frac{18(\lambda a_0^3+\lambda a_2^2 a_0)}{49 \pi  (1+32 \pi )^2}+\frac{\sqrt{5}(448 \pi  (13+2800 \pi )+577) a_0^4}{\sqrt{\pi }5795328 (1+32 \pi )^2}\\
&+\left(\frac{9}{3136 (1+32 \pi )}-\frac{5}{1792}\right) (a_0^3+a_2^2 a_0)+\frac{\sqrt{5 \pi }}{448}(a_0^2-a_2^2)\\
&-\frac{5 \sqrt{\frac{5}{\pi }} (224 \pi  (5+224 \pi )+31) a_2^2 a_0^2}{965888 (1+32 \pi )^2}\\
&+\frac{\sqrt{\frac{5}{\pi }} (1792 (1-140 \pi ) \pi -89) a_2^4}{1931776 (1+32 \pi )^2},\\
\hat{f}_{{\text{real}}_{1}}^\mathcal{S}(\mathbf{a},\lambda)(p)=&0,
\end{align*}
\begin{align}\label{eqt: reduced bifurcation equation}
\begin{aligned}
\hat{f}_{{\text{real}}_2}^\mathcal{S}(\mathbf{a},\lambda)(p)=&\lambda a_2+\frac{18 (\lambda a_2 a_0^2+\lambda a_2^3)}{49 \pi  (1+32 \pi )^2}+\frac{\sqrt{\frac{5}{\pi }} (14 (7-320 \pi ) \pi -1) a_2 a_0^3}{25872 (1+32 \pi )^2}\\&+\left(\frac{9}{3136 (1+32 \pi )}-\frac{5}{1792}\right) (a_2 a_0^2+ a_2^3)-\frac{\sqrt{5 \pi }}{224} a_2 a_0\\&-\frac{\sqrt{\frac{5}{\pi }} (56 \pi  (17+2240 \pi )+61) a_2^3 a_0}{241472 (1+32 \pi )^2}.
\end{aligned}
\end{align}
In particular, we observe that this equation can also be rewritten up to third order in $a_i$ for $i\in \{0,2\}$ and $\lambda$ as
\begin{align}\label{eqt: reduced bifurcation eqt in invariant form}
\hat{f}_{\text{real}}^\mathcal{S}(\mathbf{a},\lambda)=(\lambda+d(a_0^2+a_2^2))\begin{pmatrix} a_0\\a_2 \end{pmatrix}+c\begin{pmatrix}
a_0^2-a_2^2\\-2a_0a_2 \end{pmatrix}
\end{align}
where $c:=\frac{\sqrt{5\pi}}{448}$ and $d:=\left(\frac{9}{3136 (1+32 \pi )}-\frac{5}{1792}\right)$.
However, we need to find conditions ensuring that we consider the equation up to a sufficiently high order, in order to guarantee that the characteristics of the bifurcation do not change. Such a problem is called recognition problem and it will be presented in the context of the two-dimensional  bifurcation equation in \eqref{eqt: reduced bifurcation eqt in invariant form} in the following two sections. First we establish some theory in Section \ref{sec: recognition problem} before this theory will be applied to our particular problem in Section \ref{sec: existence transcritical bifurcation}.


\subsection{The recognition problem in two dimensions}\label{sec: recognition problem}
The aim of this section is to derive non-degeneracy conditions which ensure a one-to-one correspondence between the solutions of the two-dimensional reduced bifurcation equation and an algebraic equation of a simple form of lower order. This problem is called recognition problem. In our presentation we follow \cite[Chapter 9]{Fu2001}. \\

A one-to-one correspondence between the sets of solutions of two functions $f(\mathbf{u},\lambda)=0$ and $g(\mathbf{u},\lambda)=0$ for $\mathbf{u}\in \R^n$ and $\lambda \in \R$ exists if $f$ and $g$ are strongly equivalent.
\begin{defi}
Two $C^\infty$-functions $f$ and $g$ are strongly equivalent, denoted by $f\sim g$ if there exist $S:\R^n\times \R^m\rightarrow L(\R^n,\R^n)$ and $U:\R^n\times \R^m\rightarrow \R^n$ such that
\begin{align*}
\det S(\mathbf{u},\lambda)&>0,\\
U(0,0)&=0,\frac{\partial U}{\partial \mathbf{u}}(\mathbf{u},\lambda) \text{ is positive definite}\\
f(\mathbf{u},\lambda)&=S(\mathbf{u},\lambda)g(U(\mathbf{u},\lambda),\lambda).
\end{align*}
\end{defi}
However, often it is not straightforward to verify strong equivalence for particular cases. Instead, it is easier if the problem inherits certain symmetry properties that can be taken into account. In case of the reduced bifurcation equation in \eqref{eqt: reduced bifurcation equation}, the problem is symmetric with respect to the group action of $S_3$ acting on the space $\R^2$. In particular, $S_3$ acts on any element $\mathbf{u}\in \R^2$ by
$$g_i\mathbf{u}=\mathbf{u}' \text{ for all }i\in\{1,\dots,6\},$$
where $g_i$ are defined as
\begin{align*}
S_3=&\{g_i\in \R^{2\times 2},i=1,\dots,6\}
\\=&\Big\{
\begin{pmatrix}
1&0\\0&1
\end{pmatrix},
\begin{pmatrix}
1&0\\0&-1
\end{pmatrix},
\frac{1}{2}\begin{pmatrix}
-1&\sqrt{3}\\-\sqrt{3}&-1
\end{pmatrix},
\frac{1}{2}\begin{pmatrix}
-1&-\sqrt{3}\\\sqrt{3}&-1
\end{pmatrix},\\
&\quad\frac{1}{2}\begin{pmatrix}
-1&\sqrt{3}\\\sqrt{3}&1
\end{pmatrix},\frac{1}{2}\begin{pmatrix}
-1&-\sqrt{3}\\-\sqrt{3}&1
\end{pmatrix}\Big\}.
\end{align*}
Moreover, it is easy to show that the group action is generated by the elements
\begin{align}\label{eqt: generators S3}
g_2=\begin{pmatrix}
1&0\\0&-1
\end{pmatrix} \text{ and }g_4=\frac{1}{2}\begin{pmatrix}
-1&-\sqrt{3}\\\sqrt{3}&-1
\end{pmatrix}.
\end{align}
%
%
%
Based on this representation of the group action of $S_3$ acting on $\R^2$, we direct our attention to the concept of $S_3$-invariance, $S_3$-equivariance and $S_3$-equivariant equivalence.
\begin{defi}
A mapping $h:\R^2\rightarrow\R$ is $S_3$-invariant if $$h(g\mathbf{u})=h(\mathbf{u})$$ for all $g\in S_3$ and $\mathbf{u}\in \R^2$.
\end{defi}

\begin{defi}
A mapping $f:\R^2\times \R\rightarrow \R^2$ is $S_3$-equivariant if
$$f(g\mathbf{u},\lambda)=gf(\mathbf{u},\lambda)\text{ for all }g\in S_3\text{ and }\mathbf{u}\in\R^2.$$
\end{defi}

\begin{defi}
Two $S_3$-equivariant $C^\infty$-mappings $f:\R^2\times \R\rightarrow \R^2$ and $g:\R^2\times \R\rightarrow \R^2$ are equivariantly equivalent, denoted by $f\sim g$, if there exist $S:\R^2\times \R^2\rightarrow L(\R^2,\R^2)$ and $U:\R^2\times \R^2\rightarrow \R^2$ such that
\begin{align*}
\det S(\mathbf{u},\lambda)&>0,\\
U(0,0)&=0,\frac{\partial U}{\partial \mathbf{u}}(\mathbf{u},\lambda) \text{ is positive definite}\\
S(g\mathbf{u},\lambda)&=gS(\mathbf{u},\lambda)g^{-1},U(g\mathbf{u},\lambda) \text{ for all }g\in S_3,\mathbf{u}\in \R^2,\lambda\in \R\\
f(\mathbf{u},\lambda)&=S(\mathbf{u},\lambda)g(U(\mathbf{u},\lambda),\lambda).
\end{align*}
\end{defi}
The aim of Section \ref{sec: existence transcritical bifurcation} will be to show that the reduced bifurcation equation is in fact $S_3$-equivariant and equivariantly equivalent to its approximation up to fourth order. In order to verify the equivariant equivalence, we use the following concepts and results.
\begin{citedthm}[{{\cite[Theorem 1]{Schwarz1975}}}]
Any smooth $S_3$-invariant mapping $h$ can be represented as
$$h(\mathbf{u})=\hat{h}(h_1(\mathbf{u}),h_2(\mathbf{u}))$$
where $h_1(\mathbf{u}):=u_1^2+u_2^2$ and $h_2(\mathbf{u}):=u_1^3-3u_1u_2^2$.
\end{citedthm}

\begin{thm}\label{thm: representation of S3 invariant functions}
Any $S_3$-equivariant mapping $f:\R^2\times \R\rightarrow\R^2$ can be represented by
$$f(\mathbf{u},\lambda)=a(h_1(\mathbf{u}),h_2(\mathbf{u}),\lambda)\begin{pmatrix} u_1 \\u_2 \end{pmatrix}+b(h_1(\mathbf{u}),h_2(\mathbf{u}),\lambda)\begin{pmatrix} u_1^2-u_2^2 \\-2u_1u_2 \end{pmatrix}$$
with $h_1(\mathbf{u})=u_1^2+u_2^2$ and $h_2(\mathbf{u})=u_1^3-3u_1u_2^2$.
\end{thm}

Based on these results, we are now in the position to state the following proposition which yields the existence of a normal form under appropriate conditions.
\begin{prop}\label{prop: normal form}
Let $f:\R^2\times \R\rightarrow\R^2$ be of the form given in Theorem \textup{\ref{thm: representation of S3 invariant functions}}. Then
$$f(\mathbf{u},\lambda)\sim \epsilon \lambda \begin{pmatrix}u_1\\u_2\end{pmatrix}+\begin{pmatrix}u_1^2-u_2^2\\-2u_1u_2\end{pmatrix}$$
for $\epsilon:=\mathrm{sgn }(\frac{\partial a}{\partial \lambda})$ if and only if
$$a(0,0,0)=0,b(0,0,0)\neq 0\text{ and }\frac{\partial a}{\partial \lambda}(0,0,0)\neq 0.$$
\end{prop}

\subsection{Existence of a transcritical bifurcation of the Onsager free-energy functional up to equivariance}\label{sec: existence transcritical bifurcation}
The aim of this section is to find non-degeneracy conditions that allow us to reduce the problem of solving the bifurcation equation in \eqref{eqt: real bifurcation equation} to an algebraic equation of a simple form and lower order. In particular, we show that $\hat{f}_{\text{real}}$ is $S_3$-equivariantly equivalent to a mapping $G:\R^2\times \R\rightarrow \R^2$ by using the symmetry properties of $f_{\text{real}}$.\\

In order to prove that $f_{\text{real}}$ is $S_3$-equivariant, it is sufficient to show that it is equivariant with respect to the generators $g_2$ and $g_4$ of the group action, see \eqref{eqt: generators S3}. This in turn is a consequence of Proposition \ref{prop: equivariance of f}. Notice that for two particular choices of Euler angles we have the following two Wigner matrices \vspace{0.2cm}

$\text{and }\begin{array}{rclcc}
R_{1} & = & R\left(\frac{\pi}{2},\pi,\pi\right) & = & \left(\begin{array}{ccccc}
1 & 0 & 0 & 0 & 0\\
0 & 0 & 0 & -1 & 0\\
0 & 0 & 1 & 0 & 0\\
0 & -1 & 0 & 0 & 0\\
0 & 0 & 0 & 0 & -1
\end{array}\right)\\
\\
R_{2} & = & R\left(\frac{\pi}{2},\pi,-\frac{\pi}{2}\right) & = & \left(\begin{array}{ccccc}
0 & 0 & 0 & -1 & 0\\
-1 & 0 & 0 & 0 & 0\\
0 & 0 & -\frac{1}{2} & 0 & -\frac{\sqrt{3}}{2}\\
0 & 1 & 0 & 0 & 0\\
0 & 0 & \frac{\sqrt{3}}{2} & 0 & -\frac{1}{2}
\end{array}\right).
\end{array}$\vspace{0.2cm}\\
\noindent Both linear maps leave the subspace $\mathcal{S}=\left\{(0,0,x,0,y)\vert x,y\in\mathbb{R}\right\}$ invariant. Merely viewed on the space $\mathcal{S}$, they correspond to the two elements that generate $S_{3}.$ Hence we conclude that $f_{\text{real}}$ is $S_3$-equivariant and we know from Theorem \ref{thm: representation of S3 invariant functions} that it can therefore be written as
$$f_{\text{real}}(\mathbf{u},\lambda)=a(h_1(\mathbf{u}),h_2(\mathbf{u}),\lambda)\begin{pmatrix} u_1 \\u_2 \end{pmatrix}+b(h_1(\mathbf{u}),h_2(\mathbf{u}),\lambda)\begin{pmatrix} u_1^2-u_2^2 \\-2u_1u_2 \end{pmatrix}.$$
Moreover, we can further deduce that it reduces to the normal form 
$$\epsilon \lambda \begin{pmatrix}u_1\\u_2\end{pmatrix}+\begin{pmatrix}u_1^2-u_2^2\\-2u_1u_2\end{pmatrix}$$
with $\epsilon:=\text{sgn }(\frac{\partial a}{\partial \lambda})$ if and only if
$$a(0,0,0)=0,b(0,0,0)\neq 0\text{ and }\frac{\partial a}{\partial \lambda}(0,0,0)\neq 0.$$
Since it is not trivial to find the explicit form of the coefficients $a(\cdot,\cdot)$ and $b(\cdot, \cdot)$, it is not straightforward to verify that these conditions hold in case of the bifurcation equation in \eqref{eqt: real bifurcation equation}.

However, it is sufficient to show that they hold in case of the reduced bifurcation equation $\hat{f}^\mathcal{S}_{\text{real}}$ in \eqref{eqt: reduced bifurcation eqt in invariant form} because the first two derivatives agree. In this case the coefficients are given by
\begin{align*}
\hat{a}&(h_1(\mathbf{a}),h_2(\mathbf{a}),\lambda)=\lambda+\left(\frac{9}{3136 (1+32 \pi )}-\frac{5}{1792}\right)(a_0^2+a_2^2)\\
\hat{b}&(h_1(\mathbf{a}),h_2(\mathbf{a}),\lambda)=\frac{\sqrt{5\pi}}{448}.
\end{align*}
Recall that Remark \ref{rem:order} shows that $\hat{f}_{\text{real}}$ is a fourth order approximation to $f_{\text{real}}$. In $\hat{f}_{\text{real}}^\mathcal{S}$ we have dropped the fourth order terms so that it agrees with $f_\mathcal{S}$ up to third order and therefore also verifies the non-degeneracy condition.

Thus, we see that all three non-degeneracy conditions in Proposition \ref{prop: normal form} hold and we can conclude that there exists a one-to-one correspondence between the set of solutions of $f_{\text{real}}^\mathcal{S}$ in \eqref{eqt: reduced bifurcation eqt in invariant form} and $$G(\mathbf{a},\lambda)=\lambda \begin{pmatrix}a_0\\a_2\end{pmatrix}+\begin{pmatrix}a_0^2-a_2^2\\-2a_0a_2\end{pmatrix}.$$
The set of solutions of $G$ can easily be computed. It is given by the trivial solution 
$\begin{pmatrix}a_0\\a_2 \end{pmatrix}=\begin{pmatrix}0\\0\end{pmatrix}$ and the three branches
$$\lambda\begin{pmatrix}-1\\0 \end{pmatrix},\frac{\lambda}{2}\begin{pmatrix}1\\\sqrt{3} \end{pmatrix} \text{ and }\frac{\lambda}{2}\begin{pmatrix}1\\-\sqrt{3} \end{pmatrix}.$$
However, looking at the rotational symmetries of these three branches, we see that they in fact coincide after an application of a rotation. Thus, the solutions corresponding to all three branches are equivalent which proves the existence of a simple linear branch and thus the existence of a unique transcritical bifurcation up to rotation.

\section{Uniaxial Solutions}\label{sec: uniaxial solutions}
Recall that the Euler-Lagrange operator corresponding to the Onsager free-energy functional is given by
\begin{align}
E(\phi,\lambda):=\lambda \phi(p)-\frac{1}{Z(\phi)}\int_{\mathbb{S}^2} k(p\cdot q)\exp(-\phi(q))\;dq.\tag{\ref{eqt:Euler Lagrange operator}}
\end{align}
Now, we restrict our attention to the uniaxial solutions of this problem, that means that we assume that any solution $\rho_s:L^1(\SP)$ is axially symmetric with respect to the $z$-axis. Thus, writing $\rho$ in terms of spherical coordinates $\theta\in [0,\pi)$ and $\varphi\in [0,2\pi)$ (see Appendix \ref{app: notation spherical harmonics} for details), we assume that it  only depends on $\theta$ and is independent of $\varphi$. In this case, where
$$\int_{\SP} dp=\int_0^{2\pi}d\varphi \int_0^\pi \sin \theta \;d\theta,$$
the Onsager free-energy functional in \eqref{eqt:free energy} restricted to the set of axially symmetric probability densities can be rewritten as
\begin{align*}
\mathcal{F}_s(\rho_s):=& k_B\tau\int_{0}^\pi\rho_s(\theta_p)\ln(\rho_s(\theta_p))\sin \theta_p\; d\theta_p\\
&+\frac{1}{2}\int_{A^2}K(\theta_p,\varphi_p,\theta_q,\varphi_q)\rho_s(\theta_q)\rho_s(\theta_p)\sin \theta_p \sin \theta_q\;d\theta_p d\theta_qd\varphi_pd\varphi_q
\end{align*}
for $A:=[0,\pi]\times [0,2\pi]$.
The Euler-Lagrange equation corresponding to this restricted problem does not differ from the one of the full problem. The reason is that the Euler-Lagrange equation is derived by computing
\begin{align*}
\frac{d}{d\epsilon}\mathcal{F}_s&(\rho_s+\epsilon z_s,\lambda)\Big |_{\epsilon=0}\\=&\int_{0}^\pi z_s(\theta_p)\Bigg [\lambda \ln \rho_s(\theta_p)+\int_{A}K(\theta_p,\varphi_p,\theta_q,\varphi_q)\rho_s(\theta_q)\sin \theta_q\;d\theta_q d\varphi_q\Bigg ]\;d\theta_p
\end{align*}
where $z_s:[0,\pi)\rightarrow \R$ such that $\int_0^\pi z_s(\theta_p)\sin \theta_p\;d\theta_p=0$. In particular, the term within the square brackets does not depend on the variable $\varphi_p$, for details see Lemma \ref{lem: independence of varphi p} in Appendix \ref{app: uniaxial solutions}. This allows us to apply the fundamental lemma of the calculus of variations with respect to the variable $\theta_p$ and gives us the Euler-Lagrange equation
\begin{align*}
\lambda \ln \rho_s(\theta_p)+\int_{A}K(\theta_p,\varphi_p,\theta_q,\varphi_q)\rho_s(\theta_q)\sin \theta_q\; d\theta_q d\varphi_q=-\lambda \ln Z_s
\end{align*}
where $\lambda:=k_B\tau$ and
\begin{align*}
Z_s=\int_{A}\exp\left(-\frac{1}{\lambda}\int_{A}^{2\pi}K(\theta_p,\varphi_p,\theta_q,\varphi_q)\rho_s(\theta_q)\sin \theta_q\; d\theta_q d\varphi_q\right)\sin \theta_p\;d\theta_p d\varphi_p
\end{align*}
as before. Introducing the thermodynamic potential $\phi_s:\SP\rightarrow\R$ as
$$\phi_s(\theta_p):=\frac{1}{\lambda}\int_{A}K(\theta_p,\varphi_p,\theta_q,\varphi_q)\rho_s(\theta_q)\sin \theta_q\;d\theta_q d\varphi_q,$$
it follows that $\rho_s(\theta_p)=Z_s^{-1}\exp(-\phi_s(\theta_p))$ and we can rewrite the Euler-Lagrange equation as
\begin{align*}
\lambda \phi_s(\theta_p)-\frac{1}{Z_s}\int_{\mathbb{S}^2} K(\theta_p,\varphi_p,\theta_q,\varphi_q)\exp(-\phi_s(\theta_q))\; d\theta_q d\varphi_q=0.
\end{align*}
Thus, we conclude that the Euler-Lagrange operator for uniaxial probability distributions is equivalently given by 
\begin{align}\label{eqt: symmetric Euler Lagrange operator}
E_s(\phi,\lambda)=\lambda \phi_s(\theta_p)-\frac{1}{Z}\int_{\mathbb{S}^2} K(\theta_p,\varphi_p,\theta_q,\varphi_q)\exp(-\phi_s(\theta_q))\; d\theta_q d\varphi_q.
\end{align}
A well-known concept that establishes a relationship between the symmetric solutions of a minimisation problem and the symmetrised problem is the principle of symmetric criticality \cite{Palais1979}. It states that each critical point of the symmetric problem is a symmetric critical point of the general problem. In the above setting, this principle can be verified explicitly. Since the Euler-Lagrange operators $E_s$ in \eqref{eqt: symmetric Euler Lagrange operator} and $E$ in \eqref{eqt:Euler Lagrange operator} have exactly the same form, we can consider the corresponding Lyapunov-Schmidt decompositions for the general and the symmetric problem simultaneously. In particular, we split $\phi=u+v$ and $\phi_s=u_s+v_s$, respectively. The first step consists in solving 
\begin{align*}
(1-P)E(u+v,\lambda)&=0 \text{ and } \\ 
(1-P)E_s(u_s+v_s,\lambda)&=0
\end{align*}
for $v$ in terms of $u$ and $v_s$ in terms of $u_s$, respectively. Taking the derivative with respect to $v$ and $v_s$, respectively, we see that we can use the implicit function theorem in order to solve uniquely for $v_s$ and $v$. Because $E(u_s+v_s,\lambda)=E_s(u_s+v_s,\lambda)$, it follows that 
\[ v(u_s)=v_s(u_s)\]
is axially symmetric. Therefore, solving the bifurcation equation of the full problem with $a_{-2}=a_{-1}=a_1=a_2=0$ yields all uniaxial solutions. In particular, this equation is given by
\begin{align*}\begin{aligned}
f_s(a_0,\lambda)=&\frac{18 \lambda a_0^3}{49 \pi  (1+32 \pi )^2}+\lambda a_0+\frac{\sqrt{\frac{5}{\pi }} (448 \pi  (13+2800 \pi )+577) a_0^4}{5795328 (1+32 \pi )^2}\\&+\left(\frac{9}{3136 (1+32 \pi )}-\frac{5}{1792}\right) a_0^3+\frac{1}{448} \sqrt{5 \pi } a_0^2.
\end{aligned}
\end{align*}

This is now a bifurcation problem with a one-dimensional state variable and we say that it undergoes a transcritical bifurcation at $(0,0)$ if $(0,0)$ is a non-hyperbolic fixed point, that is
$$f(0,0)=0 \text{ and }\frac{\partial f}{\partial a_0}(0,0)=0,$$
and if the non-degeneracy conditions
\begin{align*}
\frac{\partial f}{\partial \lambda}(0,0)=0, \quad 
\frac{\partial^2 f}{\partial a_0 \partial \lambda}(0,0)\neq0 \quad \text{ and }\quad
\frac{\partial^2 f}{\partial a_0^2}(0,0)\neq0
\end{align*}
hold (these conditions can be found in the literature, see \cite[(3.1.65)-(3.1.68)]{Wiggins1990}). It is easy to see that all of these conditions hold for $f_s(a_0,\lambda)$ and we may therefore deduce that a transcritical bifurcation occurs. Since we do have a unique transcritical solution in case of the full problem and the symmetric problem, we conclude that these are the same. Therefore we proved the following statement.
\begin{thm}\label{thm:uniaxiality}
All solutions of the Onsager free-energy functional are uniaxial in a neighbourhood of the trivial solution $\phi=0$ and locally around  $\lambda=\lambda_2$.
\end{thm}



\section{Qualitative behaviour of the global bifurcation diagram}\label{sec: Qualitative behaviour of the global bifurcation diagram}
Finally, we complete our bifurcation analysis of the Onsager model by proving some properties of the global bifurcation diagram. In Section \ref{sec: an upper bound on all probability densities} we show that any minimiser of the free-energy functional in (\ref{eqt:free energy}) is bounded and that we can therefore restrict our attention to the set of bounded probability densities. Using this result, we show that the free-energy functional is in fact strictly convex for high temperatures and that its trivial solution, $\rho_0=\frac{1}{4\pi}$ is the unique global solution. In Section \ref{sec: stability} we prove that the trivial solution is a local minimiser for high temperatures and that it is not a local minimiser for low temperatures. Finally, we show in Section \ref{sec:rabinowitz continuous bifurcation branches} that any bifurcation branch either meets infinity or that it meets another bifurcation branch, thus proving that all bifurcation branches are continuous and do not end suddenly.

\subsection{An upper bound on all admissible probability density functions}\label{sec: an upper bound on all probability densities}
By constructing a test function $\rho^\star$ that is bounded above by a constant $C^\star$, we prove that $\mathcal{F}(\rho,\lambda)>\mathcal{F}(\rho^\star,\lambda)$ for all admissible probability densities $\rho$ and thus we show that we can restrict our attention to probability densities that are bounded above. This result holds for all interaction kernels that are continuous and symmetric, see Assumption \ref{assumption} (a)-(b).

\begin{lem}\label{lem:boundedness lemma}
Let $\rho\in L^1(\mathbb{S}^2)$ be a probability density and suppose that there exists a set $A\subset \mathbb{S}^2$ of positive measure such that $\rho(p)\geq C^\star:=\exp(16 M/\lambda)$ for all $p\in A$ where $M:=\max_{p,q\in \mathbb{S}^2}K(p,q)$. Then there exists a modification of $\rho$ denoted by $\rho^\star$, such that $\rho^\star$ is bounded from above by $C^\star$ and $\mathcal{F}(\rho^\star,\lambda)<\mathcal{F}(\rho,\lambda)$.
\end{lem}
 
\begin{proof}
Let
$$\rho^\star(p):=C^\star \mathbbm{1}_{\{p|\rho(p)> C^\star\}}+\rho(p)\mathbbm{1}_{\{p|C^\star\geq \rho(p)>\frac{1}{2\pi}\}}+(\rho(p)+\gamma)\mathbbm{1}_{\{p|\rho(p)\leq \frac{1}{2\pi}\}}$$
where $\gamma:=\frac{\int_{\mathbb{S}^2}(\rho(p)-C^\star)\mathbbm{1}_{\{\rho(p)> C^\star\}}\;dp}{|\{p|\frac{1}{2\pi}\geq\rho(p)\}|}$.
For the ease of presentation, we use the following notation
\begin{align*}
A:&=\left\{p\in \SP\Big|\rho(p)>C^\star\right\},\quad
B:=\left\{p\in \SP\Big|C^\star \geq \rho(p)\geq \frac{1}{2\pi}\right\} \text{ and }
\end{align*}
\begin{align*}
C:&=\left\{p\in \SP\Big| \frac{1}{2\pi} > \rho(p)\right\}.
\end{align*}
It is easy to see that $|C|>0$ since otherwise $\int_{\mathbb{S}^2}\rho(p) \;dp>1$. By the choice of $\gamma$, we make sure that $\rho^\star$ is a probability density that still integrates to one. 
Observe that $C^\star =\exp(16 M/\lambda)\geq 1 >\frac{1}{2\pi}$ and hence $A\cap B=\emptyset$. 
In order to prove that
\begin{align*}
\mathcal{F}(\rho,\lambda)-\mathcal{F}(\rho^\star,\lambda)>0,
\end{align*}
we deal with the entropic and the interaction term separately. We begin by deriving a lower bound on the entropy term. In particular,
\begin{align*}
\mathcal{F}_1(\rho,\lambda)&-\mathcal{F}_1(\rho^\star,\lambda)\\=&\lambda\int_{\mathbb{S}^2}\rho(p)\ln(\rho(p))\;dp-\lambda\int_{\mathbb{S}^2}\rho^\star(p)\ln(\rho^\star(p))\;dp\\
=&\lambda\int_{A}\left[f(\rho(p))-f(C^\star)\right]dp+\lambda\int_{C}\left[f(\rho(p))-f(\rho(p)+\gamma)\right]\;dp
\end{align*}
where $f(x):=x\ln(x)$. Moreover, defining $g(p):=\rho(p)-C^\star$, we may apply the fundamental theorem of calculus (notice that the following claim holds trivially for $g(p)=0$, so we may assume that $g(p)\neq 0$ for all $p\in \mathbb{S}^2$). Hence
\begin{align*}
\mathcal{F}_1(\rho,\lambda)&-\mathcal{F}_1(\rho^\star,\lambda)\\
=&\lambda\int_{A}\frac{f(g(p)+C^\star)-f(C^\star)}{g(p)}g(p)\;dp-\lambda\int_{C}\frac{f(\rho(p)+\gamma)-f(\rho(p))}{\gamma}\gamma \;dp\\
=&\lambda\int_{A}f'(\xi_1)g(p)\;dp-\lambda\gamma\int_{C}f'(\xi_2) \; dp
\end{align*}
where $\xi_1(p)\in [C^\star, g(p)+C^\star]$ and $\xi_2(p) \in [\rho(p), \rho(p)+\gamma]$ with $\rho \in C$. Since $f'(x)=\ln(x)+1$, which is monotonically increasing, it follows that
\begin{align*}
\mathcal{F}_1(\rho,\lambda)&-\mathcal{F}_1(\rho^\star,\lambda)\\
\geq&\lambda(\ln C^\star+1)\int_{A}g(p)\;dp-\gamma\lambda\int_{C}(\ln(\rho(p)+\gamma)+1) \;dp\\
\geq&\lambda(\ln C^\star+1)\int_{A}g(p)\;dp-\gamma\lambda\int_{C}\left(\ln\left(\frac{1}{2\pi}+\gamma\right)+1\right) \;dp
\end{align*}
where the last line follows from the fact that $\rho \in C$ and hence $\rho(p)<\frac{1}{2\pi}$ for all $p\in C$.

We now claim that
$\gamma < \frac{1}{2\pi}$.
Since $\rho$ is a probability density, it follows that $\int_{A}(\rho(p)-C^\star)\;dp\leq 1$. Hence 
\begin{align*}
\gamma:=\frac{\int_{\mathbb{S}^2}(\rho(p)-C^\star)\mathbbm{1}_{\{\rho(p)> C^\star\}}\;dp}{|\{p|\frac{1}{2\pi}\geq\rho(p)\}|}\leq \frac{1}{|\{p|\frac{1}{2\pi}\geq\rho(p)\}|}
\end{align*} 
and we only have to show that $|C|=|\{p|\frac{1}{2\pi}\geq\rho(p)\}|\geq 2\pi$. Suppose the contrary, that is $|C|<2\pi$ and thus $|A\cup B|=|\{\rho(p)> \frac{1}{2\pi}\}|\geq 2\pi$ since the total measure of the unit sphere is given by $4\pi$ and thus $|A\cup B\cup C|=4\pi$. Moreover, since $\rho$ is positive,
$$\int_{\mathbb{S}^2}\rho(p)\;dp>\int_{C^c}\rho(p)\;dp>\frac{1}{2\pi}\int_{C^c}1\;dp\geq \frac{1}{2\pi}2\pi =1$$
contradicting the fact that $\rho$ integrates to one. Therefore $\gamma < \frac{1}{2\pi}$.\\
 
Observing that $g(p)\geq 0$ for all $p\in A$ and using the monotonicity of the logarithm as well as the facts that $\gamma < \frac{1}{2\pi}$, $\ln\frac{1}{\pi}<-1$ and $\lambda, \gamma, |C|>0$, we deduce that
\begin{align*}
\mathcal{F}_1(\rho,\lambda)&-\mathcal{F}_1(\rho^\star,\lambda)\\
\geq&\lambda(\ln C^\star+1)\int_{A}g(p)\;dp-\lambda\left(\ln\left(\frac{1}{2\pi}+\frac{1}{2\pi}\right)+1\right)\gamma |C| \\
\geq&\lambda\ln(C^\star)\int_{A}g(p)\;dp.
\end{align*}
This gives us a lower bound on the entropic term and we can direct our attention to a lower bound on the interaction term.
\noindent Using the definition of $\rho^\star$, it follows that
\begin{align*}
\mathcal{F}_2(\rho,\lambda)&-\mathcal{F}_2(\rho^\star,\lambda)\\
=&\frac{1}{2}\int_{\mathbb{S}^2\times \mathbb{S}^2}K(p,q)[\rho(p)\rho(q)-\rho^\star(p)\rho^\star(q)]\;dq\;dp\\
=&\frac{1}{2}\int_{A\times A}K(p,q)h(p,q)\;dq\;dp-\frac{\gamma^2}{2}\int_{C\times C}K(p,q)\;dq\;dp\\&-\gamma\int_{C\times C}K(p,q) \rho(p)\;dq\;dp+\int_{A\times B}K(p,q)\rho(q)g(p)\;dp\;dq\\&+\int_{A\times C}K(p,q)\rho(q)g(p)\;dq\;dp-\gamma C^\star \int_{A\times C}K(p,q)\rho(q)\;dq\;dp\\&-\gamma \int_{B\times C}K(p,q)\rho(q)\; dq\;dp
\end{align*}
where we used the symmetry of the kernel and again we defined $g(p):=\rho(p)-C^\star$ and $h(p,q):=\rho(p)\rho(q)-C^\star C^\star$ which are both non-negative for $p,q\in A$. Assuming without loss of generality that the interaction kernel is positive, we observe that all terms with a positive sign are positive while only those with a negative sign are negative. Neglecting the positive terms, we may therefore deduce that
\begin{align}
\mathcal{F}_2(\rho,\lambda)&-\mathcal{F}_2(\rho^\star,\lambda)\notag\\
\geq&  -\frac{\gamma^2}{2}\int_{C\times C}K(p,q)\;dq\;dp -\gamma\int_{C\times C}K(p,q) \rho(p)\;dq\;dp\notag\\
&-\gamma \int_C\left(\int_{A}K(p,q)\rho(p)dp\right)\;dq-\gamma \int_{B\times C}K(p,q)\rho(q) \;dq\;dp\label{eqt:comment on A}
\end{align}
\begin{align}
\geq & -\frac{\gamma^2}{2}M(4\pi)^2 -4\pi\gamma M-4\pi\gamma M-4\pi M\gamma\notag\\
\geq & -\gamma^2M8\pi^2 -12\pi\gamma M\notag
\end{align}
where we used the fact that $\rho(p)\geq C^\star$ for all $p\in A$ in (\ref{eqt:comment on A}). Since $\gamma < \frac{1}{2\pi}$
\begin{align*}
\mathcal{F}_2(\rho,\lambda)-\mathcal{F}_2(\rho^\star,\lambda)> & -4 \pi \gamma M -12\pi\gamma M=-16 \pi \gamma M,
\end{align*}
we obtain a lower bound on the interaction term.\\

Thus, we are now in the position to derive an overall bound on $\mathcal{F}$.
Putting both lower bounds together and using the definition of $\gamma=\frac{\int_A g(p)dp}{|C|}$, we obtain that
\begin{align*}
\mathcal{F}(\rho,\lambda)-\mathcal{F}(\rho^\star,\lambda)> & \lambda\ln(C^\star)\int_{A}g(p)\;dp-16 \pi \gamma M\\
= & \lambda\ln(C^\star)|C|\gamma-16 \pi \gamma M.
\end{align*}
Finally, making use of the assumption that $C^\star = \exp(16M/\lambda)$ and using the fact that $|C|\leq 2\pi$, the statement follows because
\begin{align*}
\mathcal{F}(\rho,\lambda)-\mathcal{F}(\rho^\star,\lambda)> & \lambda \frac{16M}{\lambda} |C|\gamma-16 \pi \gamma M
=16M\gamma(|C|-\pi)>0.
\end{align*}
\end{proof}

\begin{cor}\label{cor: reduction principle}
Let two sets $\mathcal{A}$ and $\mathcal{B}$ be given by
$$\mathcal{A}:=\left\{\rho:\SP\rightarrow\R| \int_{\SP}\rho(p)\;dp=1,\rho\geq 0\right\}$$
$$\mathcal{B}:=\left\{\rho:\SP\rightarrow\R| \int_{\SP}\rho(p)\;dp=1,\rho\geq 0, \rho\leq \exp(16M/\lambda)\right\}.$$
Then any solution of the minimisation problem $\min_{\mathcal{A}}\mathcal{F}(\rho,\lambda)$ is in fact a solution of the reduced minimisation problem 
$\min_{\mathcal{B}}\mathcal{F}(\rho,\lambda).$
\end{cor}
\begin{proof}
The statement follows directly from Lemma \ref{lem:boundedness lemma}.
\end{proof}

\subsection{Local properties of the trivial solution $\rho_0=\frac{1}{4\pi}$}\label{sec: stability}
We concentrate on proving conditions which ensure when the trivial solution is a local minimum and when it is not.
\begin{prop}\label{prop: stability}
For all $\lambda>\lambda_2$, the trivial solution $\rho_0=\frac{1}{4\pi}$ is a local minimum. For all $0<\lambda<\lambda_2$, the trivial solution $\rho_0=\frac{1}{4\pi}$ is not a local minimum.
\end{prop}
\begin{proof}
The second variation of the Onsager free-energy functional in \eqref{eqt:free energy} in $L^\infty$ is given by
\begin{align}\label{eqt: second variation of F}
I(z,\lambda)&=D^2_z \mathcal{F}(\rho)=\frac{\partial^2}{\partial \epsilon^2}\mathcal{F}(\rho+\epsilon z)\Bigg |_{\epsilon=0}\notag\\&=\lambda \int_{\mathbb{S}^2}\frac{z(p)^2}{\rho(p)}\;dp+\int_{\mathbb{S}^2\times \mathbb{S}^2}k(p\cdot q)z(p)z(q)\;dpdq
\end{align}
with $z\in L^\infty(\SP)$ such that $\int_\SP z(p)\;dp=0$. The Stone-Weierstra{\ss} theorem guarantees that we can find $a_{m,l}$ such that $$\left\lVert\sum_{m,l}a_{m,l}Y_l^m-z\right\rVert_{L^\infty} \leq \epsilon$$ for some $\epsilon>0$. Using the fact that the interaction operator $U$ turns into a multiplication operator when it is applied to a spherical harmonic, 
\begin{align*}
U(Y^m_l)=\begin{cases}\mu_l&\text{ if }l\text{ is even}\\0&\text{ if }l\text{ is odd},\end{cases}
\end{align*}
we can rewrite the second variation as
\begin{align*}
I(z,\lambda)=&4\pi\lambda \int_{\mathbb{S}^2}\left(\sum_{l=1}^\infty\sum_{m=-l}^l a_{lm}Y^m_l(p)\right)^2\;dp\\&+\int_{\mathbb{S}^2}\left(\sum_{l=1}^\infty\sum_{m=-l}^l a_{lm}\int_{\mathbb{S}^2}k(p\cdot q)Y^m_l(p)\;dp\right)\sum_{l=1}^\infty\sum_{m=-l}^l a_{lm}Y^m_l(q)\;dq\\
=&4\pi\lambda \int_{\mathbb{S}^2}\left(\sum_{l=1}^\infty\sum_{m=-l}^l a_{lm}Y^m_l(p)\right)^2\;dp\\&+\int_{\mathbb{S}^2}\sum_{l\text{ is even}}^\infty\sum_{m=-l}^l a_{lm}\mu_lY^m_l(q)\sum_{l=1}^\infty\sum_{m=-l}^l a_{lm}Y^m_l(q)\;dq
\end{align*}
where we skip the term $l=0$ so that $\rho+z$ still integrates to $1$. Viewing the spherical harmonics as an orthonormal basis in $L^2(\mathbb{S}^2)$, it follows that 
\begin{align*}
I(z,\lambda)
&=4\pi\lambda \sum_{l=1}^\infty\sum_{m=-l}^l |a_{lm}|^2+\sum_{l\text{ is even}}^\infty\sum_{m=-l}^l \mu_l |a_{lm}|^2.
\end{align*}
Hence the second variation is positive, that is $I(z,\lambda)>0 \text{ for all } 0\neq z\in L^2(\mathbb{S}^2)$, if and only if $4\pi\lambda >-\mu_2.$ Thus, the trivial solution $\rho_0$ is a local minimiser of the bifurcation equation for all
$$\lambda>-\frac{\mu_2}{4\pi}=\lambda_2.$$
Moreover, we observe that the following holds
\begin{align*}
\mathcal{F}\left(\frac{1}{4\pi}+z(p)\right)&=\mathcal{F}\left(\frac{1}{4\pi}\right)+\mathcal{F}'\left(\frac{1}{4\pi}\right)z(p)+\frac{1}{2}\mathcal{F}''\left(\frac{1}{4\pi}\right)z(p)+\text{ h.o.t}\\
&=\mathcal{F}\left(\frac{1}{4\pi}\right)+\frac{1}{2}\underbrace{I(z(p),\lambda)}_{\leq 0 \text{ if }\lambda<\lambda_2}+\text{ h.o.t}\\
&\leq \mathcal{F}\left(\frac{1}{4\pi}\right)+\text{ h.o.t}
\end{align*}
If we choose $z$ close enough to zero, then the higher order terms are negligible and we see that $\rho=\frac{1}{4\pi}$ is not a local minimiser for all $\lambda<\lambda_2$.
\end{proof}

\subsection{Existence of a unique solution for high temperatures}\label{sec: unique solution for high temperatures}
We prove that the Onsager free-energy functional is strictly convex in $\rho$ and therefore that the isotropic state, represented by the uniform distribution $\rho_0=\frac{1}{4\pi}$, is the unique solution for high temperatures.

\begin{thm}\label{thm:convexity for high temperatures} 
Let $k(\cdot)$ be a continuous interaction potential and let $\lambda$ be given such that the following inequality holds $$8\pi M \exp(16M/\lambda)\leq \lambda.$$ Then the Onsager free-energy functional $\mathcal{F}$ is strictly convex on the set of all probability densities $\rho$ bounded by $C^\star$, see Lemma \textup{\ref{lem:boundedness lemma}}.
\end{thm}
 
\begin{proof}
Recall that the second variation of $\mathcal{F}$ is given by
\begin{align*}
\lambda\int_{\mathbb{S}^2}\frac{z^2(p)}{\rho(p)}\;dp+\int_{\mathbb{S}^2\times \mathbb{S}^2}k(p\cdot q)z(p)z(q)\;dp\;dq\geq 0,
\end{align*}
see (\ref{eqt: second variation of F}).
We know from Lemma \ref{lem:boundedness lemma} that it is enough to consider the minimisation problem among the set of probability densities $\rho \in L^1(\mathbb{S}^2)$ which are bounded by a constant $C^\star=\exp(16M/\lambda)$ where $M$ denotes an upper bound on the kernel $k(p\cdot q)$. 
Therefore an application of Jensen's inequality and imposing the condition that $8\pi M\leq \lambda/C^\star$ yield
\begin{align*}
\lambda\int_{\mathbb{S}^2}\frac{z^2(p)}{\rho(p)}\;dp + \int_{\mathbb{S}^2\times \mathbb{S}^2}&k(p\cdot q)z(p)z(q)\;dp\;dq\\
\geq &\lambda\int_{\mathbb{S}^2}\frac{z^2(p)}{\rho(p)}\;dp-M \left(4\pi\int_{\mathbb{S}^2}|z(p)|\;\frac{dp}{4\pi}\right)^2\\
\geq &\lambda\int_{\mathbb{S}^2}\frac{z^2(p)}{\rho(p)}\;dp-M 4\pi\int_{\mathbb{S}^2}z(p)^2\;dp\\
\geq &  \lambda\int_{\mathbb{S}^2}\frac{z^2(p)}{\rho(p)}\;dp-\frac{\lambda}{2C^\star}\int_{\mathbb{S}^2}z(p)^2dp\\
\geq &\frac{\lambda}{2C^\star}\int_{\mathbb{S}^2}z^2(p)\;dp\\>&0,
\end{align*}
proving that the Onsager functional is strictly convex.
Rearranging this condition and plugging in the constant $C^\star=\exp(16M/\lambda)$ yields the above inequality
$8\pi M \exp(16M/\lambda)\leq \lambda.$
\end{proof}

Theorem \ref{thm:convexity for high temperatures} proves that the Onsager free-energy functional is a strongly convex functional for all probability densities $\rho \in \mathcal{S}$ whenever the temperature parameter satisfies the inequality stated in Theorem \ref{thm:convexity for high temperatures}. Numerically, this value can be approximated by $\lambda^\star\approx 38.205$ when $M=1$ (for example in case of the Onsager interaction potential).

\begin{cor}
Let $k(\cdot)$ be a continuous interaction potential and let $\lambda^\star$ be chosen such that $8\pi M \exp(16M/\lambda)\leq \lambda$. Then the uniform density $\rho_0=\frac{1}{4\pi}$ is the unique solution to the Onsager free-energy functional for all $\lambda \geq \lambda^\star$.
\end{cor}

\begin{proof}
The uniform density $\rho_0=\frac{1}{4\pi}$ is indeed a solution to the Euler-Lagrange equation for all $\lambda\in \R$. By Corollary \ref{cor: reduction principle} and Theorem \ref{thm:convexity for high temperatures}, the necessary optimality condition is also sufficient \cite[Theorem 5.2.4]{A2}.
\end{proof}

\subsection{Continuity of all bifurcation branches}\label{sec:rabinowitz continuous bifurcation branches}
We conclude this section by showing that every bifurcation branch of the Onsager functional is either unbounded and meets infinity or that it meets another bifurcation branch.

The Euler-Lagrange equation in terms of the thermodynamic potential is given by
\begin{align*}
E(\phi,\lambda)=\lambda \phi(p)-\frac{1}{Z(\phi)}\int_{\mathbb{S}^2} K(p,q)\exp(-\phi(q))\;dq,
\end{align*}
which is of the form $\lambda \text{Id}-U$ where $U$ is compact, see Lemma \ref{lem:compactInteraction}. Therefore we can rewrite the bifurcation problem as
$$\frac{1}{\lambda}U(\phi,\lambda)=\phi$$
and it is thus of the form of bifurcation problems considered by Rabinowitz \cite{Rabinowitz1971}. 
\begin{citedthm}[{{\cite[Theorem 1.3]{Rabinowitz1971}}}]\label{thm: rabinowitz}
Let $G:E\times \R\rightarrow E$ be a compact and continuous non-linear operator and consider the problem
$$u=G(u,\lambda)$$ where $u\in E$ and $\lambda \in \R$. Moreover, let 
$$G(u,\lambda)=\lambda Lu+H(u,\lambda)$$
where $H(u,\lambda)$ is $O(||u||)$ for $u$ near the origin uniformly for bounded $\lambda$ and let $L$ be a compact and linear map on $E$. If $\lambda$ is a real non-zero eigenvalue of $L$ of odd multiplicity, then the solution set $\left\{ (u,\lambda) \mid u=G(u,\lambda) \right\}$ possesses a maximal connected and closed subset $C_u$ such that $(0,\lambda)\in C_u$ and $C_u$ either
\begin{enumerate}
\item meets infinity in $E$, or
\item meets $(0,\hat{\lambda})$ where $\lambda \neq \hat{\lambda}$ is a real non-zero eigenvalue of $L$.
\end{enumerate}
\end{citedthm}
In particular, Rabinowitz shows that any eigenvalue of odd multiplicity is a bifurcation point and that any branch bifurcating from the trivial solution must either meet infinity or meet another bifurcation branch. In case of the Onsager functional, all eigenvalues have multiplicity $2n+1$, thus we know that $(0,\lambda_s)$ is not an isolated solution but that it is instead a member of a non-trivial closed connected set described by Theorem \ref{thm: rabinowitz}.

\section{Conclusion}\label{sec:conclusion}
We found a general method to derive the eigenvalues of integral operators of the form $$U(\rho)=\int_{\SP2}K(p,q)\rho(p)\rho(q)\;dpdq$$
depending on kernels satisfying Assumption \ref{assumption} and we derived an explicit expression for the eigenvalues of the Onsager free-energy functional involving the Onsager kernel. Based on this result, we know from \cite{Rabinowitz1971}, that all eigenvalues are bifurcation points and we proved that a transcritical bifurcation occurs at the temperature $\tau_\star=\lambda_2/k_B$ where $\lambda_2=\frac{\pi}{32}$ denotes the largest of these eigenvalues. Moreover, we showed that the corresponding solution is uniaxial. Globally, 
we showed that the trivial solution $\rho=\frac{1}{4\pi}$ is the unique solution for high temperatures. Finally, we verified that all bifurcation branches either meet infinity or they meet another bifurcation branch.

In order to obtain a complete bifurcation diagram, it is left to show that no isolas occur, that means that we do not have any bifurcations that are isolated from all other solutions to our problem at hand. Our approach is not at all limited to the Onsager kernel and not even to quadratic forms given by integral kernels. This will be demonstrated in forthcoming work.

%

\begin{acknowledgements}
The author would like to thank John Ball, Yi-Chao Chen and Markus Upmeier for insightful and stimulating discussions relating to the work in this paper. The research leading to these results has received funding from the European Research Council under the European Union’s Seventh Framework Programme $(FP7/2007-2013)$/ERC grant agreement $n^\circ 291053$.
\end{acknowledgements}

\bibliographystyle{spmpsci}  
\bibliography{library4.bib}

\appendix
\section{A brief introduction to spherical harmonics and associated Legendre polynomials}\label{app: notation spherical harmonics}
The spherical harmonics are defined as
\begin{align*}
Y^m_l(\varphi,\theta):=N_{lm}e^{im\varphi}P^m_l(\cos \theta)
\end{align*}
where $\varphi\in [0,2\pi)$ denotes the polar angle and $\theta\in [0,\pi)$ the azimuthal angle.
The subscript $l$ denotes its degree while $-l\leq m\leq l$ denotes its order. $N_{lm}$ is a normalisation constant given by 
$$N_{lm}:=(-1)^m\left(\frac{2l+1}{4\pi}\right)^{1/2}\left(\frac{(l-m)!}{(l+m)!}\right)^{1/2}$$
and $P^m_l$ is the set of associated Legendre polynomials given by
$$P^m_l(\mu):=\frac{1}{2^ll!}\left(1-\mu^2\right)^{m/2}\frac{\partial^{l+m}}{\partial\mu^{l+m}}(\mu^2-1)^l.$$
One of the most important facts of this set of spherical harmonics is that they build an orthonormal basis of $L^2(\mathbb{S}^2)$. In general, spherical harmonics are known as eigenfunctions of the Laplacian. However, it seems that they can be understood in a much broader sense as eigenfunctions of rotationally symmetric operators. 

For more information on spherical harmonics and associated Legendre polynomials see \cite{Wolfram-ALP} and \cite{Wolfram-SH}.



\section{Technical results used in the derivation of an explicit expression for the eigenvalues of the Onsager functional}\label{app: eigenvalue chapter}
\begin{lem}[Supplement to Theorem \ref{thm:eigenvalues and eigenfunctions of general interaction kernels}]\label{lem: polynomial recurrence Legendre Polynomials}
\begin{align}\label{eqn: legendre power formula}
x^r=\sum_{l=r,r-2,\dots}\frac{(2l+1)r!}{2^{(r-l)/2}(\frac{1}{2}(r-l))!(l+r+1)!!}P_l(x).
\end{align}
\end{lem}

\begin{proof}
The claim follows from an induction based on the reccurrence relation for Legendre polynomials 
$$xP_{l}(x)=\frac{l+1}{2l+1}P_{l+1}(x)+\frac{l}{2l+1}P_{l-1}(x)$$ 
with the convention that $P_{-1}=0$ \cite{Wolfram-ALP}. The base case $r=0$ is trivial because
$P_{0}=1$, so let us move on to the actual induction step. Multiplying both sides of (\ref{eqn: legendre power formula})
with $x$ yields
\begin{align*}
x^{r+1} & =\sum_{l=r,r-2,\dots}\frac{(2l+1)r!}{2^{(r-l)/2}(\frac{1}{2}(r-l))!(l+r+1)!!}\left(\frac{l+1}{2l+1}P_{l+1}(x)+\frac{l}{2l+1}P_{l-1}(x)\right).
\end{align*}
Changing the index in the two parts of the sum corresponding to $P_{l+1}$ and $P_{l-1}$, we obtain 
\begin{align*}
 x^{r+1}=&\sum_{l=r-1,r-3,\dots>1}\frac{lr!(l+r+2)+2\frac{1}{2}(r+1-l)(l+1)r!}{2^{(r+1-l)/2}(\frac{1}{2}(r+1-l))!(l+r+2)!!}P_{l}(x)\\
 & +\frac{r!}{(l+1+r)!!}\frac{r+1}{1}P_{r+1}(x)+\begin{cases}
\frac{3r!}{2^{(r-1)/2}(\frac{1}{2}(r-1))!(r+2)!!}\frac{1}{3}P_{0}(x) & \text{if }r\text{ is odd}\\
\frac{5r!}{2^{(r-2)/2}(\frac{1}{2}(r-2))!(r+3)!!}\frac{2}{5}P_{1}(x) & \text{if }r\text{ even}
\end{cases}.
\end{align*}
Notice that we excluded the boundary cases from the overall sum. Putting everything together yields
\begin{align*}
x^{r+1}= & \sum_{l=r+1,r-1,\dots\geq0}\frac{(2l+1)(r+1)!}{2^{(r+1-l)/2}(\frac{1}{2}(r+1-l))!(l+r+2)!!}P_{l}(x)
\end{align*}
which concludes the induction step.
\end{proof}

\begin{lem}[Supplement to Corollary \ref{cor:eigenvalues and eigenfunctions of the Onsager kernel}]\label{lem:interchange of the sum and the integral in case of the Onsager kernel}
The change of integral and infinite sum in \eqref{eqt:intgrating interaction operator and SH} is valid in case of the Onsager kernel. In other words, we can verify condition \eqref{eqt: condition for swapping integrals}, stating that
$$\sum_{l=0}^\infty\sum_{r=l}^\infty (4l+1)^{3/2}|c_l^r|<\infty,$$
for 
\begin{align*}
c_{l}^r:=\begin{cases}\frac{4\pi (2r)!^2}{(1-2r)(r!)^2(4^r)2^{r-l}(r-l)!(2l+2r+1)!!}&\text{ if }l\leq r\\0 &\text{ if }l>r\end{cases}.
\end{align*}
Moreover, the eigenvalues of the interaction operator equipped with the Onsager kernel decrease faster than $l^{-3}$.
\end{lem}
\begin{proof}
With regard to \eqref{eq:Onsager kernel as series}, the Onsager kernel written in terms of spherical harmonics is given by
\begin{align}
k_O&(p\cdot q)=\sum_{r=0}^\infty \sum_{l=0}^r\sum_{m=-2l}^{2l} \underbrace{\frac{4\pi (2r)!^2}{(1-2r)(r!)^2(4^r)2^{r-l}(r-l)!(2l+2r+1)!!}}_{=:c_l^r}Y_{2l}^m(p){Y^\star}_{2l}^m(q)\tag{\ref{eq:Onsager kernel as series}}.
\end{align}
The coefficients $c_l^r$ are all negative except for $c_0^0=4\pi$. Taking this into account and applying Lemma \ref{lem:closed form of eigenvalues for the Onsager kernel}, we know that 
\begin{align}
\sum_{l=0}^\infty\sum_{r=l}^\infty &(4l+1)^{3/2}|c_l^r|\notag\\
&=\sum_{l=1}^\infty (4l+1)^{3/2}\sum_{r=l}^\infty |c_l^r|+\sum_{r=0}^\infty |c_0^r|\notag\\
&=-\sum_{l=1}^\infty (4l+1)^{3/2}\sum_{r=l}^\infty c_l^r+\left(-\sum_{r=0}^\infty c_0^r+2c_0^0\right)\notag\\
&=\frac{\pi}{2}\sum_{l=1}^\infty (4l+1)^{3/2}\frac{\Gamma(l+\frac{1}{2})\Gamma(l-1/2)}{\Gamma(l+1)\Gamma(l+2)}+\left(\frac{\pi\Gamma(\frac{1}{2})\Gamma(-1/2)}{2\Gamma(1)\Gamma(2)}+8\pi\right)\notag\\
&=\frac{\pi}{2}\sum_{l=1}^\infty (4l+1)^{3/2}\frac{\Gamma(l+\frac{1}{2})\Gamma(l-1/2)}{\Gamma(l+1)\Gamma(l+2)}+\left(-\pi^2+8\pi\right).\label{eqt: change of integration}
\end{align}
Using Gautschi's inequality \cite{Laforgia1984},
$$\frac{\Gamma(x+1)}{\Gamma(x+a)}<(x+1)^{1-a}\text{ for any }0<a<1,$$
it follows that
\begin{align}
\frac{\Gamma(l+\frac{1}{2})\Gamma(l-1/2)}{\Gamma(l+1)\Gamma(l+2)}&<\frac{\Gamma(l-1/2)}{(l+1)^{1/2}\Gamma(l+2)}
=\frac{\Gamma(l-1/2)}{(l+1)^{1/2}\Gamma(l)l(l+1)} \notag\\
&<\frac{1}{l(l+1)^2}<\frac{1}{l^3}\label{eqt: eigenvalues decrease more than l3}.
\end{align}
Inserting this result into \eqref{eqt: change of integration} yields
$$\sum_{l=0}^\infty\sum_{r=l}^\infty (4l+1)^{3/2}|c_l^r|<\frac{\pi}{2}\sum_{l=1}^\infty \frac{(4l+1)^{3/2}}{l^3}+\left(-\pi^2+8\pi\right).$$
Comparing this with the harmonic $p$-series, we observe that
$$\frac{\pi}{2}\sum_{l=1}^\infty \frac{(4l+1)^{3/2}}{l^3}+\left(-\pi^2+8\pi\right)\sim C_1\sum_{l=1}^\infty \frac{1}{l^{3/2}}+C_2<\infty$$
for some constants $C_1$ and $C_2$ which therefore proves the first part of the claim. The second part follows easily from \eqref{eqt: eigenvalues decrease more than l3}.
\end{proof}

\begin{lem}[Supplement to Corollary \ref{cor:eigenvalues and eigenfunctions of the Onsager kernel}]\label{lem:closed form of eigenvalues for the Onsager kernel}
Let $s\in 2\mathbb{N}$. Then
\begin{align*}
\sum_{r=s/2}^\infty &\frac{4\pi (2r)!^2}{(1-2r)(r!)^2(4^r)2^{r-s/2}(r-s/2)!(s+2r+1)!!}\\&\hspace{5cm}= -\frac{\pi\Gamma(s/2+\frac{1}{2})\Gamma(s/2-1/2)}{2\Gamma(s/2+1)\Gamma(s/2+2)}.
\end{align*}
\end{lem}

\begin{proof}
First of all, let us rewrite the sum by translating it by $s/2$. This yields
\begin{align*}
\sum_{r=s/2}^\infty &\frac{4\pi (2r)!^2}{(1-2r)(r!)^2(4^r)2^{r-s/2}(r-s/2)!(s+2r+1)!!}\\
&\hspace{3cm}=\sum_{r=0}^\infty \frac{\pi((2r+s)!)^2}{(1-2r-s)((r+s/2)!)^22^{3r+s-2}r!(2s+2r+1)!!}.
\end{align*}
In case of odd integers $n=2k-1$ for $k\geq 1$, the double factorial $n!!$ can be replaced by the expression $n!!=\frac{(2k)!}{2^kk!}$ \cite[page 823]{Weisstein2002}. Hence we can rewrite the terms such that
\begin{flalign*}
&\sum_{r=0}^\infty \frac{\pi((2r+s)!)^2}{(1-2r-s)((r+s/2)!)^2(2^{3r+s-2})r!(2(r+s+1)-1)!!}&\\
&\hspace{0.5cm}=\sum_{r=0}^\infty \frac{\pi((2r+s)!)^22^{-2r+3}(r+s+1)!}{(1-2r-s)((r+s/2)!)^2r!(2r+s+1)!}&\\
&\hspace{0.5cm}=-\pi\sum_{r=0}^\infty \left(\frac{(2r+s)!}{(r+s/2)!}\right)^2\frac{(r+s+1)!}{(2(r+s+1))!}\frac{(2r+s-2)!}{(2r+s-1)!}\frac{2^{-2r+3}}{r!}.
\end{flalign*}
Notice that we also expressed the first factor in the denominator as fraction of two factorials. This sum can now further be manipulated by rewriting it in terms of gamma functions \cite[page 1137]{Weisstein2002},
$$\Gamma(s+1)=s!.$$
It follows that
\begin{align*}
&-\pi\sum_{r=0}^\infty \left(\frac{(2r+s)!}{(r+s/2)!}\right)^2\frac{(r+s+1)!}{(2(r+s+1))!}\frac{(2r+s-2)!}{(2r+s-1)!}\frac{2^{-2r+3}}{r!}\\
&\hspace{0.5cm}=-\pi\sum_{r=0}^\infty \left(\frac{\Gamma(2r+s+1)}{\Gamma(r+s/2+1)}\right)^2\frac{\Gamma(r+s+2)}{\Gamma(2(r+s)+3))}\frac{\Gamma(2r+s-1)}{\Gamma(2r+s)}\frac{2^{-2r+3}}{r!}\\
&\hspace{0.5cm}=-\pi\sum_{r=0}^\infty \left(\frac{(2r+s)\Gamma(2r+s)}{(r+s/2)\Gamma(r+s/2)}\right)^2\frac{(r+s+1)\Gamma(r+s+1)\Gamma(2r+s-1)2^{-2r+3}}{(2(r+s+1))\Gamma(2(r+s+1))\Gamma(2r+s)r!}\\
&\hspace{0.5cm}=-\pi\sum_{r=0}^\infty \left(\frac{\Gamma(2r+s)}{\Gamma(r+s/2)}\right)^2\frac{\Gamma(r+s+1)}{\Gamma(2(r+s+1))}\frac{\Gamma(2r+s-1)}{\Gamma(2r+s)}\frac{2^{-2r+4}}{r!}.
\end{align*}
Applying the duplication formula for the gamma function \cite[page 1139]{Weisstein2002}
$$\frac{\Gamma(2z)}{\Gamma(z)}=\frac{2^{2z-1}\Gamma(z+\frac{1}{2})}{\sqrt{\pi}},$$
to the first two factors and its rearranged form
$$\Gamma(z)=\frac{2^{z-1}\Gamma(z/2+\frac{1}{2})\Gamma(z/2)}{\sqrt{\pi}},$$
to the two terms forming the third factor, we deduce that
\begin{align*}
&-\pi\sum_{r=0}^\infty \left(\frac{\Gamma(2r+s)}{\Gamma(r+s/2)}\right)^2\frac{\Gamma(r+s+1)}{\Gamma(2(r+s+1))}\frac{\Gamma(2r+s-1)}{\Gamma(2r+s)}\frac{2^{-2r+4}}{r!}\\
&\hspace{0.5cm}=-\pi\sum_{r=0}^\infty \left(\frac{2^{2r+s-1}\Gamma(r+s/2+\frac{1}{2})}{\sqrt{\pi}}\right)^2\frac{2^{1-2(r+s+1)}\sqrt{\pi}}{\Gamma(r+s+3/2)}\\
&\hspace{4.8cm}\cdot\frac{2^{2r+s-2}\Gamma(r+s/2)\Gamma(r+s/2-1/2)}{2^{2r+s-1}\Gamma(r+s/2+1/2)\Gamma(r+s/2)}\frac{2^{-2r+4}}{r!}\\
&\hspace{0.5cm}=-\sqrt{\pi}\sum_{r=0}^\infty \frac{\Gamma(r+s/2+\frac{1}{2})\Gamma(r+s/2-1/2)}{\Gamma(r+s+3/2)}\frac{1}{r!}.
\end{align*}
Let $(x)_r$ denote the Pochhammer symbols which are given by
\begin{align*}
(x)_0&=1\\
(x)_r&=\frac{\Gamma(x+r)}{\Gamma(x)}
\end{align*}
for any $x\in \mathbb{C}$ and $r\in \R$. Rewriting our expression in terms of Pochhammer symbols \cite{HardyCunningham1964}, we obtain
\begin{flalign*}
&-\sqrt{\pi}\sum_{r=0}^\infty \frac{\Gamma(r+s/2+\frac{1}{2})\Gamma(r+s/2-1/2)}{\Gamma(r+s+3/2)}\frac{1}{r!}&\\
&\hspace{0.5cm}=-\frac{\sqrt{\pi}\Gamma(s/2+\frac{1}{2})\Gamma(s/2-1/2)}{\Gamma(s+3/2)}\sum_{r=0}^\infty \frac{(s/2+\frac{1}{2})_r(s/2-1/2)_r}{(s+3/2)_r}\frac{1}{r!}.
\end{flalign*}
This sum can now easily be evaluated using generalised hypergeometric functions \cite{HardyCunningham1964}.
\begin{defi}
The generalised hypergeometric function ${_p}F_q$ is defined as
$${}_pF_q(x_1,x_2,\dots,x_p;y_1,y_2,\dots,y_q;z)=\sum_{r=0}^\infty\frac{(x_1)_r(x_2)_r\dots (x_p)_r}{(y_1)_r(y_2)_r\dots(y_q)_r}\frac{z^r}{r!}$$ where $(x)_r$ denote the Pochhammer symbols.
\end{defi}
In particular, we are now in the position to use Gauss's Hypergeometric Theorem which states that
\begin{theorem}[Gauss's Hypergeometric Theorem, {{\cite[page 1438]{Weisstein2002}}}]
$${}_2F_1(a,b,c;1)=\frac{\Gamma(c)\Gamma(c-a-b)}{\Gamma(c-a)\Gamma(c-b)}$$ for $\mathcal{R}(c-a-b)>0$.
\end{theorem}
An application of Gauss's Hypergeometric Theorem yields 
\begin{flalign*}
&-\frac{\sqrt{\pi}\Gamma(s+\frac{1}{2})\Gamma(s-1/2)}{\Gamma(2s+3/2)}\sum_{r=0}^\infty \frac{(s+\frac{1}{2})_r(s-1/2)_r}{(2s+3/2)_r}\frac{1}{r!}&\\
&\hspace{0.5cm}=-\frac{\sqrt{\pi}\Gamma(s+\frac{1}{2})\Gamma(s-1/2)}{\Gamma(2s+3/2)}{}_{2}{F}_{1}(s+1/2,s-1/2;2s+3/2,1)&\\
&\hspace{0.5cm}=-\frac{\sqrt{\pi}\Gamma(s+\frac{1}{2})\Gamma(s-1/2)}{\Gamma(2s+3/2)}\frac{\Gamma(2s+3/2)\Gamma(3/2)}{\Gamma(s+1)\Gamma(s+2)}&\\
&\hspace{0.5cm}=-\frac{\sqrt{\pi}\Gamma(s+\frac{1}{2})\Gamma(s-1/2)\Gamma(3/2)}{\Gamma(s+1)\Gamma(s+2)}.
\end{flalign*}
Using the fact that $\Gamma(3/2)=\frac{\sqrt{\pi}}{2}$, the claim follows.
\end{proof}

\section{Technical steps in the derivation of the bifurcation equation}
\subsection{Regularity of critical points and differentiability of $E$}\label{app: regularity of the EL eqt}

First we show that the interaction operator is a compact operator on the space $H^s(\SP)$ for all $s\in \R$.
\begin{lem}
\label{lem:compactInteraction}The interaction operator $U$ is a
compact operator $H^{s}(\SP)\rightarrow H^{s}(\SP)$ for all $s\in \R$.
\end{lem}
\begin{proof}
By \eqref{eq:UappliedToSH} we see that $U$ acts as a multiplication operator applied to the orthogonal basis of spherical harmonics. Due to Assumption \ref{assumption} (d), we know that its eigenvalues converge to zero. Thus, $U$ is a compact operator.
\end{proof}

The following Proposition shows that solutions to the Euler-Lagrange equation are smooth.
\begin{prop}
\label{prop:regularityCriticalPoint}Suppose that $\phi\in L^{\infty}(\SP)$
satisfies 
\[
E(\phi,\lambda)=(\lambda_{s}+\lambda)\phi(p)-\frac{1}{Z(\phi)}\int_{\mathbb{S}^{2}}k_{O}(p\cdot q)e^{-\phi(q)}\; dq=0.
\]
Then $\phi\in C^{\infty}(\SP).$\end{prop}
\begin{proof}
We rearrange the Euler-Lagrange equation as 
\[
(\lambda_{s}+\lambda)\phi(p)=\frac{1}{Z(\phi)}\int_{\mathbb{S}^{2}}k_{O}(p\cdot q)e^{-\phi(q)}\; dq.
\]
Notice that $\phi$ is bounded, see Lemma \ref{lem:boundedness lemma}. Thus, taking the limit for
a sequence $p_{n}\rightarrow p$ on the right hand side, we see that
it is continuous due to the dominated convergence theorem and so is
the left hand side. Because $k_{O}(p\cdot q)$ is differentiable almost
everywhere, except for $|p\cdot q|=1$, with an
$L^{1}$ bounded derivative, we can swap integration and differentiation
and conclude that $\phi$ is differentiable. 

Having proved that $\phi$ is continuous and differentiable, we want
to show that also the derivative of $\phi$ is continuous and differentiable.
This will then allow us to proceed by an induction argument. 

By \cite{Garrett2011b} one way to define the $C^{1}$-norm is  
\[
\left\Vert \phi\right\Vert _{C^{1}}:=\left\Vert \phi\right\Vert _{C^{0}}+\max_{1\leq i<j\leq3}\left|X_{ij}\phi\right|_{C^{0}}
\]
where 
\begin{align*}
X_{ij}\phi=\frac{d}{dt}\phi\Big(&x_{1},\dots,x_{i-1},\\
&x_{i}\cos t-x_{j}\sin t,x_{i+1},\dots,x_{j-1},x_{i}\sin t+x_{j}\cos t,x_{j+1},\dots,x_{n}\Big)
\end{align*}
are the intrinsic derivatives on $S^{n-1}$. Here we consider the
special case $n=3$. An elementary calculation shows that 
\[
X_{ij}Y_{l}^{m}(p)Y_{l}^{\star m}(q)=-Y_{l}^{m}(p)X_{ij}Y_{l}^{\star m}(q).
\]
Thus, using the expansion of $K_{O}(p,q)=\sum_{l=0}^{\infty}\sum_{m=-l}^{l}c_{l}Y_{l}^{m}(p)Y_{l}^{\star m}(q)$,
see \eqref{eq:Onsager kernel as series}, we can calculate 
\begin{align*}
 & X_{ij}\frac{1}{Z(\phi)}\int_{\mathbb{S}^{2}}K_{O}(p\cdot q)e^{-\phi(q)}\; dq\\
 & =X_{ij}\frac{1}{Z(\phi)}\int_{\mathbb{S}^{2}}\sum_{l=0}^{\infty}\sum_{m=-l}^{l}c_{l}Y_{l}^{m}(p)Y_{l}^{\star m}(q)e^{-\phi(q)}\; dq\\
 & =\sum_{l=0}^{\infty}\sum_{m=-l}^{l}c_{l}\int_{\mathbb{S}^{2}}X_{ij}Y_{l}^{m}(p)Y_{l}^{\star m}(q)e^{-\phi(q)}\; dq\\
 & =\sum_{l=0}^{\infty}\sum_{m=-l}^{l}c_{l}Y_{l}^{m}(p)\int_{\mathbb{S}^{2}}\left(-X_{ij}Y_{l}^{\star m}(q)e^{-\phi(q)}\right)\; dq\\
 & =\sum_{l=0}^{\infty}\sum_{m=-l}^{l}c_{l}Y_{l}^{m}(p)\int_{\mathbb{S}^{2}}Y_{l}^{\star m}(q)\left(-X_{ij}\phi(q)e^{-\phi(q)}\right)\; dq\\
 & =\int_{\mathbb{S}^{2}}k_{O}(p\cdot q)\left(-X_{ij}\phi(q)e^{-\phi(q)}\right)\; dq.
\end{align*}
In particular $(\lambda_{s}+\lambda)X_{ij}\phi(p)$ is continuous
and differentiable. By iteration of this argument, the statement is
proved. 
\end{proof}
The Sobolev spaces on the sphere can be defined in terms of several
equivalent ways through the coefficient in the expansion in spherical
harmonics or in local coordinates $\varphi$ and $\theta$. In terms
of $f=\sum_{l=0}^{\infty}\sum_{m=-l}^{l}a_{lm}Y_{l}^{m}$ the Sobolev
norm can be defined as 
\begin{equation}
\left\Vert f\right\Vert _{\hat{{H}^{s}}(\SP)}:=\sum_{l=0}^{\infty}\sum_{m=-l}^{l}\left|a_{lm}\right|^{2}.\label{eq:SobolevNormExpansion}
\end{equation}
In order to write down the equivalent norm in local coordinates, we
follow \cite{Hebey1996SobolevSpaceManifold} and use spherical coordinates
$\varphi$ and $\theta$ on the sphere, see Appendix \ref{app: notation spherical harmonics}. Hence the Riemannian metric
in terms of local coordinates is given by
\[
g(\varphi,\theta)=\left(\begin{array}{cc}
\sin^{2}\theta & 0\\
0 & 1
\end{array}\right)
\]
and the Christoffel symbols are given by
\[
\Gamma_{11}^{1}=0,\Gamma_{12}^{1}=\Gamma_{21}^{1}=\cot\theta,\Gamma_{22}^{1}=0,\Gamma_{11}^{2}=-\frac{1}{2}\sin2\theta,\Gamma_{12}^{2}=\Gamma_{21}^{2}=\Gamma_{22}^{2}=0.
\]
The $H^{s}(\SP)$ norm is then given by \cite{Hebey1996SobolevSpaceManifold}
\[
\left\Vert f\right\Vert _{H^{s}(\SP)}=\sum_{i=0}^{s}\left(\int\left|\nabla^{s}f\right|^{2}dp(g)\right)^{\frac{1}{2}}
\]
where $\nabla^{s}$ denotes the covariant derivative. The corresponding
norm is given by
\[
\left|\nabla^{s}f\right|^{2}=g^{i_{1}j_{1}}\dots g^{i_{s}j_{s}}\left(\nabla^{s}f\right)_{i_{1}\dots i_{s}}\left(\nabla^{s}f\right)_{j_{1}\dots j_{s}}.
\]
Subsequently, we are only interested in $\sspace$ and thus we only
need $\nabla f$ and $\nabla^{2}f$. These are given by 
\begin{align*}
\nabla f & =df\\
\left(\nabla^{2}f\right)_{ij} & =\partial_{ij}f-\Gamma_{ij}^{k}\partial_{k}u.
\end{align*}
Hence explicitly, the norm can be expressed as

\begin{align}
\left\Vert f\right\Vert _{\sspace}= & \left(\int_{A}f^{2}\sin\theta d\varphi d\theta\right)^{\frac{1}{2}}+\left(\int_{A}\left(\sin^{-2}\theta\left(\partial_{\varphi}f\right)^{2}+\left(\partial_{\theta}f\right)^{2}\right)\sin\theta d\varphi d\theta\right)^{\frac{1}{2}}\nonumber \\
 & +\left(\int_{A}\sin^{-4}\theta\left(\partial_{\varphi\varphi}f-\frac{1}{2}\sin2\theta\partial_{\theta}f\right)^{2}+\left(\partial_{\varphi\varphi}f\right)^{2}\right.\label{eq:SobolevNormSpherical}\\
 & \left.+2\sin^{-2}\theta\left(\partial_{\varphi\varphi}f\right)\left(\partial_{\varphi\varphi}f-\frac{1}{2}\sin2\theta\partial_{\theta}f\right)\sin\theta d\varphi d\theta\right)^{\frac{1}{2}}\nonumber 
\end{align}
By the Sobolev embedding theorem \cite{Garrett2011b},
\begin{equation}
\left\Vert f\right\Vert _{C^{0}}\leq C\left\Vert f\right\Vert _{\sspace}.\label{eq:SobolevEmbedding}
\end{equation}

\begin{lem}
\label{lem:differentiablityE}$E$ is infinitely many times Fr\'{e}chet
differentiable as a mapping from $\sspace\rightarrow\sspace.$\end{lem}
\begin{proof}
In order to check Fr\'{e}chet differentiability of $E$ we need to verify
that 
\[
\lim_{\left\Vert \eta\right\Vert _{\sspace}\rightarrow0}\frac{\left\Vert E\left(\phi+\eta\right)-E\left(\phi\right)+D_{\eta}E(\phi)\right\Vert _{\sspace}}{\left\Vert \eta\right\Vert _{\sspace}}=0.
\]
First we compute the first and second directional derivatives of $E$
\begin{align*}
E\left(\phi\right)(p)= & \left(\lambda_{2}+\lambda\right)\phi(p)-\frac{1}{Z}\int k(p\cdot q)e^{-\phi(q)}dq\\
\partial_{\epsilon}E\left(\phi+\epsilon\eta\right)(p)= & \left(\lambda_{2}+\lambda\right)\eta(p)-\frac{\int e^{-\phi(q)-\epsilon\eta(q)}\eta(q)dq}{\left(\int e^{-\phi(q)-\epsilon\eta(q)}dq\right)^{2}}\int k(p\cdot q)e^{-\phi(q)-\epsilon\eta(q)}dq\\
 & -\frac{1}{\int e^{-\phi(q)-\epsilon\eta(q)}dq}\int k(p\cdot q)\eta(q)e^{-\phi(q)-\epsilon\eta(q)}dq
\end{align*}
\begin{align}\label{eq:d2E}
\begin{aligned}
\partial_{\epsilon}^{2}E&\left(\phi+\epsilon\eta\right)(p)\\
=&\underbrace{\frac{1}{\int e^{-\phi(q)-\epsilon\eta(q)}dq}\int k(p\cdot q)\eta^{2}(q)e^{-\phi(q)-\epsilon\eta(q)}dq}_{T_{1}}\\
 & +\underbrace{\frac{\int e^{-\phi(q)-\epsilon\eta(q)}\eta^{2}(q)dq\left(\int e^{-\phi(q)-\epsilon\eta(q)}dq\right)^{2}}{\left(\int e^{-\phi(q)-\epsilon\eta(q)}dq\right)^{4}}\int k(p\cdot q)e^{-\phi(q)-\epsilon\eta(q)}dq}_{T_{2}} \\
 & -\underbrace{\frac{\left(\int e^{-\phi(q)-\epsilon\eta(q)}\eta(q)dq\right)^{2}2\int e^{-\phi(q)-\epsilon\eta(q)}dq}{\left(\int e^{-\phi(q)-\epsilon\eta(q)}dq\right)^{4}}\int k(p\cdot q)e^{-\phi(q)-\epsilon\eta(q)}dq}_{T_{3}} \\
 & -\underbrace{2\frac{\int e^{-\phi(q)-\epsilon\eta(q)}\eta(q)dq}{\left(\int e^{-\phi(q)-\epsilon\eta(q)}dq\right)^{2}}\int k(p\cdot q)\eta(q)e^{-\phi(q)-\epsilon\eta(q)}dq}_{T_{4}}.
 \end{aligned}
\end{align}
Thus, by the integral form of the remainder in Taylor's theorem in
$\epsilon$ we have that 
\[
E\left(\phi+\eta\right)-\left(E\left(\phi\right)+D_{\eta}E(\phi)\right)=\int_{0}^{1}\left(\partial_{\epsilon}^{2}E\right)\left(\phi+\epsilon\eta\right)(1-\epsilon)d\epsilon.
\]
The $\left\Vert \cdot\right\Vert _{\sspace}$-norm of the left hand
side can be bounded by 
\[
\int_{0}^{1}\left\Vert \left(\partial_{\epsilon}^{2}E\right)\left(\phi+\epsilon\eta\right)(1-\epsilon)\right\Vert _{H^{2}(\SP)}d\epsilon.
\]
and our aim is now to establish an overall bound of the form 
\[
\left\Vert E\left(\phi+\eta\right)-\left(E\left(\phi\right)+D_{\eta}E(\phi)\right)\right\Vert _{\sspace}\leq b(\left\Vert \phi\right\Vert _{\sspace},\left\Vert \eta\right\Vert _{\sspace})\left\Vert \eta\right\Vert _{\sspace}^{2}
\]
where $b:\mathbb{R}^{2}\rightarrow\mathbb{R}$ is locally bounded.
This would yield that $E$ is Fr\'{e}chet differentiable. In order to
show this, we treat each of the four terms occurring in \eqref{eq:d2E}
separately using the triangle inequality. In the following we sketch
the calculation for the first two terms. In case of the first term,
we have 
\begin{align*}
\left\Vert T_{1}\right\Vert _{\sspace}&=\left\Vert \frac{1}{\int e^{-\phi(q)-\epsilon\eta(q)}dq}\int k(p\cdot q)\eta^{2}(q)e^{-\phi(q)-\epsilon\eta(q)}dq\right\Vert _{H^{2}(\SP)} \\
&\leq\frac{\left\Vert U\right\Vert _{\sspace}}{4\pi}e^{\left\Vert \phi\right\Vert _{\infty}+\epsilon\left\Vert \eta\right\Vert _{\infty}}\left\Vert \eta^{2}e^{-\phi-\epsilon\eta}\right\Vert _{H^{2}(\SP)}.
\end{align*}
We now use the definition of the $\left\Vert \cdot\right\Vert _{\sspace}$-norm
given in \eqref{eq:SobolevNormSpherical} on the last of these factors
$\left\Vert \eta^{2}e^{-\phi-\epsilon\eta}\right\Vert $ in order
to show that this term can be bounded by an expression of the form
$b(\left\Vert \phi\right\Vert _{\sspace},\left\Vert \eta\right\Vert _{\sspace})\left\Vert \eta\right\Vert _{\sspace}^{2}$.
Because the expression in \eqref{eq:SobolevNormSpherical} is quite
long, we will also break it up in its individual summands and call
them $T_{1,i}$ for $i=1,\dots,3$. We start with the first one $\left(\int_{A}f^{2}\sin\theta d\varphi d\theta\right)$
for $f=\eta^{2}(\phi,\theta)\exp\left(-\phi-\epsilon\eta\right)$
(notice that we will drop the $(\varphi,\theta)$-dependence from
now on). In particular,
\begin{align*}
T_{1,1}&=\left(\int_{A}\eta^{4}e^{-2\phi-2\epsilon\eta}\sin\theta d\varphi d\theta\right)^{\frac{1}{2}} \\
&\leq Ce^{\left\Vert \phi\right\Vert _{\infty}+\epsilon\left\Vert \eta\right\Vert _{\infty}}\left(\int_{A}\eta^{4}\sin\theta d\varphi d\theta\right)^{\frac{1}{2}}\\
&\leq Ce^{\left\Vert \phi\right\Vert _{\infty}+\epsilon\left\Vert \eta\right\Vert _{\infty}}\left\Vert \eta^{4}\right\Vert _{\infty}^{\frac{1}{2}}=Ce^{\left\Vert \phi\right\Vert _{\infty}+\epsilon\left\Vert \eta\right\Vert _{\infty}}\left\Vert \eta\right\Vert _{\infty}^{2}\\
 & \leq Ce^{\left\Vert \phi\right\Vert _{\infty}+\epsilon\left\Vert \eta\right\Vert _{\infty}}\left\Vert \eta\right\Vert _{\sspace}^{2}.
\end{align*}
This gives us a bound on the term $T_{1,1}$. In case of the second
term, we have
\begin{align*}
T_{1,2}^{2}= & \int_{A}\left(\sin^{-2}\theta\left(\partial_{\varphi}f\right)^{2}+\left(\partial_{\theta}f\right)^{2}\right)\sin\theta d\varphi d\theta\\
= & \underbrace{\int_{A}\sin^{-2}\theta\left(2\eta\partial_{\varphi}\eta e^{-\phi-\epsilon\eta}+\eta^{2}\partial_{\varphi}(-\phi-\epsilon\eta)e^{-\phi-\epsilon\eta}\right)^{2}\sin\theta d\varphi d\theta}_{T_{1,2,1}}\\
 & +\underbrace{\int_{A}\left(2\eta\partial_{\theta}\eta e^{-\phi(q)-\epsilon\eta(q)}+\eta^{2}\partial_{\theta}(-\phi(q)-\epsilon\eta(q))e^{-\phi(q)-\epsilon\eta(q)}\right)^{2}\sin\theta d\varphi d\theta}_{T_{1,2,2}}
\end{align*}
Using $(a+b)^{2}\leq2\left(a^{2}+b^{2}\right)$, we obtain that 
\begin{align*}
T_{1,2,1}\leq&\int_{A}\sin^{-2}\theta\left(\left(\partial_{\varphi}\eta2\eta e^{-\phi-\epsilon\eta}\right)^{2}+\left(\eta^{2}\partial_{\varphi}(-\phi-\epsilon\eta)e^{-\phi-\epsilon\eta}\right)^{2}\right)\sin\theta d\varphi d\theta\\
 \leq &\int_{A}\sin^{-1}\theta\left(\left\Vert 2\eta e^{-\phi-\epsilon\eta}\right\Vert _{\infty}^{2}\left(\partial_{\varphi}\eta\right)^{2}+\left(\partial_{\varphi}(-\phi-\epsilon\eta)\right)^{2}\right)d\varphi d\theta\\
 \leq & 4\left\Vert \eta\right\Vert _{\infty}^{2}\left\Vert e^{-\phi-\epsilon\eta}\right\Vert _{\infty}^{2}\int_{A}\sin^{-1}\theta\left(\partial_{\varphi}\eta\right)^{2}d\varphi d\theta\\&+\left\Vert \eta\right\Vert _{\infty}^{4}\left\Vert e^{-\phi-\epsilon\eta}\right\Vert _{\infty}^{2}\int_{A}\sin^{-1}\theta\left(\partial_{\varphi}(-\phi-\epsilon\eta)\right)^{2}d\varphi d\theta\\
 \leq & C\left\Vert \eta\right\Vert _{\infty}^{2}\left\Vert e^{-\phi-\epsilon\eta}\right\Vert _{\infty}^{2}\left\Vert \eta\right\Vert _{\sspace}^{2}+C\left\Vert \eta\right\Vert _{\sspace}^{4}\left\Vert e^{-\phi-\epsilon\eta}\right\Vert _{\infty}^{2}\left\Vert \phi+\epsilon\eta\right\Vert _{\sspace}^{2}\\
 \leq & C\left\Vert \eta\right\Vert _{\sspace}^{4}\left\Vert e^{-\phi-\epsilon\eta}\right\Vert _{\infty}^{2}+C\left\Vert \eta\right\Vert _{\sspace}^{4}\left\Vert e^{-\phi-\epsilon\eta}\right\Vert _{\infty}^{2}\left\Vert \phi+\epsilon\eta\right\Vert _{\sspace}^{2}.
\end{align*}
Similarly, we obtain for $T_{1,2,2}$ that
\[
T_{1,2,2}\leq C\left\Vert \eta\right\Vert _{\sspace}^{4}\left\Vert e^{-\phi-\epsilon\eta}\right\Vert _{\infty}^{2}+C\left\Vert \eta\right\Vert _{\sspace}^{4}\left\Vert e^{-\phi-\epsilon\eta}\right\Vert _{\infty}^{2}\left\Vert \phi+\epsilon\eta\right\Vert _{\sspace}^{2}.
\]
Using $\sqrt{a+b}\le\sqrt{a}+\sqrt{b}$, we obtain for $T_{1,2}$ 

\[
T_{1,2}\leq C\left\Vert \eta\right\Vert _{\sspace}^{2}\left\Vert e^{-\phi-\epsilon\eta}\right\Vert \left(1+\left\Vert \phi+\epsilon\eta\right\Vert _{\infty}\right).
\]
{} Bounds for all other terms can be obtained similarly and will be
omitted. \end{proof}

\subsection{Computing products of spherical harmonics}\label{sec:algorithm for product of SHs}

\begin{algorithm}[H]
  \caption{\label{alg:products of spherical harmonics}Rules for the computation of products of spherical harmonics}

\begin{algorithmic}[1]
  \REQUIRE a polynomial $f$ of spherical harmonics of arbitrary degree.
  \ENSURE an expansion in terms of spherical harmonics.
  \STATE \textbf{If} $a$ is a constant, \textbf{ then define } $S[a]:=a$ and $T[a]:=a$.
  \STATE \textbf{Define } $N_{l,m}:=\sqrt{\frac{2l+1}{4\pi}}\sqrt{\frac{(l-m)!}{(l+m)!}}$.
  \STATE \textbf{Define } $S\left[Y^m_l\right]:=Y^m_l$ and $T\left[Y^m_l\right]:=Y^m_l$.
  
  \STATE \textbf{If} $a$ is a constant, \textbf{ then define } $S[a y]:=aS\left[y\right]$ and $T[a y]:=aT\left[y\right]$.
  \STATE \textbf{Define } $S\left[x+y\right]:=S\left[x\right]+S\left[y\right]$.
  \STATE \textbf{Define } $S\left[a Y^m_l Y^q_p\right]:=S\left[aS\left[Y^m_lY^q_p\right]\right]$.

  \STATE \textbf{If }$p>q$ and $l-2\geq|m|$, \textbf{ then define }
  \begin{align*}
  S\left[Y_{l}^m Y_{p}^q\right]
	:=&\frac{2l-1}{l-m}\frac{p-q+1}{2p+1}\frac{N_{l,m}}{N_{(l-1),m}}\frac{N_{p,q}}{N_{(p+1),q}}S\left[Y_{p+1}^qY^m_{l-1}\right]\\
&+\frac{2l-1}{l-m}\frac{p+q}{2p+1}\frac{N_{l,m}}{N_{(l-1),m}}\frac{N_{p,q}}{N_{(p-1),q}} S\left[Y_{p-1}^qY^m_{l-1}\right]\\
&-\frac{l-1+m}{l-m}\frac{N_{l,m}}{N_{(l-2),m}} S\left[Y_{l-2}^mY^q_p\right].
  \end{align*}	
    	\algstore{myalg}
\end{algorithmic}
\end{algorithm}	

	\textcolor{white}{?}
\algorithmstyle{plain}
	
	\begin{algorithm}[H]        
	\begin{algorithmic}[1] 	
	
	\algrestore{myalg}	 
	\STATE \textbf{If }$l=m$, \textbf{ then define }
	\begin{align*}
	S\left[Y_{l}^mY^q_p\right]:=(-1)^lN_{l,m}(2l-1)!!S\left[(1-x^2)^{l/2}Y^q_p\right].
	\end{align*}

	\STATE \textbf{If }$l=m+1$, \textbf{ then define}
	\begin{align*}
	S\left[Y^m_lY_{p}^q\right]:=N_{l,m}(2l-1)(-1)^{l-1}(2l-3)!! S\left[(1-x^2)^{\frac{l-1}{2}}xY_{p}^{q}\right].
	\end{align*}	
	
	\STATE \textbf{If }$l=-m$, \textbf{ then define }
	\begin{align*}
	S\left[Y_{l}^m Y_{p}^q\right]:=\frac{N_{l,m}}{(2l)!}(2l-1)!!T\left[(1-x^2)^{l/2}Y^q_p\right].
  \end{align*}
  
\STATE \textbf{If }$l=-m+1$, \textbf{ then define }
	\begin{align*}
	S\left[Y_{l}^m Y_{p}^q\right]:=\frac{N_{l,m}}{(2l-2)!}(-1)^{m+l-1}(2l-3)!!T\left[(1-x^2)^{\frac{l-1}{2}}xY^q_p\right].
  \end{align*}	
	\STATE \textbf{If }$l>|m|$, \textbf{ then define }
	\begin{align*}
	S\left[a\cdot xY^m_l\right]:=\frac{N_{l,m}}{N_{(l+1),m}}S\left[a\frac{l-m+1}{2l+1}Y^m_{l+1}\right]+\frac{N_{l,m}}{N_{(l-1),m}}S\left[a\frac{l+m}{2l+1}Y_{l-1}^m\right].
	\end{align*}	
	\STATE \textbf{If }$l=|m|$, \textbf{ then define }
	\begin{align*}
	S\left[a\cdot xY^m_l\right]:=\frac{N_{l,m}}{N_{(l+1),m}}S\left[a\frac{l-m+1}{2l+1}Y^m_{l+1}\right].
	\end{align*}	
	\STATE \textbf{If }$l>m+2$, \textbf{ then define }
	\begin{align*}
	S\left[(1-x^2)^s Y^m_l\right]:=&\frac{1}{2l+1}\frac{N_{l,m}}{N_{(l+1),(m+1)}}S\left[(1-x^2)^{s-1/2}Y^{m+1}_{l-1}\right]\\&-\frac{1}{2l+1}\frac{N_{l,m}}{N_{(l+1),(m+1)}}S\left[(1-x^2)^{s-1/2}Y^{m+1}_{l+1}\right].
	\end{align*}		
	\STATE \textbf{If }$l=m$ or $l=m+1$, \textbf{ then define }
	\begin{align*}
	S\left[(1-x^2)^s Y^m_l\right]:=-\frac{1}{2l+1}\frac{N_{l,m}}{N_{(l+1),(m+1)}}S\left[(1-x^2)^{s-1/2}Y^{m+1}_{l+1}\right].
	\end{align*}	
	\STATE \textbf{Define } $S\left[a(Y^m_l)^s\right]:=S\left[a(Y^m_l)^{s-2}S\left[(Y_l^m)^2\right]\right]$.
	\STATE \textbf{If }$l-2>|m|$, \textbf{ then define }
	\begin{align*}
	S\left[(Y_l^m)^2\right]:=&\frac{2l-1}{l-m}\frac{l-m+1}{2l+1}\frac{N_{l,m}}{N_{(l-1),m}}\frac{N_{l,m}}{N_{(l+1),m}}S\left[Y^m_{l-1}Y^m_{l+1}\right]\\
	&+\frac{2l-1}{l-m}\frac{l+m}{2l+1}\left(\frac{N_{l,m}}{N_{(l-1),m}}\right)^2S\left[Y^m_{l-1}Y^m_{l-1}\right]\\
	&-\frac{l+m-1}{l-m}\frac{N_{l,m}}{N_{(l-2),m}}S\left[Y^m_{l-2}Y^m_l\right].
\end{align*}
	\STATE \textbf{If }$l=m$, \textbf{ then define }
	$$S\left[(Y_l^m)^2\right]:=N_{l,m}(-1)^l(2l-1)!!S\left[(1-x^2)^{l/2}Y^m_l\right].$$
	\STATE \textbf{If }$l=-m$, \textbf{ then define }
	$$S\left[(Y_l^m)^2\right]:=N_{l,-l}((2l)!)^{-1}(2l-1)!!T\left[(1-x^2)^{l/2}Y^m_l\right].$$
	\STATE \textbf{If }$l=m+1$, \textbf{ then define }
	$$S\left[(Y_l^m)^2\right]:=N_{l,m}(-1)^{l-1}(2l-1)(2l-3)!!S\left[(1-x^2)^{\frac{l-1}{2}}xY^m_l\right].$$
		\algstore{myalg}
	\end{algorithmic}
	\end{algorithm}	
	\algorithmstyle{plain}
	\textcolor{white}{?}
	\begin{algorithm}[H]        
	\begin{algorithmic}[1] 	
		\algrestore{myalg}	 

	\STATE \textbf{If }$l=-m+1$, \textbf{ then define }
	\begin{align*}
		S&\left[(Y_l^m)^2\right]\\&:=N_{l,m}(-1)^{m+l-1}((2l-1)!)^{-1}(2l-1)(2l-3)!!T\left[(1-x^2)^{\frac{l-1}{2}}xY^m_l\right].
		\end{align*}
	\STATE \textbf{If }$l>|m|$, \textbf{ then define }
	\begin{align*}	
	S\left[a\cdot xY_l^m\right]:=\frac{N_{l,m}}{N_{(l+1),m}}\frac{l-m+1}{2l+1}T\left[aY^m_{l+1}\right]+\frac{N_{l,m}}{N_{(l-1),m}}\frac{l+m}{2l+1}T\left[aY^m_{l-1}\right].
	\end{align*}

	\STATE \textbf{If }$l=|m|$, \textbf{ then define }
	\begin{align*}	
	S\left[a\cdot xY_l^m\right]:=\frac{N_{l,m}}{N_{(l+1),m}}\frac{l-m+1}{2l+1}T\left[aY^m_{l+1}\right].
	\end{align*}		
	\STATE \textbf{If }$l\geq -m+2$, \textbf{ then define }
	\begin{align*}	
	T\left[(1-x^2)^{s}Y_l^m\right]:=&\frac{(l-m+1)(l-m+2)N_{l,m}}{(2l+1)N_{(l+1),(m-1)}}T\left[(1-x^2)^{s-1/2}Y^{m-1}_{l+1}\right]\\
	&-\frac{(l+m-1)(l+m)N_{l,m}}{(2l+1)N_{(l-1),(m-1)}}T\left[(1-x^2)^{s-1/2}Y^{m-1}_{l-1}\right].
	\end{align*}	
	\STATE \textbf{If }$l= -m+1$ or $l=-m$, \textbf{ then define }
	\begin{align*}	
	T\left[(1-x^2)^{s}Y_l^m\right]:=\frac{(l-m+1)(l-m+2)N_{l,m}}{(2l+1)N_{(l+1),(m-1)}}T\left[(1-x^2)^{s-1/2}Y^{m-1}_{l+1}\right].
	\end{align*}	
	\end{algorithmic}
	\end{algorithm}
	
\subsection{The bifurcation equation in terms of real spherical harmonics}\label{app: real bifurcation equation}
The bifurcation equation in terms of real spherical harmonics is given by
\begin{align*}
f_{\text{real}}(u,&\lambda)(p)=\sum_{m=-2}^2f_{m}(u,\lambda) Y_{2,m}(p).
\end{align*}
with
\begin{align*}
f_{-2}=&\lambda  u_{-2}+\frac{1}{448} \sqrt{\frac{15 \pi }{2}} u_{-1}^2-\frac{1}{224} \sqrt{5\pi } u_{-2} u_0-\frac{(-1+1120 \pi ) u_{-2} u_0^2}{12544 (1+32 \pi )}\\&
+\frac{18\lambda  u_{-2} u_0^2}{49 \pi  (1+32 \pi )^2}+\frac{\sqrt{\frac{15}{2 \pi }}\left(89-1792 \pi +250880 \pi ^2\right) u_{-1}^2 u_0^2}{1931776 (1+32 \pi
   )^2}\\&
   -\frac{\sqrt{\frac{5}{\pi }} \left(1-98 \pi +4480 \pi ^2\right) u_{-2}
   u_0^3}{25872 (1+32 \pi )^2}+\frac{(-1+1120 \pi ) u_{-2} u_{-1} u_1}{6272 (1+32 \pi)}\\&
 -\frac{36 \lambda  u_{-2} u_{-1} u_1}{49 \pi  (1+32 \pi)^2}-\frac{\sqrt{\frac{15}{2 \pi }} \left(89-1792 \pi +250880 \pi ^2\right) u_{-1}^3
   u_1}{965888 (1+32 \pi )^2}\\&
   +\frac{\sqrt{\frac{5}{\pi }} \left(23-9184 \pi +250880 \pi^2\right) u_{-2} u_{-1} u_0 u_1}{965888 (1+32 \pi )^2}\\&
   +\frac{5 \sqrt{\frac{15}{2 \pi
   }} \left(31+1120 \pi +50176 \pi ^2\right) u_{-2}^2 u_1^2}{965888 (1+32 \pi
   )^2}-\frac{(-1+1120 \pi ) u_{-2}^2 u_2}{6272 (1+32 \pi )}\\&
   +\frac{36 \lambda  u_{-2}^2u_2}{49 \pi  (1+32 \pi )^2}+\frac{\sqrt{\frac{15}{2 \pi }} \left(61+952 \pi +125440\pi ^2\right) u_{-2} u_{-1}^2 u_2}{241472 (1+32 \pi )^2}\\&-\frac{\sqrt{\frac{5}{\pi }}\left(61+952 \pi +125440 \pi ^2\right) u_{-2}^2 u_0 u_2}{120736 (1+32 \pi )^2},
\end{align*}

\begin{align*}
f_{-1}=&\lambda  u_{-1}+\frac{1}{448} \sqrt{5 \pi } u_{-1} u_0-\frac{(-1+1120 \pi ) u_{-1}
   u_0^2}{12544 (1+32 \pi )}+\frac{18 \lambda  u_{-1} u_0^2}{49 \pi  (1+32 \pi
   )^2}\\&
   +\frac{\sqrt{\frac{5}{\pi }} \left(577+5824 \pi +1254400 \pi ^2\right) u_{-1}
   u_0^3}{5795328 (1+32 \pi )^2}-\frac{1}{224} \sqrt{\frac{15 \pi }{2}} u_{-2}
   u_1\\
   &+\frac{(-1+1120 \pi ) u_{-1}^2 u_1}{6272 (1+32 \pi )}-\frac{36 \lambda  u_{-1}^2
   u_1}{49 \pi  (1+32 \pi )^2}\\
   &-\frac{\sqrt{\frac{5}{\pi }} \left(61+952 \pi +125440 \pi^2\right) u_{-1}^2 u_0 u_1}{241472 (1+32 \pi )^2}\\&
   -\frac{\sqrt{\frac{15}{2 \pi }}\left(89-1792 \pi +250880 \pi ^2\right) u_{-2} u_0^2 u_1}{965888 (1+32 \pi)^2}\\&
   +\frac{3 \sqrt{\frac{15}{2 \pi }} \left(111+672 \pi +250880 \pi ^2\right) u_{-2}u_{-1} u_1^2}{965888 (1+32 \pi )^2}-\frac{(-1+1120 \pi ) u_{-2} u_{-1} u_2}{6272
   (1+32 \pi )}\\
   &+\frac{36 \lambda  u_{-2} u_{-1} u_2}{49 \pi  (1+32 \pi )^2}+\frac{5
   \sqrt{\frac{15}{2 \pi }} \left(31+1120 \pi +50176 \pi ^2\right) u_{-1}^3 u_2}{965888(1+32 \pi )^2}\\
   &-\frac{\sqrt{\frac{5}{\pi }} \left(221+12992 \pi +250880 \pi ^2\right)
   u_{-2} u_{-1} u_0 u_2}{965888 (1+32 \pi )^2}\\&
   -\frac{\sqrt{\frac{15}{2 \pi }}
   \left(89-1792 \pi +250880 \pi ^2\right) u_{-2}^2 u_1 u_2}{482944 (1+32 \pi )^2},
   \end{align*}
\begin{align*}
f_0=&\lambda  u_0+\frac{1}{448} \sqrt{5 \pi } u_0^2-\frac{(-1+1120 \pi ) u_0^3}{12544 (1+32\pi )}+\frac{18 \lambda  u_0^3}{49 \pi  (1+32 \pi )^2}\\&
   +\frac{\sqrt{\frac{5}{\pi }}
   \left(577+5824 \pi +1254400 \pi ^2\right) u_0^4}{5795328 (1+32 \pi
   )^2}-\frac{1}{448} \sqrt{5 \pi } u_{-1} u_1\\&
   +\frac{(-1+1120 \pi ) u_{-1} u_0
   u_1}{6272 (1+32 \pi )}-\frac{36 \lambda  u_{-1} u_0 u_1}{49 \pi  (1+32 \pi
   )^2}\\&
   -\frac{\sqrt{\frac{5}{\pi }} \left(577+5824 \pi +1254400 \pi ^2\right) u_{-1}
   u_0^2 u_1}{1931776 (1+32 \pi )^2}\\&
   +\frac{\sqrt{\frac{5}{\pi }} \left(89-1792 \pi+250880 \pi ^2\right) u_{-1}^2 u_1^2}{965888 (1+32 \pi )^2}\\&
   +\frac{5\sqrt{\frac{15}{2 \pi }} \left(31+1120 \pi +50176 \pi ^2\right) u_{-2} u_0
   u_1^2}{965888 (1+32 \pi )^2}-\frac{1}{224} \sqrt{5 \pi } u_{-2} u_2\\&
   -\frac{(-1+1120\pi ) u_{-2} u_0 u_2}{6272 (1+32 \pi )}+\frac{36 \lambda  u_{-2} u_0 u_2}{49 \pi (1+32 \pi )^2}\\&
   +\frac{5 \sqrt{\frac{15}{2 \pi }} \left(31+1120 \pi +50176 \pi^2\right) u_{-1}^2 u_0 u_2}{965888 (1+32 \pi )^2}
  \end{align*}
\begin{align*}   
   &
   -\frac{5 \sqrt{\frac{5}{\pi }}\left(31+1120 \pi +50176 \pi ^2\right) u_{-2} u_0^2 u_2}{482944 (1+32 \pi)^2}\\&
   +\frac{\sqrt{\frac{5}{\pi }} \left(89-1792 \pi +250880 \pi ^2\right) u_{-2}u_{-1} u_1 u_2}{965888 (1+32 \pi )^2}\\&
   -\frac{\left(-12365-2052416 \pi -12020736 \pi^2+1040449536 \pi ^3\right) u_{-2} u_{-1} u_0 u_1 u_2}{1997632 \pi  (1+32 \pi )^3(5+512 \pi )}\\&
   -\frac{\sqrt{\frac{5}{\pi }} \left(89-1792 \pi +250880 \pi ^2\right)u_{-2}^2 u_2^2}{482944 (1+32 \pi )^2},
 \end{align*}
\begin{align*}
f_1=&\lambda  u_1+\frac{1}{448} \sqrt{5 \pi } u_0 u_1-\frac{(-1+1120 \pi ) u_0^2 u_1}{12544(1+32 \pi )}+\frac{18 \lambda  u_0^2 u_1}{49 \pi  (1+32 \pi)^2}\\&
+\frac{\sqrt{\frac{5}{\pi }} \left(577+5824 \pi +1254400 \pi ^2\right) u_0^3u_1}{5795328 (1+32 \pi )^2}+\frac{(-1+1120 \pi ) u_{-1} u_1^2}{6272 (1+32 \pi)}
    \\
&-\frac{36 \lambda  u_{-1} u_1^2}{49 \pi  (1+32 \pi )^2}-\frac{\sqrt{\frac{5}{\pi}} \left(61+952 \pi +125440 \pi ^2\right) u_{-1} u_0 u_1^2}{241472 (1+32 \pi)^2}\\&
+\frac{5 \sqrt{\frac{15}{2 \pi }} \left(31+1120 \pi +50176 \pi ^2\right) u_{-2}u_1^3}{965888 (1+32 \pi )^2}-\frac{1}{224} \sqrt{\frac{15 \pi }{2}} u_{-1}u_2\\
&-\frac{\sqrt{\frac{15}{2 \pi }} \left(89-1792 \pi +250880 \pi ^2\right) u_{-1}u_0^2 u_2}{965888 (1+32 \pi )^2}-\frac{(-1+1120 \pi ) u_{-2} u_1 u_2}{6272 (1+32 \pi)}\\&
+\frac{36 \lambda  u_{-2} u_1 u_2}{49 \pi  (1+32 \pi )^2}+\frac{3\sqrt{\frac{15}{2 \pi }} \left(111+672 \pi +250880 \pi ^2\right) u_{-1}^2 u_1u_2}{965888 (1+32 \pi )^2}\\&
-\frac{\sqrt{\frac{5}{\pi }} \left(221+12992 \pi +250880\pi ^2\right) u_{-2} u_0 u_1 u_2}{965888 (1+32 \pi )^2}\\&
-\frac{\sqrt{\frac{15}{2 \pi}} \left(89-1792 \pi +250880 \pi ^2\right) u_{-2} u_{-1} u_2^2}{482944 (1+32 \pi)^2},
 \end{align*}
 and
\begin{align*}
f_2=&\lambda  u_2+\frac{1}{448} \sqrt{\frac{15 \pi }{2}} u_1^2\frac{\sqrt{\frac{15}{2 \pi }}\left(89-1792 \pi +250880 \pi ^2\right) u_0^2 u_1^2}{1931776 (1+32 \pi)^2}\\&
-\frac{\sqrt{\frac{15}{2 \pi }} \left(89-1792 \pi +250880 \pi ^2\right) u_{-1}u_1^3}{965888 (1+32 \pi )^2}-\frac{1}{224} \sqrt{5 \pi } u_0u_2\\&
-\frac{(-1+1120 \pi ) u_0^2 u_2}{12544 (1+32 \pi )}+\frac{18 \lambda  u_0^2u_2}{49 \pi  (1+32 \pi )^2}-\frac{\sqrt{\frac{5}{\pi }} \left(1-98 \pi +4480 \pi^2\right) u_0^3 u_2}{25872 (1+32 \pi )^2}\\&
+\frac{(-1+1120 \pi ) u_{-1} u_1 u_2}{6272(1+32 \pi )}-\frac{36 \lambda  u_{-1} u_1 u_2}{49 \pi  (1+32 \pi)^2}\\&
+\frac{\sqrt{\frac{5}{\pi }} \left(23-9184 \pi +250880 \pi ^2\right) u_{-1} u_0u_1 u_2}{965888 (1+32 \pi )^2}
  \end{align*}
\begin{align*}
\\&
+\frac{\sqrt{\frac{15}{2 \pi }} \left(61+952 \pi+125440 \pi ^2\right) u_{-2} u_1^2 u_2}{241472 (1+32 \pi )^2}-\frac{(-1+1120 \pi )u_{-2} u_2^2}{6272 (1+32 \pi )}\\&
+\frac{36 \lambda  u_{-2} u_2^2}{49 \pi  (1+32 \pi)^2}+\frac{5 \sqrt{\frac{15}{2 \pi }} \left(31+1120 \pi +50176 \pi ^2\right)u_{-1}^2 u_2^2}{965888 (1+32 \pi )^2}\\&
-\frac{\sqrt{\frac{5}{\pi }} \left(61+952 \pi+125440 \pi ^2\right) u_{-2} u_0 u_2^2}{120736 (1+32 \pi )^2}.
\end{align*}

\subsection{Properties of the Bifurcation Equation}\label{sec:properties of the bifurcation equation}
\begin{lem}\label{lem: smoothness of the bifurcation equation}
The bifurcation equation $f(u,\lambda)$ associated to the Euler-Lagrange equation for the Onsager free-energy functional is smooth.
\end{lem}
\begin{proof}
The bifurcation equation is defined by
$$P\mathcal{R}(u+v(u,\lambda))=f(u,\lambda)$$
where $\mathcal{R}$ is given by
$$E(\rho,\lambda)=\mathcal{L}(\rho,\lambda)+\mathcal{R}(\rho,\lambda).$$
Since $\mathcal{L}$ is a linear operator, the smoothness of $E$ and $\mathcal{R}$ is the same. Moreover, the implicit function theorem tells us, that the smoothness of $R$ and $v$ is the same. Since the smoothness of the bifurcation equation is given by the minimal number of derivatives of $\mathcal{R}$ and $v$ (using the chain rule) which are then equal, we conclude that the smoothness of $f$ is the same as the smoothness of $E$ (the Euler-Lagrange equation). The Euler-Lagrange operator is infinitely many times differentiable, see Proposition \ref{prop:regularityCriticalPoint}.
\end{proof}

\section{Details of the dimension reduction based on invariant theory}
\subsection{The Wigner rotation matrix for degree $l=2$ spherical harmonics}\label{app: complex Wigner matrix}
The Wigner rotation matrix for complex spherical harmonics of degree $l=2$ is given by
\begin{align*}
D(\phi_R,\psi_R,\theta_R)=
\left(\begin{array}{ccccc}	d_{11}&d_{12}&d_{13}&d_{14}&d_{15}\\d_{21}&d_{22}&d_{23}&d_{24}&d_{25}\\	d_{31}&d_{32}&d_{33}&d_{34}&d_{35}\\d_{41}&d_{42}&d_{43}&d_{44}&d_{45}\\
	d_{51}&d_{52}&d_{53}&d_{54}&d_{55}\\
\end{array}
\right)
\end{align*}
with
\begin{align*}
\begin{array}{cl}
d_{11}&=e^{-2 i\phi_R-2 i\psi_R} \cos ^4\left(\frac{\theta_R}{2}\right),\\
d_{12}&=2 e^{-2 i \phi_R-i \psi_R} \cos ^3\left(\frac{\theta_R}{2}\right) \sin \left(\frac{\theta_R}{2}\right),\\
d_{13}&= \sqrt{6} e^{-2 i \phi_R} \cos ^2\left(\frac{\theta_R}{2}\right) \sin ^2\left(\frac{\theta_R}{2}\right),\\
d_{14}&= 2 e^{i\psi_R-2 i\phi_R} \cos \left(\frac{\theta_R}{2}\right) \sin ^3\left(\frac{\theta_R}{2}\right),\\
d_{15}&= e^{2 i\psi_R-2 i\phi_R} \sin ^4\left(\frac{\theta_R}{2}\right),\\
d_{21}&= -2 e^{-i \phi_R-2 i \psi_R} \cos ^3\left(\frac{\theta_R}{2}\right) \sin \left(\frac{\theta_R }{2}\right),\\
d_{22}&=e^{-i\phi_R-i\psi_R} \cos ^2\left(\frac{\theta_R}{2}\right) (2\cos (\theta_R)-1),\\
d_{23}&=\sqrt{6} e^{-i\phi_R} \cos \left(\frac{\theta_R}{2}\right) \cos (\theta_R) \sin \left(\frac{\theta_R}{2}\right),\\
d_{24}&=e^{i\psi_R-i\phi_R} (2 \cos (\theta_R)+1) \sin ^2\left(\frac{\theta_R}{2}\right),\\
d_{25}&=2 e^{2 i\psi_R-i\phi_R} \cos \left(\frac{\theta_R}{2}\right) \sin ^3\left(\frac{\theta_R}{2}\right),\\
d_{31}&= \sqrt{6} e^{-2 i\psi_R} \cos ^2\left(\frac{\theta_R}{2}\right) \sin ^2\left(\frac{\theta_R}{2}\right),\\
d_{32}&=-\sqrt{6} e^{-i\psi_R} \cos \left(\frac{\theta_R}{2}\right) \cos (\theta_R) \sin \left(\frac{\theta_R}{2}\right),\\
d_{33}&=\left(\frac{3}{2}\cos ^2(\theta_R )-\frac{1}{2}\right),\\
d_{34}&=\sqrt{6} e^{i\psi_R} \cos \left(\frac{\theta_R}{2}\right) \cos (\theta_R) \sin \left(\frac{\theta_R}{2}\right),\\
d_{35}&=\sqrt{6} e^{2 i\psi_R} \cos ^2\left(\frac{\theta_R}{2}\right) \sin ^2\left(\frac{\theta_R}{2}\right),\\
d_{41}&=-2 e^{i\phi_R-2 i\psi_R} \cos \left(\frac{\theta_R}{2}\right) \sin ^3\left(\frac{\theta_R}{2}\right),\\
d_{42}&=e^{i\phi_R-i\psi_R} (2\cos (\theta_R)+1) \sin ^2\left(\frac{\theta_R}{2}\right),\\ 
d_{43}&=-\sqrt{6} e^{i\phi_R} \cos \left(\frac{\theta_R}{2}\right) \cos (\theta_R) \sin \left(\frac{\theta_R}{2}\right),\\
d_{44}&=e^{i\phi_R+i\psi_R} \cos ^2\left(\frac{\theta_R}{2}\right) (2\cos (\theta_R)-1),\\ 
d_{45}&=2 e^{i \phi_R+2 i\psi_R} \cos ^3\left(\frac{\theta_R}{2}\right) \sin \left(\frac{\theta_R}{2}\right), \\
d_{51}&=e^{2 i\phi_R-2 i\psi_R} \sin ^4\left(\frac{\theta_R}{2}\right),\\
d_{52}&=-2 e^{2 i\phi_R-i\psi_R} \cos \left(\frac{\theta_R}{2}\right) \sin ^3\left(\frac{\theta_R}{2}\right),\\
d_{53}&=\sqrt{6} e^{2 i\phi_R} \cos ^2\left(\frac{\theta_R}{2}\right) \sin ^2\left(\frac{\theta_R}{2}\right),\\
d_{54}&=-2 e^{2 i\phi_R+i\psi_R} \cos ^3\left(\frac{\theta_R}{2}\right) \sin \left(\frac{\theta_R}{2}\right),\\ 
d_{55}&= e^{2 i\phi_R+2 i\psi_R} \cos ^4\left(\frac{\theta_R}{2}\right).
\end{array}
\end{align*}
where $\phi_R$,  $\psi_R$ and $\theta_R$ denote the Euler angles corresponding to a rotation.

\subsection{Representation of the group action for real spherical harmonics}\label{app: real Wigner matrix}
An application of the mapping given in \eqref{eq: mapping to real SH} yields an explicit form of the representation of the group action of $SO(3)$ acting on $\R^5$, see \eqref{eqt: definition matrix M - real Wigner matrix},
$$M(\phi_R,\psi_R,\theta_R)=\left(\begin{array}{ccccc}
	M_{11}&M_{12}&M_{13}&M_{14}&M_{15}\\M_{21}&M_{22}&M_{23}&M_{24}&M_{25}\\
	M_{31}&M_{32}&M_{33}&M_{34}&M_{35}\\M_{41}&M_{42}&M_{43}&M_{44}&M_{45}\\
	M_{51}&M_{52}&M_{53}&M_{54}&M_{55}
\end{array}\right)$$
with
\begin{align*}
\begin{array}{cl}
M_{11}= & \cos(\theta_{R})\cos(2\psi_{R})\cos(2\phi_{R})-\frac{1}{4}(\cos(2\theta_{R})+3)\sin(2\psi_{R})\sin(2\phi_{R}),\\
M_{12}= & \sin(\theta_{R})(\cos(\theta_{R})\sin(2\psi_{R})\sin(\phi_{R})-\cos(2\psi_{R})\cos(\phi_{R})),\\
M_{13}= & -\sqrt{3}\sin^{2}(\theta_{R})\sin(\psi_{R})\cos(\psi_{R}),\\
M_{14}= & \sin(\theta_{R})(\cos(\theta_{R})\sin(2\psi_{R})\cos(\phi_{R})+\cos(2\psi_{R})\sin(\phi_{R})),\\
M_{15}= & -\cos(\theta_{R})\cos(2\psi_{R})\sin(2\phi_{R})-\frac{1}{4}(\cos(2\theta_{R})+3)\sin(2\psi_{R})\cos(2\phi_{R}),\\
M_{21}= & \sin(\theta_{R})(\cos(\psi_{R})\cos(2\phi_{R})-2\cos(\theta_{R})\sin(\psi_{R})\sin(\phi_{R})\cos(\phi_{R})),\\
M_{22}= & \cos(\theta_{R})\cos(\psi_{R})\cos(\phi_{R})-\cos(2\theta_{R})\sin(\psi_{R})\sin(\phi_{R}),\\
M_{23}= & \sqrt{3}\sin(\theta_{R})\cos(\theta_{R})\sin(\psi_{R}),\\
M_{24}= & -\cos(\theta_{R})\cos(\psi_{R})\sin(\phi_{R})-\cos(2\theta_{R})\sin(\psi_{R})\cos(\phi_{R}),\\
M_{25}= & \sin(\theta_{R})(-\cos(\psi_{R}))\sin(2\phi_{R})-\frac{1}{2}\sin(2\theta_{R})\sin(\psi_{R})\cos(2\phi_{R}),\\
M_{31}= & \sqrt{3}\sin^{2}(\theta_{R})\sin(\phi_{R})\cos(\phi_{R}),\\
M_{32}= & \sqrt{3}\sin(\theta_{R})\cos(\theta_{R})\sin(\phi_{R}),\\
M_{33}= & \frac{1}{4}(3\cos(2\theta_{R})+1),\\
M_{34}= & \sqrt{3}\sin(\theta_{R})\cos(\theta_{R})\cos(\phi_{R}),\\
M_{35}= & \frac{1}{2}\sqrt{3}\sin^{2}(\theta_{R})\cos(2\phi_{R}),\\
M_{41}= & \sin(\theta_{R})(\cos(\theta_{R})\cos(\psi_{R})\sin(2\phi_{R})+\sin(\psi_{R})\cos(2\phi_{R})),\\
M_{42}= & \cos(2\theta_{R})\cos(\psi_{R})\sin(\phi_{R})+\cos(\theta_{R})\sin(\psi_{R})\cos(\phi_{R}),\\
M_{43}= & -\sqrt{3}\sin(\theta_{R})\cos(\theta_{R})\cos(\psi_{R})\\
M_{44}= & \cos(2\theta_{R})\cos(\psi_{R})\cos(\phi_{R})-\cos(\theta_{R})\sin(\psi_{R})\sin(\phi_{R}),\\
M_{45}= & \frac{1}{2}\sin(2\theta_{R})\cos(\psi_{R})\cos(2\phi_{R})-\sin(\theta_{R})\sin(\psi_{R})\sin(2\phi_{R}),\\
M_{51}= & \frac{1}{4}(\cos(2\theta_{R})+3)\cos(2\psi_{R})\sin(2\phi_{R})+\cos(\theta_{R})\sin(2\psi_{R})\cos(2\phi_{R}),\\
M_{52}= & -\frac{1}{2}\sin(2\theta_{R})\cos(2\psi_{R})\sin(\phi_{R})-\sin(\theta_{R})\sin(2\psi_{R})\cos(\phi_{R}),\\
M_{53}= & \frac{1}{2}\sqrt{3}\sin^{2}(\theta_{R})\cos(2\psi_{R}),\\
M_{54}= & \sin(\theta_{R})\sin(2\psi_{R})\sin(\phi_{R})-\frac{1}{2}\sin(2\theta_{R})\cos(2\psi_{R})\cos(\phi_{R}),\\
M_{55}= & \frac{1}{4}(\cos(2\theta_{R})+3)\cos(2\psi_{R})\cos(2\phi_{R})-\cos(\theta_{R})\sin(2\psi_{R})\sin(2\phi_{R}).
\end{array}
\end{align*}

\subsection{The explicit form of the Cartan representation}\label{app: Cartan representation}
The explicit form of the $5\times 5$ matrix representing the Cartan representation is given by
$$M_C(\phi_R,\psi_R,\theta_R)=\left(\begin{array}{ccccc}
	C_{11}&C_{12}&C_{13}&C_{14}&C_{15}\\C_{21}&C_{22}&C_{23}&C_{24}&C_{25}\\
	C_{31}&C_{32}&C_{33}&C_{34}&C_{35}\\C_{41}&C_{42}&C_{43}&C_{44}&C_{45}\\
	C_{51}&C_{52}&C_{53}&C_{54}&C_{55}
\end{array}\right)$$
with
\begin{align*}
\begin{array}{cl}
C_{11}= & \frac{1}{4}\left(-4\cos^{2}(\phi_{R})\cos^{2}(\psi_{R})\sin^{2}(\theta_{R})+(\cos(2\theta_{R})+2\cos(2\phi_{R})-1)\cos(2\psi_{R})\right)\\
 & +2\cos(\theta_{R})(\cos(\theta_{R})-4\cos(\phi_{R})\cos(\psi_{R})\sin(\phi_{R})\sin(\psi_{R})),\\
C_{12}= & 2\cos(\theta_{R})\cos(\psi_{R})\sin(\phi_{R})\sin(\psi_{R})\cos(\phi_{R})+\cos(2\theta_{R})\cos^{2}(\psi_{R})\sin^{2}(\phi_{R})\\
 & +\left(\sin^{2}(\psi_{R})-\cos^{2}(\psi_{R})\sin^{2}(\theta_{R})\right)\cos^{2}(\phi_{R}),\\
C_{13}= & \frac{1}{4}\left((\cos(2\theta_{R})+3)\cos(2\psi_{R})-2\sin^{2}(\theta_{R})\right)\sin(2\phi_{R})\\&+\cos(\theta_{R})\cos(2\phi_{R})\sin(2\psi_{R}),\\
C_{14}= & 2\cos(\psi_{R})\sin(\theta_{R})(\sin(\phi_{R})\sin(\psi_{R})-\cos(\theta_{R})\cos(\phi_{R})\cos(\psi_{R})),\\
C_{15}= & -2\cos(\psi_{R})\sin(\theta_{R})(\cos(\theta_{R})\cos(\psi_{R})\sin(\phi_{R})+\cos(\phi_{R})\sin(\psi_{R})),\\
C_{21}= & \cos^{2}(\psi_{R})\sin^{2}(\theta_{R})\sin^{2}(\phi_{R})+\frac{1}{2}\cos(\theta_{R})\sin(2\phi_{R})\sin(2\psi_{R})\\
 & +\frac{1}{2}\left(\cos(2\theta_{R})+\cos(2\psi_{R})\left(\sin^{2}(\theta_{R})-\cos^{2}(\theta_{R})\cos(2\phi_{R})\right)\right),
 \end{array}
\end{align*}
\begin{align*}
\begin{array}{cl}
C_{22}= & -\cos(\phi_{R})\cos(\psi_{R})\sin(\phi_{R})\sin(\psi_{R})\cos^{3}(\theta_{R})\\
 & +\frac{1}{2}(\cos(2\phi_{R})\cos(2\psi_{R})+1)\cos^{2}(\theta_{R})\\
 & +\frac{1}{8}(\cos(2\theta_{R})-3)\sin(2\phi_{R})\sin(2\psi_{R})\cos(\theta_{R})\\
 & +\sin^{2}(\theta_{R})\left(\cos^{2}(\phi_{R})\cos(2\psi_{R})-\sin^{2}(\phi_{R})\sin^{2}(\psi_{R})\right),\\
C_{23}= & -2\cos(\phi_{R})\cos^{2}(\psi_{R})\sin(\phi_{R})\sin^{2}(\theta_{R})\\
 & -\cos(\theta_{R})(\cos(\theta_{R})\cos(2\psi_{R})\sin(2\phi_{R})+\cos(2\phi_{R})\sin(2\psi_{R})),\\
C_{24}= & -2\sin(\theta_{R})\sin(\psi_{R})(\cos(\psi_{R})\sin(\phi_{R})+\cos(\theta_{R})\cos(\phi_{R})\sin(\psi_{R})),\\
C_{25}= & 2\sin(\theta_{R})\sin(\psi_{R})(\cos(\phi_{R})\cos(\psi_{R})-\cos(\theta_{R})\sin(\phi_{R})\sin(\psi_{R})),\\
C_{31}= & \frac{1}{8}\Big(-4\cos(\theta_{R})\cos(2\psi_{R})\sin(2\phi_{R})\\&-\left((\cos(2\theta_{R})+3)\cos(2\phi_{R})-6\sin^{2}(\theta_{R})\right)\sin(2\psi_{R})\Big),\\
C_{32}= & \frac{1}{8}\Big(4\cos(\theta_{R})\cos(2\psi_{R})\sin(2\phi_{R})\\&+\left(6\sin^{2}(\theta_{R})+(\cos(2\theta_{R})+3)\cos(2\phi_{R})\right)\sin(2\psi_{R})\Big),\\
C_{33}= & \cos(\theta_{R})\cos(2\phi_{R})\cos(2\psi_{R})-\frac{1}{4}(\cos(2\theta_{R})+3)\sin(2\phi_{R})\sin(2\psi_{R}),\\
C_{34}= & \sin(\theta_{R})(\cos(2\psi_{R})\sin(\phi_{R})+\cos(\theta_{R})\cos(\phi_{R})\sin(2\psi_{R})),\\
C_{35}= & \sin(\theta_{R})(\cos(\theta_{R})\sin(\phi_{R})\sin(2\psi_{R})-\cos(\phi_{R})\cos(2\psi_{R})),\\
C_{41}= & \frac{1}{2}\sin(\theta_{R})(\cos(\theta_{R})(\cos(2\phi_{R})+3)\cos(\psi_{R})-2\cos(\phi_{R})\sin(\phi_{R})\sin(\psi_{R})),\\
C_{42}= & \frac{1}{2}\sin(\theta_{R})(\sin(2\phi_{R})\sin(\psi_{R})-\cos(\theta_{R})(\cos(2\phi_{R})-3)\cos(\psi_{R})),\\
C_{43}= & \sin(\theta_{R})(\cos(\theta_{R})\cos(\psi_{R})\sin(2\phi_{R})+\cos(2\phi_{R})\sin(\psi_{R})),\\
C_{44}= & \cos(2\theta_{R})\cos(\phi_{R})\cos(\psi_{R})-\cos(\theta_{R})\sin(\phi_{R})\sin(\psi_{R}),\\
C_{45}= & \cos(2\theta_{R})\cos(\psi_{R})\sin(\phi_{R})+\cos(\theta_{R})\cos(\phi_{R})\sin(\psi_{R}),\\
C_{51}= & -\frac{1}{2}\sin(\theta_{R})(\cos(\psi_{R})\sin(2\phi_{R})+\cos(\theta_{R})(\cos(2\phi_{R})+3)\sin(\psi_{R})),\\
C_{52}= & \frac{1}{2}\sin(\theta_{R})(\cos(\psi_{R})\sin(2\phi_{R})+\cos(\theta_{R})(\cos(2\phi_{R})-3)\sin(\psi_{R})),\\
C_{53}= & \sin(\theta_{R})(\cos(2\phi_{R})\cos(\psi_{R})-2\cos(\theta_{R})\cos(\phi_{R})\sin(\phi_{R})\sin(\psi_{R})),\\
C_{54}= & -\cos(\theta_{R})\cos(\psi_{R})\sin(\phi_{R})-\cos(2\theta_{R})\cos(\phi_{R})\sin(\psi_{R}),\\
C_{55}= & \cos(\theta_{R})\cos(\phi_{R})\cos(\psi_{R})-\cos(2\theta_{R})\sin(\phi_{R})\sin(\psi_{R}).
\end{array}
\end{align*}

\subsection{Isomorphism between the representation of the group action of $SO(3)$ acting on $\R^5$ in \eqref{eqt: definition matrix M - real Wigner matrix} and the Cartan representation}\label{app: isomorphism between the group representations}
\begin{lem}[Isomorphism between the group representations]\label{lem:isomorphism between the group actions}
An isomorphism between the two representations of the group action is given by
\begin{align}\label{eqt: isomorphism between the representations}
\Phi=\left(
\begin{array}{ccccc}
 0 & 0 & -\frac{1}{\sqrt{3}} & 0 & 1 \\
 0 & 0 & -\frac{1}{\sqrt{3}} & 0 & -1 \\
 1 & 0 & 0 & 0 & 0 \\
 0 & 0 & 0 & 1 & 0 \\
 0 & 1 & 0 & 0 & 0 
\end{array}
\right)\end{align}
such that $M_C=\Phi M\Phi^{-1}$.
\end{lem}

\begin{proof}
Omitted.
\end{proof}

\section{Uniaxial solutions}\label{app: uniaxial solutions}

\begin{lem}\label{lem: independence of varphi p}
The expression
$$\int_{\theta_q=0}^\pi \int_{\varphi_q=0}^{2\pi}K(\theta_p,\varphi_p,\theta_q,\varphi_q)\rho(\theta_q)\sin \theta_q\;d\theta_q d\varphi_q$$
does not depend on $\varphi_p$.
\end{lem}

\begin{proof}
Let $R_z$ be a rotation around the $z$-axis. Using the rotational invariance of the interaction kernel $k(p\cdot q)$, we see that
\begin{align*}
\int_{\SP}k(R_zp\cdot q)\exp(-\phi(q))\;dq&=\int_{\SP}k(p\cdot R_z^tq)\exp(-\phi(q))\;dq\\&=\int_{\SP}k(p\cdot q)\exp(-\phi(R_zq))\;dq\\&=\int_{\SP}k(p\cdot q)\exp(-\phi(q))\;dq.
\end{align*}
Thus, the expression is invariant with respect to rotations around the $z$-axis and therefore independent of the variable $\varphi_p$.
\end{proof}




\end{document}